\documentclass[12pt]{amsart}
\usepackage{a4wide,enumerate,color,graphicx}
\usepackage{amsmath}
\usepackage{amssymb}
\usepackage{amsfonts}
\usepackage{amsthm}
\usepackage{enumerate}
\usepackage{mathrsfs}
\usepackage{mhchem}     
\newtheorem*{theoo}{Theorem}
\newtheorem{lemma}{Lemma}[section]
\newtheorem{prop}[lemma]{Proposition}
\newtheorem{theo}[lemma]{Theorem}
\newtheorem{defin}[lemma]{Definition}
\newtheorem{rem}[lemma]{Remark}
\newtheorem{coro}[lemma]{Corollary}

\newtheorem{conve}[lemma]{Convention}
\DeclareMathOperator{\divv }{div}

\allowdisplaybreaks

\begin{document}

\title[Multicomponent fluid dynamics]{Maximal mixed parabolic-hyperbolic regularity for the full equations of multicomponent fluid dynamics}

\subjclass[2010]{35M33, 76N10, 35Q35, 35Q79, 35D35, 80A17, 80A32}	
\keywords{Multicomponent fluids, continuum thermodynamics, PDE system of mixed type, local-in-time well-posedness, maximal regularity}

\author[P.-E.~Druet]{Pierre-Etienne Druet}
\address{Weierstrass Institute, Mohrenstr. 39, 10117 Berlin, Germany}
\email{pierre-etienne.druet@wias-berlin.de}

\date{\today}

\maketitle

\begin{abstract}
We consider a Navier--Stokes--Fick--Onsager--Fourier system of PDEs describing mass, energy and momentum balance in a Newtonian fluid with composite molecular structure. For the resulting parabolic-hyperbolic system, we introduce the notion of optimal regularity of mixed type, and we prove the short-time existence of strong solutions for a typical initial boundary-value-problem. By means of a partial maximum principle, we moreover show that such a solution cannot degenerate in finite time due to blow-up or vanishing of the temperature or the partial mass densities. This second result is however only valid under certain growth conditions on the phenomenological coefficients. In order to obtain some illustration of the theory, we set up a special constitutive model for volume-additive mixtures.
\end{abstract}

\section{Transport and mechanics for a composite Newtonian fluid}

We consider a molecular mixture of $N\geq 2$ chemical substances $\ce{A_1,\ldots,A_N}$ that constitute a fluid phase. While being composite at the molecular level, the fluid is macroscopically homogeneous and it obeys the continuum principles of mass, energy and momentum conservation
\begin{align}
\label{mass0}\partial_t \varrho + \divv( \varrho \, v ) & = 0 \, ,  \\
\label{energy0} \partial_t (\varrho u) + \divv (\varrho  u \, v+ J^{\rm h}) & = (-p \, \mathbb{I} + \mathbb{S}) \, : \, \nabla v \, , \\
\label{momentum0}\partial_t (\varrho \, v) + \divv( \varrho \, v\otimes v - \mathbb{S}) + \nabla p & = \varrho \, b  \, , 
\end{align}
with the mass density $\varrho$ of the fluid and the velocity field $v= (v_1, \, v_2, \, v_3)$. The absolute temperature $T$ -- which does not yet occur explicitely in the PDE system \eqref{mass0}, \eqref{energy0}, \eqref{momentum0} -- is usually the chosen as the additional main state variable.
In \eqref{energy0}, the function $\varrho u$ is the internal energy density, while $J^{\rm h}$ denotes the heat flux. The thermodynamic pressure is denoted by $p$, and $\mathbb{S}$ is the viscous stress tensor. The field $b$ in \eqref{momentum0} denotes the external forces. 

To solve these PDEs in some domain $\Omega \times ]0,\bar{\tau}[$ of $\mathbb{R}^4$, closure equations relating $\varrho u$, $p$ and $J^{\rm h}$, $\mathbb{S}$, $b$ to the main variables are further needed. For single component fluids, the thermodynamic state is expressed by the absolute temperature $T$ and the density $\varrho$, and \eqref{mass0}, \eqref{energy0}, \eqref{momentum0} leads to the Navier--Stokes--Fourier system. 
On the contrary, for multicomponent fluids, all thermodynamic driving mechanisms are essentially affected by the molecular composition. Not only the constitutive equations for $\varrho u$ and $p$, but also the material parameters occuring in the definition of the heat flux (thermodynamic diffusivities) and the stress tensor (viscosity coefficients) depend on which composition of the $N$ substances is locally available. 

The thermodynamic state of a mixture is thus locally expressed by $N+1$ variables\footnote{Depending on the context, other sets of variables might arise, which is commented below.}
\begin{align}\label{basicvariables}
 T \,\, \text{-- absolute temperature,} \quad \rho = (\rho_1, \ldots,\rho_N) \,\, \text{ -- partial mass densities.}
\end{align}
The partial mass density $\rho_i$ is the mass of the component ${\rm A}_i$ available per unit volume \emph{of the mixture}\footnote{The same function was also called the \emph{apparent} mass density of ${\rm A}_i$, see \cite{lowe}.}.
With this definition, the mass density of the fluid is nothing else but the sum of the partial mass densities: $\varrho := \sum_{i=1}^N \rho_i$ for which reason it shall be called the total mass density. The mass transport in a  multicomponent fluid is described, instead of having only one scalar conservation law \eqref{mass0} for the mass, by $N$ partial mass balances driving the densities $\rho_1, \ldots,\rho_N$. It might also be practically useful to specify different body forces $b^1,\ldots,b^N$ for the different species (This occurs for instance for charge carriers). In general, the mass and energy transport, and the mechanical behaviour of a multicomponent fluids in $\Omega \times ]0, \, \bar{\tau}[$, are described by the equations 
\begin{align}
\label{mass}\partial_t \rho_i + \divv( \rho_i \, v + J^i) & = r_i\, , \qquad  \text{ for } i = 1,\ldots,N\, ,\\
\label{energy} \partial_t (\varrho u) + \divv (\varrho u \, v+ J^{\rm h}) & = (-p \, \mathbb{I} + \mathbb{S}) \, : \, \nabla v + \sum_{i=1}^N J^i\cdot b^i \, , \\
\label{momentum}\partial_t (\varrho \, v) + \divv( \varrho \, v\otimes v - \mathbb{S}) + \nabla p & = \sum_{i=1}^N \rho_i \, b^i \, . 
\end{align}
These PDEs \eqref{mass}, \eqref{energy}, \eqref{momentum} form the basic equation of multicomponent fluid dynamics.
With $\jmath^i$ denoting the mass flux of substance ${\rm A}_i$, the \emph{mass diffusions fluxes} $J^1,\ldots,J^N$ are defined via $$J^i := \jmath^i -\rho_i \, v \, ,$$ as the non-convective part of the mass fluxes. To remain consistent with the continuity equation \eqref{mass0}, one has to postulate that
\begin{align}\label{netmassfluxequalsrhotimesv}
 \sum_{i=1}^N \jmath^i = \varrho \, v \quad \text{or, equivalently,} \quad \sum_{i=1}^N J^i = 0 \, .
\end{align}
This means that diffusive transport does not result, at the continuum level, into a net mass flux -- which would have to be associated with mechanical motion. 
Likewise, the production functions $r_1, \ldots, r_N$ modelling chemical reactions conserve the total mass, and therefore we must have $\sum_{i=1}^N r_i = 0$.


Depending on the choice of the thermodynamic potential, which will be discussed just hereafter, different characterisations of $\varrho u$ are possible. In any case, the internal energy density possesses an expression 
\begin{align}\label{varrhoubasic}\varrho u = \epsilon(T, \, \rho_1,\ldots,\rho_N) \, ,
\end{align}
 with a certain constitutive function $\epsilon$ of the main thermodynamic state-variables \eqref{basicvariables}.
%

The constitutive equations for the diffusion fluxes and the heat flux rely on the laws of Fick, Onsager and Fourier, more specifically
\begin{align}\label{DIFFUSFLUX}
J^i  = & - \sum_{j=1}^N M_{ij} \, \Big( \nabla\frac{\mu_j}{T} - \frac{b^j}{T}\Big) + l_i \, \nabla \frac{1}{T} \, , \qquad \sum_{i=1}^N M_{ij} = 0\,  \text{ for all } j \quad\textrm{and}\quad \sum_{i=1}^N l_{i} = 0 \, ,\\
\label{HEATFLUX} J^{\rm h} = & -\sum_{j=1}^N \, l_j \,  \Big(\nabla\frac{\mu_j}{T} - \frac{b^j}{T}\Big) +\kappa \, \nabla \frac{1}{T}~.
 \end{align}
In \eqref{DIFFUSFLUX}, $\{M_{ij}\}$ is a symmetric, positive semi-definite $N \times N$ matrix, and $\kappa$ in \eqref{HEATFLUX} is a positive scalar. The coefficients $l_j$ with $j=1,\ldots,N$ allow to describe thermo-diffusion. The thermodynamic diffusion coefficients $M_{ij}$, $l_j$ and $\kappa$ in general exhibit a strong dependence on the state variables $T$ and $\rho_1, \ldots,\rho_N$.

Hereby, as shown for instance in  \cite{bothedruetMS}, section D, consistency with the second law of thermodynamics requires that the $(N+1) \times (N+1)$ matrix 
\begin{align}\label{extendedM}
 \mathcal{M}(T, \, \rho_1, \ldots, \rho_N) =
 \begin{pmatrix}
  \{M(T, \, \rho_1, \ldots, \rho_N)\} & (l(T, \, \rho_1, \ldots, \rho_N))\\
  (l(T , \, \rho_1, \ldots, \rho_N) \, )^{\sf T} & \kappa(T, \, \rho_1, \ldots, \rho_N)
 \end{pmatrix}\, 
\end{align}
is positive semi-definite on all states $(T, \, \rho_1,\ldots,\rho_N) \in \mathbb{R}^{N+1}_+$.\footnote{We define $\mathbb{R}_+ := ]0, \, + \infty[$ and $\mathbb{R}^N_+ = (\mathbb{R}_+)^N$.} This can be achieved for instance requiring
\begin{align}\label{extendedcondi}
 l_i = - \sum_{j=1}^N M_{ij} \, \tilde{l}_j \quad  \text{ for } i=1,\ldots,N, \quad \kappa = \tilde{\kappa} + M \tilde{l} \cdot \tilde{l} \, , 
\end{align}
where $\tilde{\kappa} > 0$ and $\tilde{l}_1, \ldots,\tilde{l}_N$ are functions of the state variables.

If the matrix $M(T,\, \rho)$ which occurs in \eqref{DIFFUSFLUX} possesses rank $N-1$, then $\mathcal{M}(T, \, \rho)$ possesses rank $N$ and, due to the constraint in \eqref{DIFFUSFLUX}$_2$, the kernel of $\mathcal{M}$ is the span of the vector $(1,\ldots,1, 0)$ in $\mathbb{R}^{N+1}$. The entries in $\mathcal{M}$ are the so-called thermodynamic diffusivities. Details concerning their modelling by the Fick--Onsager or Maxwell--Stefan equations are to be found for instance in \cite{bothedreyer}, \cite{bothedruetMS} and further references given there.


In \eqref{DIFFUSFLUX} and \eqref{HEATFLUX} we have introduced the (mass-based) chemical potentials $\mu_1,\ldots,\mu_N$. In this paper, the constitutive theory relating the thermodynamic state variables follows the postulates of local thermodynamics and the theory of irreversible processes in the fashion of \cite{MR59}, \cite{dGM63}.\footnote{Personally I learned to know these notions with the more recent \cite{bothedreyer}, \cite{guhlkethesis} or \cite{dreyerguhlkemueller19}.} A very comprehensive overview of the models of ideal gas dynamics, and beyond, is offered by the book \cite{giovan}.

We start with the assumption that the entropy of the system possesses only a bulk contribution with density $\varrho s$, which is of the special form
\begin{align}\label{ENTROPE}
 \varrho s = - h(\rho_1, \ldots, \rho_N, \, \varrho u) \, .
\end{align}
The constitutive function $h: \, \mathcal{D} \subseteq \mathbb{R}^N_+ \times \mathbb{R} \rightarrow \mathbb{R}$ is assumed strictly convex and sufficiently smooth in its domain $\mathcal{D}$. It is to remark that, in local thermodynamics, the natural variables of the entropy functional are not $T$ and $\rho_1 ,\ldots,\rho_N$ but rather $\varrho u$ and $\rho_1,\ldots,\rho_N$ which are the quantities directly subject to conservation laws. However, attempts to construct the entropy functional from available data frequently use other variables like $T$ and $\rho_1, \ldots, \rho_N$ or, even better, $T$, the thermodynamic pressure $p$, and the mole fractions $x_1 ,\ldots,x_N$. We refer to \cite{bothedreyerdruet}, or to \cite{giovan}, Section 6 for details. In Section \ref{IDMIX}, we shall discuss an explicit example for the constitutive theory of volume-additive mixtures.

With the entropy functional at hand, we have the following well-known definitions relating the chemical potentials and the internal energy density to the state variables
\begin{align}
\label{TEMPCHEMPOT} -\frac{1}{T} & = \partial_{\varrho u} h( \rho_1, \ldots, \rho_N, \, \varrho u)  \quad \text{ and } \quad  \frac{\mu_i}{T} =  \partial_{\rho_i}h(\rho_1, \ldots, \rho_N, \, \varrho u) \, \, \text{ for } i = 1,\ldots,N \, .
\end{align}
We introduce the $(N+1)-$vector of extended state variables
\begin{align}\label{Defw}
w_i := \rho_i \quad \text{ for } i =1,\ldots,N \,, \quad \quad  w_{N+1} := \varrho u \,,
\end{align}
to denote the variables occurring in the entropy functional. The combinations
\begin{align}\label{Defwprime}
w^*_i := \frac{\mu_{i}}{T} \quad \text{ for } i = 1, \ldots, N \, , \quad \quad w^*_{N+1} := - \frac{1}{T}\, ,
\end{align}
are the dual variables, also called the ''entropic variables'' (See \cite{giovan}, Ch.\ 8.4 or \cite{juengel15}).

The thermodynamic pressure $p$ that occurs in \eqref{energy} and \eqref{momentum} obeys the Gibbs-Duhem equation
\begin{align}\label{GIBBSDUHEMEULER}
p = -\varrho u + T \, \varrho s  + \sum_{i=1}^N \rho_i \, \mu_i  \, .
\end{align}
With the variables in \eqref{Defw}, we also have
\begin{align}\label{PRESSUREDEF}
p = T \, \left(-\varrho u \, \frac{1}{T} + \varrho s + \sum_{i=1}^N \rho_i \, \frac{\mu_i}{T} \right)
= T \, \Big( - h(w) + \sum_{i=1}^{N+1} w_i \, \partial_{w_i}h(w)\Big) = T \, g(\nabla_w h(w)) \, ,
\end{align}
where $g$ is the classical Legendre transform of $h$.

In the Navier-Stokes equations \eqref{momentum}, the viscous stress tensor is denoted $\mathbb{S}$. We restrict to the choice
\begin{equation}\label{thermo5}
    \mathbb{S}  =\lambda(T,\rho_1, \ldots,\rho_N)\, (\divv  v)\, \mathbb{I}+ 2\,\eta(T,\rho_1, \ldots, \rho_N)\, (\nabla  v)_\textrm{sym} \quad\textrm{with}\quad
    \eta\geq 0,\quad \lambda+\frac{2}{3}\eta\geq 0~.
\end{equation}
In general, the viscosity coefficients are functions of the state variables $T$ and $\rho$ as well. The vector fields $b^1,\ldots,b^N$ occurring in \eqref{momentum} are external body forces, depending on space and time and assumed given.

To close the equations \eqref{mass}, \eqref{energy}, \eqref{momentum} it remains to explain the reaction terms $r_i$. We do not enter the details how to choose these functions, and we assume that $r_i = r_i(T, \, \rho_1,\ldots,\rho_N)$ are certain function of the state variables. As explained above, these $r_i$ are subject to $\sum_{i=1}^N r_i \equiv 0 $. In order to reach thermodynamic consistency with the second law without cross-effects between the different dissipative mechanisms, it is usual to require that $\sum_{i=1}^N r_i \, \mu_i/T \leq 0$. Examples and some discussions are to be found, a.o.\ in \cite{augnerbothe,herbergpruess,dredrugagu20}, and of course in many other references.
\\[0.5ex]


%

\paragraph{
\textbf{A boundary-value-problem.}} We shall investigate the system \eqref{mass}, \eqref{energy}, \eqref{momentum} in a cylinder $Q_{\bar{\tau}} := \Omega \times ]0, \, \bar{\tau}[$ with bounded cross-section $\Omega \subset \mathbb{R}^3$ and $\bar{\tau}>0$.

We impose initial conditions for the state variables and the velocity field via
\begin{align}\label{initial} \begin{split}
\rho_i(x,\, 0) & = \rho^0_i(x) \quad \text{ for } \quad i = 1,\ldots, N \, ,\\
 T(x, \, 0) & = T_0(x)\, ,\\
v_j(x, \, 0) & = v^0_j(x) \quad \text{ for } \quad j = 1,2,3\, ,
\end{split}
\qquad \text{ for } x \in \Omega \, .
\end{align}
Then, under the assumption \eqref{varrhoubasic}, we also obtain the initial state of the internal energy via
\begin{align}\label{initialeps}
\varrho_0 u_0(x) := \varrho u(x, \, 0) = \epsilon(T_0(x), \, \rho^0_1(x), \ldots,\rho^0_N(x)) \quad \text{ for }  x \in \Omega \, .
\end{align}
On the lateral surface $\partial \Omega \times ]0, \, \bar{\tau}[ = S_{\bar{\tau}}$, we for simplicity consider no-slip boundary conditions for the velocity field, and conditions of reaction-type for the transport of mass and energy. These read
\begin{align}\label{lateral} 
\begin{split}
v  = 0 \, , & \\
\nu \cdot J^i  = -r_i^{\Gamma} + J^{\Gamma}_i \text{ for } i = 1,\ldots,N &\, , \quad \nu \cdot J^{\rm h} = -r_{\rm h}^{\Gamma} + J^{\Gamma}_{{\rm h}} \, ,
\end{split} \quad \text{ on } S_{\bar{\tau}} \, ,
\end{align}
where $r^{\Gamma}_i$ and $r^{\Gamma}_{\rm h}$ are certain functions of the state variables $T$ and $\rho_1 ,\ldots, \rho_N$ and also of $(x,t)$. The contributions $J^{\Gamma}_i$ and $J^{\Gamma}_{{\rm h} }$ are given functions of space and time. The data $r^{\Gamma}_i$ and $J^{\Gamma}_i$ are moreover subject to the side conditions
\begin{align*}
\sum_{i=1}^N r^{\Gamma}_i = 0 \, ,\quad  \sum_{i=1}^N J^{\Gamma}_i = 0 \, .
\end{align*}
The conditions \eqref{lateral} for the mass fluxes describe exchange processes at the boundary: sorption phenomena or chemical reactions with adjacent matter. In \cite{dredrugagu20}, we for instance derived this type of boundary conditions from the more general principles of surface thermodynamics for electrolyte-electrode interactions exposed in \cite{guhlkethesis}, \cite{dreyerguhlkemueller19}. We also refer to \cite{augnerbothe,augnerbothe2} for interesting discussions in the context of catalysis, or to \cite{bothepruess}. For the heat flux, typical are for instance cooling conditions like $-r^{\Gamma}_{\rm h}(x,t, \, T) = \alpha \, (T-T^{\text{ext}}(x,t))$, where $T^{\text{ext}}$ is the given external temeprature and $\alpha$ a positive coefficient.\\[0.5ex]

\paragraph{\textbf{Thermodynamic state variables.} }

Beside the main variables $T$ and $\rho_1, \ldots,\rho_N$ introduced in \eqref{basicvariables}, we encounter in this investigation different sets of thermodynamic variables. At first, the natural variables of the entropy functional are
\begin{align}\label{variables1}
(\rho_1,\ldots,\rho_N, \,\varrho u) \, .
\end{align}
With the molar masses $M_1, \ldots, M_N > 0$ of the constituents, we introduce 
\begin{align*}
 n_i := \frac{\rho_i}{M_i} \, \, \text{ --mole density},& \quad n = \sum_{i= 1}^N n_i \, \, \text{--total mole density}, \, \quad
 x_i = \frac{n_i}{n} \, \, \text{--mole fraction.}
\end{align*}
A second equivalent set of variables is
\begin{align}\label{variables2}
 (T, \, p, \, x_1, \ldots,x_N) \, ,
\end{align}
which shall be used in connection with the discussion of ideal mixtures and the construction of the entropy functional.

For the PDE-analysis, the relevant set of variables relies partly on the dual variables and reads
\begin{align}\label{variables3}
\Big(\varrho, \, \frac{\mu_1-\mu_N}{T}, \ldots, \ \frac{\mu_{N-1}-\mu_N}{T}, \, -\frac{1}{T}\Big) \, . 
\end{align}
It can be shown that, in the stable fluid phase, the sets of state variables \eqref{basicvariables}, \eqref{variables1}, \eqref{variables2} and \eqref{variables3} are equivalent, meaning that there exist smooth bijections transforming these vectors into one another. These bijections shall be introduced below whenever needed. 

When we estimate a function of the state variables (free energy, pressure, internal energy, etc.) or want to calculate its derivatives using the chain rule, there is the typical problem how to indicate in which variables this function is currently expressed. The clearest way is certainly using superspripts. For instance, if $\varrho \psi$ is the Helmholtz free energy, we denote $\widetilde{\varrho \psi}$ its representation in $(T, \rho_1, \ldots,\rho_N)$, with $\widehat{\varrho \psi}$ its representation in $(T,p,x_1, \ldots,x_N)$, and so on. However, since we encounter here at least four different sets of variables, we often shall prefer to rely on the context than to introduce a completely consequent, but heavy system of notations.

\section{Multicomponent systems and mathematical analysis}

Although fluids with composite molecular structure are omnipresent in every-day life and applications, the basic PDEs describing their behaviour have been, comparatively to their relevance, rather less studied from the viewpoint of mathematical analysis. 
There might exist several reasons explaining this fact, among them certainly the two following ones. At first, the complexity of the underlying thermodynamic models makes investigations in this area long and involved. At second there are tough, yet unsolved mathematical problems already in the single-component case, like the global existence of regular solutions or handling the quadratic terms on the right-hand of the energy equation. Such difficulties are rather amplified by the multicomponent character.

Nevertheless, through the effort of several research groups, it has become clear since several years that multicomponent fluid dynamics raises, in PDE analysis too, interesting problems which are worth being investigated for themselves. The consistent coupling of chemistry and mechanics necessary to handle the basic problems of multicomponent gas dynamics leads, as exhibited in the book \cite{giovan}, to interesting PDE systems of mixed parabolic-hyperbolic type. We refer to the section 9, in particular 9.4, of \cite{giovan} for first results on existence of strong solutions in a neighourhood of the chemical equilibrium for this type of problems.

In several subsequent investigations, the equations \eqref{mass}, \eqref{energy}, \eqref{momentum} or the constitutive equations \eqref{DIFFUSFLUX}, \eqref{HEATFLUX} have been partly investigated.
The paper \cite{feipetri08} deals with the global weak solution analysis of multicomponent flow models, for the case of mixtures of mono-atomic gases. Here the diffusion fluxes are assumed of the Fickian form $J^i = - D \, \nabla y_i$ with a single common diffusivity $D > 0$ and $y_1, \ldots,y_N$ being the mass fractions. This constitutive law preserves the condition $\sum_{i=1}^N J^i = 0$, but the pressure and temperature contributions to the thermodynamic driving forces and the fluxes do not occur. The constitutive theory in \cite{feipetri08} yields chemical potentials of the form $\mu_i = \mu_0(T, \,\varrho) - T \, s_i$ where $s_i $ denote the partial specific entropy of the species, there assumed constant. Hence this form of the fluxes results from another modelling strategy than \eqref{DIFFUSFLUX}. In \cite{feipetri08}, the equation of state for the pressure depends only on $T$ and $\varrho$.
In this way, the mechanical behaviour is essentially independent on the molecular composition. (The coupling between partial mass balance equations and the mechanics is unilateral). This is quite different a situation than the one considered in the present investigation. 

In \cite{chenjuengel}, \cite{mariontemam} and \cite{bothepruess}, the total mass density $\varrho$ is assumed constant, so that \eqref{momentum} reduces to the usual incompressible Navier-Stokes equations. In this way the hyperbolic component and the bilateral coupling of chemistry and mechanics are avoided. This type of models can be well justified for dilute mixtures, see \cite{bothesoga}. Similar remarks apply to available investigations of the classical Nernst--Planck model of electrochemistry, or to Navier--Stokes--Nernst--Planck: see \cite{roubignoles}, \cite{bofisa16}, \cite{fisal} or \cite{schmuck}, \cite{constantin}. In \cite{mupoza15}, beside diverse simplifications, a ''parabolic regularisation'' (Bresch--Desjardins technique) is employed to obtain a control on the density gradient. In this case too, the hyperbolic character is relaxed.

In \cite{dredrugagu20} (electrolytes), \cite{druetmaxstef} (Maxwell--Stefan diffusion), we separate the hyperbolic from the parabolic components in the spirit of the book \cite{giovan} to obain well-posedness results in classes of weak solutions for isothermal systems. In \cite{bothedruet} we perform, still in the isothermal case, the local-in-time strong solution analysis for a general thermodynamic model. Other teams have investigated the models of ideal gas dynamics exposed in \cite{giovan} with different methods and focusses: Concerning strong solutions in \cite{piashiba19}, \cite{piashiba19pr}, concerning (stationary) weak solutions in \cite{bujupoza}, \cite{Axman}.

The purpose of the present paper is mainly to prove the local well-posedness of the full system \eqref{mass}, \eqref{momentum}, \eqref{energy} subject to a general thermodynamic model. This result is lacking in the literature.
We show that the approach of \cite{bothedruet} extends to non-isothermal systems driven by a general entropy functional of Legendre--type.  In comparison to the isothermal case, the temperature is an additional \emph{constrained parabolic variable}. We deal with this new situation by combining the setting of optimal mixed regularity for Navier--Stokes (see \cite{solocompress} for the inspiring precursor) applied in \cite{bothedruet}, and a partial maximum principle for strong solutions. The latter technique shows that a solution with mixed regularity cannot blow-up due to the temperature alone. 

In order to illustrate our concept, we moreover develop a thermodynamically consistent explicit constitutive model for volume-additive mixtures.\\[0.5ex]

Since multicomponent systems are not only encountered in the context of fluid dynamics, we conclude this section on the state of the art with a few additional remarks.

The system \eqref{mass} occurs in the context of reaction-diffusion systems, with or without energy equation.
These problems are essentially parabolic, but techniques like the entropic variables are obviously related to the ones in use in multicomponent fluid dynamics. We mention \cite{justel13,juengel15,juengelbook} or, for non-isothermal problems, \cite{piashiba20,buliro,helmer,eitrig} with recent interesting advances on thermodynamically-driven reaction-diffusion equations.

The paper \cite{dreyerguhlkemueller} raised a few years ago -- in the context of the Nernst--Planck model of electrochemistry -- the interesting question how pressure is handled in similar investigations. Very often indeed, the thermodynamic constitutive hypotheses combined with the Gibbs--Duhem equation do not yield isobaric systems. An exception is the paper \cite{herbergpruess}, where the free energy is modelled to result into an isobaric system. However, \cite{herbergpruess} assumes a constant total mass density $\varrho$. 

In the introduction to \cite{augnerbothe2}, we find interesting reflexions on how mechanics is handled in reaction-diffusion systems. 

As far as reaction-diffusion systems for fluids are concerned, it is sometimes claimed that other velocity fields are used (mainly the \emph{volume-averaged} velocity) than the one occurring in Newton's equations.  
Another interesting idea is using the Darcy velocity $v = -\nabla p$ in combination with the Gibbs--Duhem equation. The resulting mass fluxes $\jmath^i = -\rho_i \, \nabla p + J^i$ yield a full-rank reaction-diffusion parabolic system if $J^i$ obeys \eqref{DIFFUSFLUX}. The Darcy velocity was for instance used in \cite{druetjuengel}. However, the question why, for certain situations, the Darcy law or other velocities can be substituted to momentum balance, seems not to have been yet completely answered for multicompenent fluids.
Sometimes $v = 0$ has been postulated (See \cite{dreyerguhlkemueller}, \cite{fuhrmann15}, \cite{zubkokovtu}, \cite{helmer} or \cite{druetparaweak,druetparastrong}) as the fastest way to come to reaction--diffusion systems, but often the only solution consistent with momentum balance is then the hydrostatic and chemical equilibrium. In non-equilibrium, the Navier--Stokes equations cannot be simultaneously fulfilled.\\[0.5ex]

Mixture models also occur in the context of Cahn--Hilliard equations, that have been intensively studied from the viewpoint mathematical analysis. In the context of Cahn--Hilliard--Navier--Stokes models, a coupling to mechanics quite in the spirit of the system \eqref{mass}, \eqref{energy}, \eqref{momentum} has been set up, too. We refer to \cite{lowe} for the modelling of two-phase (binary) mixtures quite in the same spirit than the present investigation. 

Very often, however, significant differences occur. One important aspect are the variables. In many investigations on Cahn--Hilliard--Navier--Stokes equations, the volume-fractions are used to describe the molecular composition of the mixture: See for instance \cite{boyer,abelsgarckegruen} or, more recently, \cite{celorawagner}, section 3.
Then the models resemble the system \eqref{mass}, \eqref{energy}, \eqref{momentum} but the definition of the mechanical velocity is sometimes different, and the conservation of mass is handled differently in both approaches. 

Interesting questions arise for instance from the choice of diffusion constants: Cahn--Hilliard models often set the mobilities (corresponding to the matrix $M$) \emph{constant}, while the Maxwell--Stefan approach sets them proportional by a constant factor -- the Fickian diffusivities -- to the mass densities, see \cite{bothedruetMS}.

Despite a profound relashionship between the models, from the viewpoint of mathematical analysis, the techniques to handle reaction-diffusion systems, Cahn--Hilliard systems and the model studied in this paper are in fact quite different. We therefore do not further discuss this point here, referring to the Section \ref{discussion} for some additional reflexions on the modelling topic.

\section{Main results}

We will first formulate our main result in terms of the main state-variables $(T, \, \rho, \, v)$ which is a vector of size $1+N+3$. From purely structural PDE considerations, these variables are  not the most natural ones, though, since parabolic and hyperbolic components are not clearly separated.
The concept of optimal mixed regularity that we shall introduce afterwards relies on so-called entropic variables\footnote{From the point of view of the mathematical structure of multicomponent fluid dynamics the relevant distinction is the one between the total mass density, driven by hyperbolic dynamics, and other suitable components with parabolic dynamics.}, see \eqref{Defwprime}. 

Let $\Omega \subset \mathbb{R}^3$ be a bounded domain, and $\bar{\tau}$ a positive number. We denote by $Q_{\bar{\tau}}$ the cylinder $\Omega \times ]0,\bar{\tau}[$ in $\mathbb{R}^4$.
To formulate our results, we need several usual function spaces. We refer to standard monographs, like for instance the second Chapter of \cite{ladu} or, for up-to-date language and the most recent results, to \cite{denkhieberpruess} with references. We introduce the parabolic Sobolev spaces
\begin{align*}
 W^{2,1}_p(Q_{\bar{\tau}}) := & \{u \in L^p(Q_{\bar{\tau}}) \, : \, \partial_{t}^{\beta} \, \partial_x^{\alpha} u \in L^p(Q_{\bar{\tau}}) \text{ for all } 0 < 2\beta+|\alpha| \leq 2 \} \, ,\\
 W^{1,0}_p(Q_{\bar{\tau}}) := & \{u \in L^p(Q_{\bar{\tau}}) \, : \,   \partial_x^{\alpha} u \in L^p(Q_{\bar{\tau}}) \text{ for all } |\alpha|  = 1 \} \, ,
\end{align*}
and the anisotropic Sobolev spaces over the cylinder domain
\begin{align*}
 W^{1,1}_{p,q}(Q_{\bar{\tau}}) := \{u \in L^{p,q}(Q_{\bar{\tau}}) \, : \, D^{\alpha} u \in L^{p,q}(Q_{\bar{\tau}}) \text{ for all } |\alpha| = 1\} \, .
\end{align*}
Here, $L^{p,q}(Q_{\bar{\tau}})$ is the Banach--space of all Lebesgue--measurable classes over $Q_{\bar{\tau}}$ with finite integral $\int_{0}^{\bar{\tau}} ( \int_{\Omega} |u(x,\tau)|^p \, dx)^{q/p} \, d\tau$ or, for $q = +\infty$, with finite $\underset{\tau \in [0, \, \bar{\tau}]}{\text{ess sup}} \int_{\Omega} |u(x,\tau)|^p \, dx$.

We denote by $W^{1,1}_{p,\infty}(Q_{\bar{\tau}})$ the closure of $C^1(\overline{\Omega} \times [0, \, \bar{\tau}])$ with respect to the $L^{p,\infty}-$norm. Thus, elements of this space possess continuous-in-time weak derivatives with values in $L^p(\Omega)$. We moreover need the trace spaces
\begin{align*}
 & W^{2-\frac{2}{p}}_p(\Omega) := \Big\{u \in W^{1,p}(\Omega) \, : \, [\partial_{x_i}u]^{(1-\frac{2}{p})}_{p,\Omega} < + \infty \text{ for } i = 1,\ldots,N \Big\} \, \\
\text{ with }  \quad &  [u]^{(\lambda)}_{p,\Omega} = \left( \int_{\Omega} \int_{\Omega} \frac{|u(x) - u(y)|^p}{|x-y|^{3+\lambda \, p} } \, dxdy\right)^{\frac{1}{p}}\, ,\\
& W^{1-\frac{1}{p}, \frac{1}{2}-\frac{1}{2p}}_p(S_{\bar{\tau}}) := \Big\{u \in L^p(S_{\bar{\tau}}) \, : \, [u]_{p,x,S_{\bar{\tau}}}^{(1-\frac{1}{p})} + [u]_{p,t,S_{\bar{\tau}}}^{(\frac{1}{2}-\frac{1}{2p})} < +\infty \Big\}\\
\text{ with } \quad & [u]_{p,x,S_{\bar{\tau}}}^{(\lambda)} := \left(\int_0^{\bar{\tau}} \int_{\partial\Omega}\int_{\partial \Omega} \frac{|u(x,t) - u(y,t)|^p}{|x-y|^{2 + p\lambda}} \, dydxdt\right)^{\frac{1}{p}} \, ,\\
& [u]_{p,t,S_{\bar{\tau}}}^{(\lambda)} := \left(\int_{\partial\Omega} \int_{0}^{\bar{\tau}}\int_{0}^{\bar{\tau}} \frac{|u(x,t) - u(x,s)|^p}{|t-s|^{1 + p\lambda}} \, dtdsdx\right)^{\frac{1}{p}} \, .
 \end{align*}
 For a function $r^{\Gamma} = r^{\Gamma}(x, \, t, \, T, \,\rho_1, \ldots,\rho_N)$, defined on $S_{\bar{\tau}} \times \mathbb{R}^{N+1}_+$ we write 
 \begin{align*}
 r^{\Gamma} \in C^{\lambda,\frac{\lambda}{2},1}(S_{\bar{\tau}} \times \mathbb{R}^{N+1}_+) \, ,
 \end{align*}
 iff $r^{\Gamma}$ is continuous in $(x,t)$ for all $(T, \, \rho)$, continuously differentiable in $(T, \, \rho)$ for all $(x,t)$ and, for all $K \subset \mathbb{R}^{N+1}_+$ compact, we have
  \begin{align*}
 & [r^{\Gamma}]_{x}^{(\lambda)} := \sup_{x \neq y \in \partial \Omega, \, (t, \, T, \, \rho) \in ]0,\bar{\tau}[ \times K\, } \frac{|r^{\Gamma}(x, \, t, \, T, \, \rho) -  r^{\Gamma}(y, \, t, \, T, \, \rho)}{|x-y|^{\lambda}} < +\infty\\
 & [r^{\Gamma}]^{(\frac{\lambda}{2})}_{t} := \sup_{t \neq s \in ]0, \, \bar{\tau}[, \, (x, \, T, \, \rho) \in \partial \Omega \times K} \frac{|r^{\Gamma}(x, \, t, \, T, \, \rho) -  r^{\Gamma}(x, \, s, \, T, \, \rho)}{|t-s|^{\frac{\lambda}{2}}} < + \infty \, ,\\
& \|D r^{\Gamma}\|_{C(S_{\bar{\tau}} \times K)} :=  \sup_{(x,t) \in S_{\bar{\tau}}, \, (T,\rho) \in K} |D_{T,\rho}^1 r^{\Gamma}(x,t, \, T, \, \rho)| < + \infty  \, .
\end{align*}
To formulate the necessary compatibility conditions on the line $\partial\Omega \times \{0\}$, we also need the Slobodecki spaces $W^s_p(\partial \Omega)$. The definitions are well-known too.
 
We also define $\mathcal{P}: \, \mathbb{R}^{N+1} \rightarrow \{(1^N, \, 0)\}^{\perp}$ to be the orthogonal projection onto the orthogonal complement of the span of $(1^N, \, 0) = (1,\ldots,1, 0)$ in $\mathbb{R}^{N+1}$. We choose 
\begin{align}\label{Dgrund}
 \mathcal{D} := \{(\rho, \, \varrho u) \in \mathbb{R}^N_+ \times \mathbb{R} \, : \, \varrho u > \epsilon^{\min}(\rho)\} \, ,
\end{align}
where we assume that $\epsilon^{\min} \in C( \mathbb{R}^N_+)$ is a given convex function, or $\epsilon^{\min} \equiv -\infty$.
We next formulate our main theorem.
\begin{theo}\label{MAIN}
We fix $p > 3$, and we assume that \begin{enumerate}[(a)]
\item $\Omega \subset \mathbb{R}^3$ is a bounded domain of class $\mathcal{C}^2$;
\item $M: \, \mathbb{R}^{N+1}_+ \rightarrow \mathbb{R}^{N\times N}$ is a mapping of class $C^2(\mathbb{R}^{N+1}_+; \, \mathbb{R}^{N\times N})$ into the positive semi-definite matrices of rank $N-1$ with constant kernel $1^N = \{1, \ldots, 1\}$; $l \in C^2(\mathbb{R}^{N+1}_+; \, \mathbb{R}^{N})$ maps into the orthognal complement of $1^N$; $\kappa \in C^2(\mathbb{R}^{N+1}_+)$ is a strictly positive function. These functions are chosen such that the matrix $\tilde{M}(T, \, \rho)$ of \eqref{extendedM} is positive semi-definite at all $(T, \, \rho) \in \mathbb{R}^{N+1}_+$, for instance by means of the construction\eqref{extendedcondi};
\item $\eta, \, \lambda \in C^2(\mathbb{R}^{N+1}_+)$ are subject to the restrictions in \eqref{thermo5}; 
\item $\mathcal{D}$ is an open convex set of the epigraphial form \eqref{Dgrund}; $h: \, \mathcal{D} \subset \mathbb{R}^N_+ \times \mathbb{R} \rightarrow \mathbb{R}$ is of class $C^3(\mathcal{D})$ and, together with its convex conjugate $h^*$ on the domain $\mathbb{R}^N \times \mathbb{R}_- =: \mathcal{D}^*$, it forms a pair of Legendre--type;
\item $r \in C^1(\mathbb{R}^{N+1}_+; \, \mathbb{R}^N)$ maps into the orthogonal complement of $1^N$;
\item $r^{\Gamma}_1, \ldots,r^{\Gamma}_N, \, r^{\Gamma}_{\rm h} \in C^{\lambda, \,  \frac{\lambda}{2},\, 1}(S_{\bar{\tau}} \times \mathbb{R}^{N+1}_+)$ with $\lambda > 1-1/p$, and $r^{\Gamma} = (r^{\Gamma}_1, \ldots,r^{\Gamma}_N)$ maps into the orthogonal complement of $1^N$;
\item \label{force} The external forcing $b^1,\ldots,b^N$ satisfies 
\begin{gather*}
b - 1^N \cdot b/N \, 1^N \in W^{1,0}_p(Q_{\bar{\tau}}; \, \mathbb{R}^{N\times 3}) \cap C([0,T]; \, W^{1-\frac{2}{p}}_p(\Omega; \, \mathbb{R}^{N\times 3})) \, ,\\
(b - 1^N \cdot b/N \, 1^N)  \cdot \nu \in W^{1-\frac{1}{p},\frac{1}{2}-\frac{1}{2p}}_p(S_{\bar{\tau}};\, \mathbb{R}^N) \, ,
\end{gather*}
and $1^N \cdot b \in L^p(Q_{\bar{\tau}}; \, \mathbb{R}^{3})$. 
\item The functions $J^{\Gamma}_1, \ldots, J^{\Gamma}_N, \, J^{\Gamma}_{\rm h}$ all belong to $W^{1-\frac{1}{p},\frac{1}{2}-\frac{1}{2p}}_p(S_{\bar{\tau}})$ and $\sum_{i=1}^N J^{\Gamma}_i = 0$;
\item The initial data $T_0, \, \rho^0_{1}, \ldots \rho^0_{N}: \, \Omega \rightarrow \mathbb{R}_+$ are strictly positive measurable functions satisfying the following conditions:
\begin{itemize}
\item The initial total mass density $\varrho_0 := \sum_{i=1}^N \rho_{i}^0$ is of class $W^{1,p}(\Omega)$, and there is $m_0 > 0$ such that $ 0 < m_0 \leq \varrho_0(x)$ for all $x \in \Omega$;
\item The initial reciprocal temperature $-1/T_0$ belongs to $W^{2-2/p}_p(\Omega)$, and there is $\theta_1 > 0$ such that $0 < T_0(x) \leq \theta_1$ for all $x \in \Omega$;
 \item Defining $\varrho_0 u_0$ via \eqref{initialeps}, the vector field $w^{0*} := \nabla_{w}h(\rho^0_{1}, \ldots \rho^0_{N}, \, \varrho_0u_0)$ (initial dual variables) satisfies $\mathcal{P} \,w^{0*} \in W^{2-\frac{2}{p}}_p(\Omega; \, \mathbb{R}^{N+1})$;
\item The initial and boundary data satisfy, as traces in $W^{1-\frac{3}{p}}_p(\partial \Omega)$, the compatibility conditions
\begin{align*}
 & - \nu \cdot \Big(\sum_{j=1}^N M_{ij}(T_0,\rho^0) (\nabla w^{0*}_j -b^j(0)/T_0) + l_i(T_0,\rho^0) \, \nabla w^{0*}_{N+1}\Big) \\
 & \quad =- r^{\Gamma}_i(x, \, 0, \, T_0, \, \rho^0) + J^{\Gamma}_i(x,0) \, , \quad \text{ for } i = 1,\ldots,N \, ,\\
& -\nu \cdot \Big(\sum_{j=1}^N \, l_j(T_0,\rho^0) \,   (\nabla w^{0*}_j- b^j(0)/T_0) + \kappa(T_0,\rho^0) \, \nabla w^{0*}_{N+1}\Big) \\
& \quad
 = -r^{\Gamma}_{\rm h}(x, \, 0, T_0,\rho^0) + J^{\Gamma}_{\rm h}(x, \, 0) \, .
\end{align*}
%
%
%
%
%
\end{itemize}
\item The initial velocity $v^0$ belongs to $W^{2-\frac{2}{p}}_{p}(\Omega; \, \mathbb{R}^3)$ with $v^0 = 0$ in $W^{2-\frac{3}{p}}_{p}(\partial \Omega; \, \mathbb{R}^3)$.
\end{enumerate}
Then, there exists $0 < t^* \leq \bar{\tau}$ such that the problem \eqref{mass}, \eqref{energy}, \eqref{momentum} with closure relations \eqref{DIFFUSFLUX}, \eqref{HEATFLUX}, \eqref{GIBBSDUHEMEULER} and boundary conditions \eqref{initial}, \eqref{lateral} ($=: (P)$) possesses a unique solution in the class
\begin{align*}
T \in W^{2,1}_p(Q_{t^*}; \, \mathbb{R}_+), \quad  \rho \in W^{1}_{p}(Q_{t^*}; \, \mathbb{R}^N_+), \quad v \in W^{2,1}_p(Q_{t^*}; \, \mathbb{R}^3) \, .
\end{align*}
Moreover, the vector $w^* := \nabla_{w}h(\rho, \, \epsilon(T, \, \rho))$ satisfies $ \mathcal{P} \,w^* \in  W^{2,1}_p(Q_{t^*}; \, \mathbb{R}^{N+1})$. 
\end{theo}
From the purely mathematical viewpoint, there seem to be a more natural formulation of this result. Indeed, the variables $(T, \, \rho, \, v)$ do not account for the mixed parabolic-hyperbolic structure of the PDE system. In order to formulate results in classes of optimal regularity, we use other variables:
\begin{alignat}{2}
\label{parabvar} & \left. \begin{matrix}
\frac{\mu_1 - \mu_N}{T}, \ldots,  \frac{\mu_{N-1} - \mu_N}{T} & \text{ relative chemical potentials}\\[0.5ex]
 -\frac{1}{T}  & \text{ reciprocal of temperature}\\[0.5ex]
 v_1, \, v_2, \, v_3  & \text{ velocities}
 \end{matrix} \right\} \quad & & \text{ the parabolic variables,}\\[0.1ex]
\label{hyperbvar} & \varrho \, \, \text{ total mass density}  & & \text{the hyperbolic variable.}
\end{alignat}
The origin of a distinction of parabolic and hyperbolic variables is retraced in \cite{giovan}, Chapter 8. The name ''relative chemical potentials'' was used in \cite{dredrugagu20}.\footnote{In other contexts, one sometimes find the definition of chemical potentials as $\mu_i := \partial_{y_i} \psi$ with $y_1, \ldots,y_{N-1}$ being the  $N-1$ \emph{independent} mass fraction, and $\psi$ the specific Helmholtz free energy as a function of $T$, $\varrho$ and $y_1, \ldots,y_N$. Then, by definition, there are only $N-1$ chemical potentials, the ''relative'' ones.}
\begin{defin}
A solution for the PDE system \eqref{mass}, \eqref{energy}, \eqref{momentum} is said to possess optimal mixed regularity (of index $p$, on the interval $]0, \, \bar{\tau}[$) if all associated parabolic variables in \eqref{parabvar} belong to $W^{2,1}_p(Q_{\bar{\tau}})$ and the hyperbolic component $\varrho$ belongs to $W^{1,1}_{p,\infty}(Q_{\bar{\tau}})$.
\end{defin}
The statement of Theorem \ref{MAIN} can be equivalently reformulated in more simple fashion.
\begin{theoo}
 Under the assumptions of Theorem \ref{MAIN}, there exists $0 < t^* \leq \bar{\tau}$ such that the problem $(P)$ possesses a unique solution of optimal mixed regularity with index $p$ on $]0,t^*[$. 
\end{theoo}
We shall also raise an additional question concerning blow-up criteria and the maximal existence-time for the strong solution. Let $(T, \, \rho,  \, v) $ be a solution of optimal mixed regularity for $(P)$ on $]0, \, \bar{\tau}[$, and assume that $\sup_{(x,\tau) \in Q_t} T(x, \, \tau)$ is bounded for all $t < \bar{\tau}$.
Then, the parabolic regularity of the reciprocal temperature $- 1/T$ \emph{alone} does not prevent possible blow-up of the temperature for $t \rightarrow \bar{\tau}-$. In other words, $\|T\|_{L^{\infty}(Q_t)} \rightarrow + \infty$ implies certainly that $\sup_{Q_t} (-1/T)\rightarrow 0-$, but this is no contradiction to $-1/T$ remaining finite and even smooth.\\[0.2ex]

Hence, in the context of solutions with optimal mixed regularity of index $p$, a maximum principle for the temperature has to be proved independently. We shall show how to achieve such a result by imposing certain growth restrictions on the data in the thermodynamic model. This can be done even in the general case. However it looks in our eyes more meaningful to formulate the result for a particular example. In this way, we exhibit the general procedure well enough, while being able to calculate explicit growth conditions. We refer to the section \ref{IDMIX} for the construction of a particular constitutive model for volume-additive mixtures, and to Theorem \ref{MAIN2} for the formulation of the maximum principle in this case.


\section{Transformation to a parabolic-hyperbolic system}

The system \eqref{mass}, \eqref{momentum} for constant temperature has been studied in \cite{bothedruet}. Here we show that the extension to energy systems can be dealt with by means of essentially the same method.

\subsection{Change of variables}\label{changevariables}

We recall the definitions \eqref{Defw}, \eqref{Defwprime} of primal variables $w = (\rho, \, \varrho u) \in \mathcal{D}$ and dual variables $w^* := (\mu/T, \, -1/T) \in \mathbb{R}^N \times \mathbb{R}_-$. With the help of the conjugate convex function $h^*$ to $h$, we can invert the relations 
\begin{align*}
\frac{\mu_i}{T} = \partial_{\rho_i}h(\rho, \, \varrho u) \quad \text{ for } i=1,\ldots,N \, , \quad \quad 
-\frac{1}{T} = \partial_{\varrho u} h(\rho, \, \varrho u) \, ,
\end{align*}
which more compactly now read as $ w^* = \nabla_w h (w)$.
\begin{lemma}\label{Legendre}
We assume that $h$ is a function of {\em Legendre--type} on $\mathcal{D} \subseteq \mathbb{R}^N_+ \times \mathbb{R}$ open, convex. We define $\mathcal{D}^* = \mathbb{R}^N \times \mathbb{R}_-$ and assume that the image of $\nabla_w h$ on $\mathcal{D}$ is equal to $\mathcal{D}^*$. Then, the convex conjugate $h^*(w^*) := \sup_{w \in \mathcal{D}} \{w^* \cdot w - h(w)\}$ is a function of Legendre type on $\mathcal{D}^*$. The gradients $\nabla_w h$ on $\mathcal{D}$ and $\nabla_{w^*} h^*$ on $\mathcal{D}^*$ are inverse to each other. Moreover $h \in C^k(\mathcal{D})$ iff $h^* \in C^k(\mathcal{D}^*)$ for all $k \geq 1$.
\end{lemma}
\begin{proof}
Due to the Theorem 26.5 of \cite{rockafellar}, we directly obtain that $(h, \, \mathcal{D})$ and $(h^*, \, \mathcal{D}^*)$ is a Legendre pair. This means that $\mathcal{D}$, $\mathcal{D}^*$ are open and convex, and they coincide  with the interior of the domain of $h$, $h^*$. Moreover, $h$ and $h^*$ are strictly convex and essentially smooth therein and $(\nabla_w h_{|_{\mathcal{D}}})^{-1} = \nabla_{w^*} h^*_{|_{\mathcal{D}^*}}$. Since $\nabla_w h$ and $\nabla_{w^*} h^*$ are inverse to each other, we get the $C^k$ property due to the implicit function theorem.
\end{proof}

The next topic is passing to variables that allow to exhibit the parabolic-hyperbolic structure of the PDE system. To this aim, we choose new axes $\xi^1,\ldots,\xi^{N}, \, \xi^{N+1}$ of $\mathbb{R}^{N+1}$ such that 
\begin{align}\label{labase}
\left\{
\begin{matrix} 
 \xi^N :=  e^{N+1} = (0, \ldots, \, 0, \, 1) \quad \text{ and } \quad \xi^{N+1} := (1, \ldots, \, 1, \, 0) \, ,  \\[0.2cm]
 \xi^1, \ldots, \xi^{N-1} \in \Big\{x \in \mathbb{R}^{N+1} \, : \, x_{N+1} = 0, \,  \sum_{i=1}^N x_i = 0\Big\} =: \{\xi^{N}, \,  \xi^{N+1}\}^{\perp} \, .
\end{matrix} \right.
 \end{align}
We let $\eta^1,\ldots,\eta^{N+1} \in \mathbb{R}^{N+1}$ be the dual basis to $\xi^1,\ldots,\xi^{N+1}$. We prove easily that
\begin{align}\label{etaspecia}
 \eta^N = e^{N+1}, \quad \eta^{N+1} = \frac{1}{N} \, \xi^{N+1} \, .
\end{align}
For $w^* = (\mu_1/T, \ldots,\mu_N/T, \, -1/T)$, we define the projections
\begin{align}\label{relativepot}
q_{\ell} := \eta^{\ell} \cdot w^* := \sum_{i=1}^{N+1} \eta^{\ell}_i \, w^*_i\quad  \text{ for } \quad  \ell = 1,\ldots,N \, .
\end{align}
Due to the properties of the chosen basis, in particular to \eqref{etaspecia}, we have a relationship 
\begin{align}\label{TRAFO1}
& w^* = \sum_{\ell = 1}^N q_{\ell} \, \xi^{\ell} + (w^* \cdot \eta^{N+1}) \, \xi^{N+1}  \quad 
\text{ implying that } \quad  w_{N+1}^* = q_{N} \, .
\end{align}
Now, since the coordinate $w_{N+1}^*$ has the physical meaning of $-1/T$, the relevant domain for the new variable $q$ is the half-space
\begin{align*}
\mathcal{H}^N_- := \Big\{ (q_1, \ldots, q_N) \in \mathbb{R}^N \, :\, q_{N} < 0 \Big\} = \mathbb{R}^{N-1} \times \mathbb{R}_{-}  \, . 
\end{align*}
Since $( \mu_1/T, \ldots, \mu_N/T, \ -1/T) = \nabla_w h(\rho_1, \, \ldots, \rho_N, \varrho u)$, use of the conjugate convex function implies that
\begin{align*}
(\rho_1, \, \ldots, \rho_N, \varrho u) = \nabla_{w^*} h^*(\mu_1/T, \ldots, \mu_N/T, \ -1/T) \, .
\end{align*}
In order to isolate the hyperbolic component $\varrho$ (total mass density), we now express
\begin{align*}
\varrho = \sum_{i=1}^N \rho_i &  = \sum_{i=1}^{N+1} \xi^{N+1}_i \, \partial_{w^*_i} h^*(w^*_1,\ldots,w^*_{N+1}) \\
& = \xi^{N+1} \cdot \nabla_{w^*}h^*\Big(\sum_{\ell = 1}^{N} q_{\ell} \, \xi^{\ell} + (w^* \cdot \eta^{N+1}) \, \xi^{N+1}\Big) \, .
\end{align*}
This is an algebraic equation of the form $F(w^* \cdot \eta^{N+1}, \, q_1, \ldots, q_{N}, \, \varrho) = 0$. We notice that
\begin{align*}
\partial_{w^* \cdot \eta^{N+1}} F(w^* \cdot \eta^{N+1}, \, q_1, \ldots, q_{N}, \, \varrho) = D^{2}h^*(w^*) \xi^{N+1} \cdot \xi^{N+1} > 0 \, ,
\end{align*}
due to the strict convexity of the conjugate function (Lemma \ref{Legendre}). Thus, the latter algebraic equation defines the last component $w^* \cdot \eta^{N+1}$ implicitly as a differentiable function of $\varrho$ and $q_1, \ldots, q_{N}$. We call this function $\mathscr{M}$ and obtain the equivalent formulae
\begin{align}\label{MUAVERAGE}
w^* & = \sum_{\ell=1}^{N} q_{\ell} \, \xi^{\ell} + \mathscr{M}(\varrho, \, q_1,\ldots,q_{N}) \, \xi^{N+1} \, ,\\
\label{RHONEW}\rho_i & = \partial_{w^*_i}h^*\Big( \sum_{\ell=1}^{N} q_{\ell} \, \xi^{\ell} + \mathscr{M}(\varrho, \, q_1,\ldots,q_{N}) \, \xi^{N+1}\Big) =: \mathscr{R}_i(\varrho, \, q) \, , \quad \text{ for } i = 1,\ldots,N \, ,\\
\label{RHOUNEW} \varrho u & = \partial_{w^*_{N+1}} h^*\Big( \sum_{\ell=1}^{N} q_{\ell} \, \xi^{\ell} + \mathscr{M}(\varrho, \, q_1,\ldots,q_{N}) \, \xi^{N+1}\Big) \, ,
 \end{align}
with $\varrho$ and $q_1,\ldots,q_{N}$ as the free variables. Since $p$ obeys \eqref{GIBBSDUHEMEULER}, we have $p = T \, g(\nabla_w h(w))$,
and here $g$ denotes the Legendre transform of $h$. Using Lemma \ref{Legendre}, we find $ p = T \, h^*(w^*)  = - h^*(w^*)/w_{N+1}^* $.
%
%
%
%
%
We combine the latter with \eqref{MUAVERAGE} to obtain that
\begin{align}\label{PNEW}
p = P(\varrho, \, q) := - \frac{1}{q_N} \,   h^*\Big( \sum_{\ell=1}^{N} q_{\ell} \, \xi^{\ell} + \mathscr{M}(\varrho, \, q_1,\ldots,q_{N}) \, \xi^{N+1}\Big) \, . 
\end{align}
With $\bar{q} = (q_1,\ldots,q_{N-1})$, this expression makes sense in the domain
\begin{align*}
 (\varrho, \, \bar{q}, \, q_N) \in \mathbb{R}_+ \times \mathcal{H}^N_- = \mathbb{R}_+ \times \mathbb{R}^{N-1} \times \mathbb{R}_{-} \, .
\end{align*}
To analyse the PDEs \eqref{mass}, \eqref{energy}, \eqref{momentum}, we next need some regularity properties of the transformed coefficients. 
\begin{lemma}\label{pressurelemma}
 Suppose that $\mathcal{D}$ is of the form \eqref{Dgrund}, and $h \in C^{3}(\mathcal{D})$ satisfies the assumptions of Lemma \ref{Legendre}.
 Then, the formula \eqref{PNEW} defines a function $P$ which belongs to $C^2(\mathbb{R}_{+} \times \mathcal{H}^{N}_{-})$.
\end{lemma}
\begin{proof}
We first investigate the regularity of the function $\mathscr{M}$ introduced in \eqref{MUAVERAGE}. 

For fixed $q \in \mathcal{H}^N_-$ and $\varrho > 0$, we define
\begin{align*}
 f(\mathscr{M}) := h^*\Big(\sum_{\ell = 1}^{N} q_{\ell} \, \xi^{\ell} + \mathscr{M} \, \xi^{N+1}\Big) - \varrho \, \mathscr{M} \quad \text{ for } \mathscr{M} \in \mathbb{R} \, .
\end{align*}
This $f$ is of class $C^2(\mathbb{R})$ and strictly convex. By the definition of $h^*$, we have, for all points $w = (t \, 1^N/N, \epsilon) = t \, \eta^{N+1} + \epsilon \, \eta^N \in \mathcal{D}$ with $t > 0$ and $\epsilon > \epsilon^{\min}(t \, 1^N/N)$  
\begin{align*}
f(\mathscr{M}) \geq & \Big(\sum_{\ell = 1}^{N} q_{\ell} \, \xi^{\ell} + \mathscr{M} \, \xi^{N+1}\Big) \cdot (t \, \eta^{N+1} + \epsilon \, \eta^N) - h(t \, 1^N/N, \, \epsilon) - \varrho \, \mathscr{M}\\
= & q_N \, \epsilon - h(t \, 1^N/N, \, \epsilon)  + (t- \varrho) \, \mathscr{M}  \, . 
\end{align*}
%
Hence, choosing fixed $t = \varrho/2$ and $t = 2\,\varrho$, we easily show that
$ \lim_{|\mathscr{M}| \rightarrow \infty} f(\mathscr{M}) = +\infty$.
Thus, $f$ possesses a unique global minimiser, denoted by $\mathscr{M}(\varrho, \, q) \in \mathbb{R}$. Since $f^{\prime}(\mathscr{M}(\varrho, \, q)) = 0$, we thus see that $$\xi^{N+1} \cdot \nabla_{w^*}h^*\Big(\sum_{\ell = 1}^{N} q_{\ell} \, \xi^{\ell} + \mathscr{M} \, \xi^{N+1}\Big) = \varrho  \, ,$$ and $\mathscr{M}$ is well-defined. For the derivatives of $\mathscr{M}$, we obtain the expressions
\begin{align}\label{Mnachrho}
\partial_{\varrho} \mathscr{M}(\varrho, \, q) = \frac{1}{D^2h^*\xi^{N+1} \cdot \xi^{N+1}}, \quad \partial_{q_k}\mathscr{M}(\varrho, \, q) = -\frac{D^2h^*\xi^{N+1} \cdot \xi^k}{D^2h^*\xi^{N+1} \cdot \xi^{N+1}} \, ,
\end{align}
in which the Hessian $D^2h^*$ is evaluated at $w^* = \sum_{\ell=1}^{N} q_{\ell} \, \xi^{\ell} + \mathscr{M}(\varrho, \, q) \, \xi^{N+1}$. We thus see that $\mathscr{M}\in C^2(\mathbb{R}_+ \times \mathcal{H}^{N}_-)$. Clearly, the formula \eqref{PRESSUREDEF} implies that $P \in C^2(\mathbb{R}_+ \times \mathcal{H}^{N}_-)$.
\end{proof}
\begin{rem}\label{NEWVARB}
All thermodynamic quantities can now be introduced as functions of the variables $\varrho, \, q$. 
Indeed, considering a thermodynamic function $f = \tilde{f}(T, \, \rho_1, \ldots, \rho_N)$ with a certain constitutive function $\tilde{f}$, we use that $T = -1/q_N$ and $\rho = \mathscr{R}(\varrho, \, q)$ (see \eqref{RHONEW}), and we define
\begin{align}\label{changetoentropic}
 f(\varrho, \, q) := \tilde{f}\Big(-\frac{1}{q_N}, \, \mathscr{R}_1(\varrho, \, q), \ldots, \mathscr{R}_N(\varrho, \, q)\Big)
\end{align}
to obtain the equivalent representation in the entropic variables.
\end{rem}
Using this remark, we introduce next some functions that we will need later on. The heat capacity $c_{\upsilon}$ at constant volume is defined as
\begin{align*}
c_{\upsilon} = \widetilde{c_{\upsilon}}(T, \, \rho) = \frac{1}{\varrho}\, \partial_T \epsilon(T, \, \rho) \, ,
\end{align*}
with the $\epsilon$ from \eqref{varrhoubasic}. We define
\begin{align}\label{NEWCV}
c_{\upsilon}(\varrho, \, q) := \frac{1}{\varrho}\, \partial_T \epsilon\Big(-\frac{1}{q_N}, \, \mathscr{R}_1(\varrho, \, q), \ldots, \mathscr{R}_N(\varrho, \, q)\Big) \, .
\end{align}
Similarily, for the shear viscosity coefficient $\eta$, we have
\begin{align*}
 \eta = \tilde{\eta}(T, \, \rho) = \tilde{\eta}\Big(-\frac{1}{q_N}, \, \mathscr{R}(\varrho, \, q)\Big) =: \eta(\varrho, \, q) \, , 
\end{align*}
and the same for the bulk viscosity $\lambda$.
We shall write
\begin{align*}
 \mathbb{S}(\varrho, \, q, \, \nabla v) = 2\eta(\varrho,q) \, (\nabla v)_{\text{sym}}  + \lambda(\varrho, q) \, (\divv v) \,  \mathbb{I} \, .
\end{align*}

\subsection{Reformulation of the partial differential equations and of the main theorem}\label{reformulation}

We introduce the combined flux vector
\begin{align}\label{headdiff}
\mathcal{J} := (J^1, \ldots,J^N, \, J^{\rm h}) \quad \text{ in } \quad \mathbb{R}^{(N+1 )\times 3} \, ,
\end{align}
the force vector
\begin{align}\label{bodydy}
b(x,t) = (b^1(x,t),\ldots,b^N(x,t), \, 0) \quad \text{ in } \quad \mathbb{R}^{(N+1) \times 3}
\end{align}
and the bulk production vector
\begin{align}\label{produit}
 \pi = \big(r_1, \ldots, r_N, \, (-p\, \mathbb{I}+ \mathbb{S}) \, :\, \nabla v + \mathcal{J} \, : \, b\big) \quad \text{ in } \quad \mathbb{R}^{N+1} \, . 
\end{align}
With these notations, we express the relations \eqref{DIFFUSFLUX} and \eqref{HEATFLUX} more compactly as
\begin{align*}
 \mathcal{J}^i = - \sum_{j=1}^{N+1} \mathcal{M}_{ij}(T, \, \rho) \, \nabla (w^*_j - b^j(x,t)/T) \, .
\end{align*}
We define
\begin{align*}
 \tilde{b}^{\ell} := \sum_{i=1}^{N+1} \eta^{\ell}_i \, b^i(x,t) \text{ for } \ell = 1, \ldots, N, \quad \bar{b} = \sum_{i=1}^{N+1} \eta^{N+1}_i \, b^i \, ,
\end{align*}
so that $ b(x,t) = \sum_{\ell = 1}^N \tilde{b}^{\ell}(x,t) \, \xi^{\ell} + \bar{b}(x,t) \, \xi^{N+1}$. We recall in this place that $\mathcal{M}$ has the one-dimensional kernel $\{\xi^{N+1}\}$, that $-1/T = q_N$, and we see that the fluxes have the form
\begin{align*}
& \mathcal{J}^i = - \sum_{j=1}^{N+1} \mathcal{M}_{i,j}(T,\rho) \, (\nabla w^*_j  +q_N \,  b^j(x, \, t))\\
&=  -\sum_{j=1}^{N+1}\left[ \sum_{\ell = 1}^{N}  \mathcal{M}_{i,j}(T,\rho) \, \xi^{\ell}_j \, (\nabla q_{\ell} + q_N \,  \tilde{b}^{\ell}) + \mathcal{M}_{i,j}(T,\rho) \, \xi^{N+1}_j \, (\nabla \mathscr{M}(\varrho, \,q) + q_N \, \bar{b}(x, \, t))\right] \\
& \quad = - \sum_{\ell = 1}^{N}  \left[\sum_{j=1}^{N+1} \mathcal{M}_{i,j}(T,\rho) \, \xi^{\ell}_j\right] \, (\nabla q_{\ell} + q_N \, \tilde{b}^{\ell}) \, .
 \end{align*}
If we introduce the rectangular projection matrix $\mathcal{Q}_{j\ell} = \xi^{\ell}_j$ for $\ell = 1,\ldots,N$ and $j = 1,\ldots,N+1$, then $\mathcal{J} = - \mathcal{M} \, \mathcal{Q} (\nabla q + q_N \, \tilde{b})$. Thus, the basic PDEs read equivalently
\begin{alignat*}{2}
\partial_t w + \divv( w \, v - \mathcal{M} \, \mathcal{Q}\, (\nabla q + q_N \, \tilde{b}(x, \,t))) & = \pi \, , & & \\
\partial_t (\varrho \, v) + \divv( \varrho \, v\otimes v - \mathbb{S}(\varrho, q, \, \nabla v)) + \nabla p & = \sum_{i=1}^N w_i \, b^i & &  \, .
\end{alignat*}
Next we define, for $k = 1,\ldots, N$, the maps
\begin{align}\label{RHONEWPROJ}
R_{k}(\varrho, \, q) & := \sum_{j=1}^{N+1} \xi^{k}_j \, w_j = \sum_{j=1}^{N+1} \xi^{k}_j \, \partial_{w^*_j}h^*( \sum_{\ell=1}^{N} q_{\ell} \, \xi^{\ell} + \mathscr{M}(\varrho, \,q) \, \xi^{N+1}  ) \, .
\end{align}
Obviously we can express 
$$w_i = \sum_{k = 1}^{N+1} w \cdot \xi^k \, \eta^k_i = \sum_{k=1}^{N} R_k(\varrho, \, q) \, \eta^k_i + \varrho \, \eta^{N+1}_i \, . $$
\begin{lemma}\label{rhonewlemma}
 Suppose that $h \in C^{3}(\mathbb{R}^N_+ \times \mathbb{R})$ satisfies the assumptions of Lemma \ref{Legendre}. Then, the formula \eqref{RHONEWPROJ} defines $R$ as a vector field of class 
 $C^2(\mathbb{R}_{+} \times \mathcal{H}^{N}_-; \, \mathbb{R}^{N})$. The Jacobian $\{R_{k,q_j}\}_{k,j=1,\ldots,N}$ is symmetric and positively definite at every $(\varrho, \, q) \in \mathbb{R}_{+} \times \mathcal{H}^{N}_-$ and
\begin{align*}
R_{q}(\varrho, \, q) = \mathcal{Q}^T \, D^{2}h^* \, \mathcal{Q} - \frac{\mathcal{Q}^T \, D^2h^* \xi^{N+1} \otimes \mathcal{Q}^T \, D^2h^* \xi^{N+1}}{D^2h^* \xi^{N+1} \cdot \xi^{N+1}} \, .
\end{align*}
In this formula, the Hessian $D^2h^*$ is evaluated at $w^* = \sum_{\ell=1}^{N} q_{\ell} \, \xi^{\ell} + \mathscr{M}(\varrho, \, q) \, \xi^{N+1}$.
\end{lemma}
The proof is direct, using Lemma \ref{Legendre} and the properties of $\mathscr{M}$.

For $k=1,\ldots,N$, we multiply the mass transport and energy transport equations with $\xi^{k}_i$, and we obtain that
\begin{align*}
\partial_t R_k +\divv\Big(R_k \, v - \sum_{\ell=1}^N \widetilde{\mathcal{M}}_{k,\ell}& \,  (\nabla q_{\ell} + q_N \,  \tilde{b}^{\ell}) \Big) = (\mathcal{Q}^T \, \pi)_k\, , \\
 \widetilde{\mathcal{M}}(T,\rho) & := \mathcal{Q}^T \, \mathcal{M}(T,\rho) \, \mathcal{Q} \, .
\end{align*}
It turns out that if the rank of $\mathcal{M}(T,\rho)$ is $N$ on all states $(T,\rho) \in \mathbb{R}^{N+1}_+$, the matrix $\widetilde{\mathcal{M}}(T,\rho)$ is strictly positively definite on all states $(T,\rho) \in \mathbb{R}^{N+1}_+$. Making use of \eqref{MUAVERAGE}, \eqref{RHONEW}, we can also consider $\widetilde{\mathcal{M}}$ as a mapping of the variables $\varrho$ and $q$. Using the Lemma \ref{rhonewlemma}, we then establish the following properties of this map.
\begin{lemma}\label{Mnewlemma}
 Suppose that $h \in C^{3}(\mathbb{R}^N_+\times \mathbb{R})$ satisfies the assumptions of Lemma \ref{Legendre}. Suppose further that $\mathcal{M}: \, \mathbb{R}^{N+1}_+ \rightarrow \mathbb{R}^{N+1\times N+1}$ is a map into the positively semi-definite matrices of rank $N$ with kernel $\{\xi^{N+1}\}$, having entries $\mathcal{M}_{ij}$ of class $C^{2}(\mathbb{R}^{N+1}_+) \cap C(\overline{\mathbb{R}}^{N+1}_{+})$. Then the formula $ \widetilde{\mathcal{M}}(\varrho, \, q) := \mathcal{Q}^T \, \mathcal{M}(T,\rho) \, \mathcal{Q}$ defines a map $\widetilde{\mathcal{M}}: \, \mathbb{R}_+ \times \mathcal{H}^{N}_- \rightarrow \mathbb{R}^{N\times N}$ into the symmetric positively definite matrices. The entries $\widetilde{\mathcal{M}}_{k,j}$ are functions of class $C^{2}(]0, \, +\infty[ \times \mathcal{H}^{N}_-)$. 
\end{lemma}
Overall, we get for the variables $(\varrho, \, q_1, \ldots, q_{N}, \, v)$ instead of \eqref{mass}, \eqref{energy}, \eqref{momentum} the equivalent equations
\begin{alignat}{2}
\label{mass2} \partial_t R(\varrho, \, q) + \divv( R(\varrho, \, q) \, v - \widetilde{\mathcal{M}}(\varrho, \, q) \, (\nabla q + q_N \, \tilde{b}(x, \, t)) ) & = \tilde{\pi}(x,t, \, \varrho, \, q, \, \nabla q, \, \nabla v)  \, ,& & \\
\label{mass2tot}\partial_t \varrho + \divv(\varrho \, v) & = 0 \, ,& & \\
\label{momentum2} \partial_t (\varrho \, v) + \divv( \varrho \, v\otimes v - \mathbb{S}(\varrho,q, \, \nabla v)) + \nabla P(\varrho, \, q) & = R(\varrho, \, q) \cdot \tilde{b}(x, \, t) + \varrho \, \bar{b}(x, \, t)  & & \, .
\end{alignat}
Here we have reinterpreted
\begin{align*}
& \mathcal{Q}^{\sf T} \pi =  \mathcal{Q}^{\sf T} \big(r_1(T,\rho), \ldots, r_N(T, \, \rho), \, (-p \, \mathbb{I} + \mathbb{S}) \, : \, \nabla v + \mathcal{J} \, : \, b \big) \\
 =& \mathcal{Q}^{\sf T} \, \big(r(\frac{-1}{q_N}, \, \mathscr{R}(\varrho, \, q)), \,  (-P(\varrho, \, q) \, \mathbb{I} + \mathbb{S}(\varrho, \, q, \, \nabla v)) \, : \, \nabla v - \widetilde{\mathcal{M}}(\varrho,q) \, (\nabla q+q_N \,  \tilde{b}) \cdot \tilde{b} \big) \\
 =: & \tilde{\pi}(x,t,\varrho, \, q, \, \nabla q, \, \nabla v) \, .
\end{align*}
In the PDE system \eqref{mass2}, \eqref{mass2tot}, \eqref{momentum2} we are faced with two type of constraints: The positivity constraint on the total mass density $\varrho$, and the half-space constraint $q \in \mathcal{H}^N_{-}$.

Our aim is now to show that the system \eqref{mass2}, \eqref{mass2tot}, \eqref{momentum2} for the variables $(\varrho, \, q_1, \ldots, q_{N}, \, v)$ is locally well-posed. We consider initial conditions
\begin{align}\label{initialpr}
\begin{split}
q(x, \, 0) & = q^0(x)  ,\\
\varrho(x, \, 0) & = \varrho_0(x) \, ,\\
v(x, \, 0) & = v^0(x)  \, ,
\end{split}
\quad
\text{ for } x \in \Omega\,
\end{align}
To find the initial data for $q^0(x)$, we first use the function $\varrho_0 u_0$ of \eqref{initialeps}. Then we define
\begin{align}\label{initialq}
 q^0_{\ell}(x) := \eta^{\ell} \cdot \nabla_wh(\rho^0(x), \, \varrho_0u_0(x)) \quad \text{ for } \quad \ell = 1,\ldots,N \, .
\end{align}
Under the assumptions of Theorem \ref{MAIN}, the new initial data are of class
\begin{align*}
q^0 \in W^{2-\frac{2}{p}}_p(\Omega; \, \mathbb{R}^{N}), \, \varrho_0 \in W^{1,p}(\Omega), \, v^0 \in W^{2-\frac{2}{p}}_p(\Omega; \, \mathbb{R}^{3}) \, ,
\end{align*}
satisfying $\varrho^0(x) \geq m_0 > 0$ in $\Omega$, $q^0_N(x) \leq - 1/\theta_1 < 0$ for all $x \in \Omega$.

%
%
%

We already showed that $\mathcal{J} = - \mathcal{M} \, \mathcal{Q} (\nabla q+ q_N \, \tilde{b})$. With $\pi^{\Gamma} := -(r^{\Gamma}, \, r^{\Gamma}_{\rm h})$ and $\mathcal{J}^{\Gamma} := (J^{\Gamma}, \, J^{\Gamma}_{\rm h})$ on the surface $S_{\bar{\tau}}$, we thus obtain the conditions
\begin{align*}
- \mathcal{M} \, \mathcal{Q} (\nabla q+ q_N \, \tilde{b}) \cdot \nu(x) = \pi^{\Gamma} + \mathcal{J}^{\Gamma} \, ,
\end{align*}
and therefore, with $\widetilde{\mathcal{M}} = \mathcal{Q}^{\sf T} \mathcal{M} \, \mathcal{Q} $
\begin{align*}
 - \widetilde{ \mathcal{M}} \,  (\nabla q + q_N \, \tilde{b}) \cdot \nu(x) = \mathcal{Q}^{\sf T}\pi^{\Gamma} + \mathcal{Q}^{\sf T}\mathcal{J}^{\Gamma}  \, .
\end{align*}
Next the matrix $\widetilde{\mathcal{M}}$, which is strictly positive definite, can also be inverted, and we get
\begin{align*}
 (\nabla q + q_N \, \tilde{b})\cdot \nu(x) = [\widetilde{\mathcal{M}}]^{-1} \, \mathcal{Q}^{\sf T}\pi^{\Gamma} + [\widetilde{\mathcal{M}}]^{-1} \, \mathcal{Q}^{\sf T}\mathcal{J}^{\Gamma} \, .
\end{align*}
Hence, for simplicity, we might consider the boundary conditions
\begin{align}\label{lateralpr}
\begin{split}
v & = 0 \, , \\ 
\nu(x) \cdot \nabla q_{k} & = \tilde{\pi}^{\Gamma}_k(x,t, \, \varrho, \, q)  \text{ for } k = 1,\ldots,N \, .
\end{split} \quad \text{ on } S_{\bar{\tau}} \, ,
\end{align}
where $\tilde{\pi}^{\Gamma}(x,t,\varrho, \, q)$ stands for
\begin{align*}
\tilde{\pi}^{\Gamma}(x,t, \, \varrho, q) = - q_N \, \tilde{b}(x,t) \cdot \nu(x)+ [\widetilde{\mathcal{M}}(\varrho, \, q)]^{-1} \, \mathcal{Q}^{\sf T}\pi^{\Gamma}(x,t, \varrho, \, q) + [\widetilde{\mathcal{M}}(\varrho,q)]^{-1} \, \mathcal{Q}^{\sf T}\mathcal{J}^{\Gamma}(x,t) \, .
\end{align*}
Owing to the Lemmas \ref{pressurelemma}, \ref{rhonewlemma} and \ref{Mnewlemma}, the coefficient functions $R$, $\widetilde{\mathcal{M}}$ and $P$ are of class $C^2$ in the domain of definitions $\mathbb{R}_+ \times \mathcal{H}^{N}_-$.

\vspace{0.2cm}

\section{Proof of the existence theorem}

\subsection{General method}

For the equivalent problem \eqref{mass2}, \eqref{mass2tot}, \eqref{momentum2} with boundary conditions \eqref{initialpr}, \eqref{lateralpr} we search for a solution in the class of optimal mixed regularity
\begin{align}\label{MAXMIXREG}
(q, \, \varrho,\,v) \in W^{2,1}_p(Q_{t}; \, \mathcal{H}^N_{-}) \times W^{1,1}_{p,\infty}(Q_{t}; \, \mathbb{R}_+) \times W^{2,1}_p(Q_{t}; \, \mathbb{R}^{3}) =: \mathcal{X}_t  \, .
\end{align}
The analysis of equations \eqref{mass2}, \eqref{mass2tot}, \eqref{momentum2} in this class was basically studied in the paper \cite{bothedruet} devoted to the isothermal case. In the present paper, we must deal with the following extensions:
\begin{itemize}
\item The production term $\tilde{\pi}$ depends on the velocity gradient squared;
\item The viscosity coefficients $\eta$ and $\lambda$ depends on the thermodynamic state;
\item There are non-trivial boundary fluxes;
\item The variable $q$ is restricted by the half-space constraint $\mathcal{H}^N_-$.
\end{itemize}
We will prove the local well-posedness by adapting the method of \cite{bothedruet} to this new case.

Due to the fact that the coefficients of the transformed PDE system are singular when $\varrho \rightarrow 0+$ or $q_N \rightarrow 0-$, the domain of the differential operator is restricted to the subset
\begin{align*}
  \mathcal{X}_{\bar{\tau},+} := \{(q, \, \varrho, \, v) \in \mathcal{X}_{\bar{\tau}}\, : \, q_N(x, \, t) <0 , \, \varrho(x, \, t) > 0 \text{ for all } (x, \, t ) \in Q_{\bar{\tau}}\} \, .
\end{align*}
For the parabolic components, we also introduce
\begin{align*}
  \mathcal{Y}_{\bar{\tau}} := \{(q, \, \, v) \, : \, q \in W^{2,1}_p(Q_{\bar{\tau}}, \, \mathbb{R}^N), \,  \, v \in W^{2,1}_p(Q_{\bar{\tau}}, \, \mathbb{R}^3)\} \, , 
\end{align*}
For proving a self-map property, we introduce for $K_0, \, \bar{\theta} > 0$ the closed subsets
\begin{align*}
 \mathcal{Y}_{\bar{\tau},K_0,\bar{\theta}} := \{(q, \, v) \in \mathcal{Y}_{\bar{\tau}} \, :\, q_{N} \leq -1/\bar{\theta}, \, \|(q,\, v)\}_{\mathcal{Y}_{\bar{\tau}}} \leq K_0\} \, .
\end{align*}
We linearise the PDEs as follows. We give arbitrary $(q^*, \, v^*) \in \mathcal{Y}_{\bar{\tau},K_0,\bar{\theta}}$, and for unknowns $u = (q, \, \varrho, \, v)$, we consider the following system of equations
\begin{align}\label{linearT1}
\partial_t \varrho + \divv (\varrho \, v^*) = & 0 \, ,\\
\label{linearT2} R_{q}(\varrho, \,q^*) \, \partial_t q - \divv (\widetilde{\mathcal{M}}(\varrho, \, q^*) \, \nabla q)  = & g(x, \, t, \, q^*, \, \varrho,\, v^*, \, \nabla q^*, \, \nabla \varrho, \, \nabla v^*) \, ,\\
\label{linearT3} \varrho \, \partial_t v - \divv \mathbb{S}(\varrho, \, q^*, \, \nabla v) = & f(x, \, t,\, q^*, \, \varrho,\, v^*, \, \nabla q^*, \, \nabla \varrho, \, \nabla v^*) \, .
\end{align}
Here we have set
\begin{align}\label{stressnew}
 \mathbb{S}(\varrho, \, q^*, \, \nabla v) = 2\eta(\varrho, \, q^*) \, (\nabla v)_{\text{sym}} + \lambda(\varrho, \, q^*) \, \mathbb{I} \, \divv v \, .
\end{align}
In \eqref{linearT1} and \eqref{linearT3}, the right-hands $f$ and $g$ are given by
\begin{align}\label{A1right}
& g(x, \, t,\, q, \, \varrho,\, v, \, \nabla q, \, \nabla \varrho, \, \nabla v) := (R_{\varrho}(\varrho,\, q) \, \varrho - R(\varrho,\, q)) \, \divv v - R_q(\varrho,\, q) \, v \cdot \nabla q\nonumber \\
 & \qquad + q_N \, \widetilde{\mathcal{M}}_{\varrho}(\varrho, \, q) \, \nabla \varrho \cdot \tilde{b}(x, \,t) + q_N \, \widetilde{\mathcal{M}}_{q}(\varrho, \, q) \, \nabla q \cdot \tilde{b}(x, \,t)  + q_N \, \widetilde{\mathcal{M}}(\varrho, \, q) \, \divv \tilde{b}(x, \,t)  \nonumber\\
 & \qquad +  \widetilde{M}(\varrho,q) \,  \tilde{b}(x,t) \cdot \nabla q_N + \tilde{\pi}(x,t, \, \varrho, \, q, \, \nabla q, \, \nabla v) \, , \\[0.2cm]
& \label{A3right} f(x, \, t,\, q, \, \varrho,\, v, \,  \nabla q, \, \nabla \varrho, \, \nabla v) := - P_{\varrho}(\varrho,\, q) \, \nabla \varrho -  P_{q}(\varrho,\, q) \, \nabla q 
 - \varrho \, (v\cdot \nabla)v \nonumber\\
 & \phantom{f(x, \, t,\, q, \, \varrho,\, v, \,  \nabla q, \, \nabla \varrho, \, \nabla v)  }  \qquad+ R(\varrho, \,q) \cdot \tilde{b}(x, \,t) + \varrho \, \bar{b}(x, \,t) \, .
\end{align}
We consider these PDEs together with the initial conditions \eqref{initialpr}, and with
\begin{align}\label{lateralvlinear}
v = 0, \quad  \nu \cdot \nabla q_k = \tilde{\pi}^{\Gamma}_k(x, \, t, \, \varrho, \, q^*) \, , \qquad \text{ on } S_{\bar{\tau}} \,.
\end{align}

The continuity equation for $\varrho$ is first solved independently in $W^{1,1}_{p,\infty}(Q_{\bar{\tau}})$, and the solution remains positive on the entire time-interval (see Theorem 1 in \cite{solocompress}, Proposition 7.5 in \cite{bothedruet}). Then, the problem \eqref{linearT2}, \eqref{linearT3} is linear in $(q, \, v)$ and can be solved in $\mathcal{X}_{\bar{\tau}}$ with the theory of Petrovski parabolic systems (see \cite{ladu}, Chapter 10, \cite{bothedruet}, Proposition 7.1). However, the equations \eqref{linearT2}, \eqref{linearT3} might not be solvable in $\mathcal{X}_{\bar{\tau},+}$ on the entire time interval, because the solution $q$ might reach the boundary of the half-plane $\mathcal{H}^N_-$ in finite time.

Nevertheless, we can show that the solution map $(q^*, \, v^*) \mapsto (q, \, v)$, denoted $\mathcal{F}$, is well defined from the closed subset $\mathcal{Y}_{t_0,K_0,\bar{\theta}}$ into itself for appropriate choices of the positive constants $t_0$ (small constant), and $K_0$ and $\bar{\theta}$ (large constants). 

To establish this key point, we rely on continuous estimates expressing the controlled growth of the solution in time. For the pair $(q, \, v) \in \mathcal{Y}_{\bar{\tau}}$ solving \eqref{linearT2} and \eqref{linearT3}, and for all $0 < t \leq \bar{\tau}$, we obtain an estimate
\begin{align}\label{continuity1}
 \|(q, \, v)\|_{W^{2,1}_p(Q_t; \, \mathbb{R}^{N+3})} \leq \Psi(t, \, R_0, \, \theta^*(t), \, \|(q^*, \, v^*)\|_{W^{2,1}_p(Q_t; \, \mathbb{R}^{N+3}) })=: \Psi^*_t \, .
\end{align}
Here $R_0$ is a fixed number depending on the initial data $q^0$, $\varrho_0$ and $v^0$, and on the external forces $b$ and fluxes $r^{\Gamma}$ and $J^{\Gamma}$ in their respective norms, while
\begin{align*}
 \theta^*(t) := -\frac{1}{ \sup_{(x,\tau) \in Q_t} q_N^*(x,\tau) } \quad \text{ for } \quad 0 < t < \bar{\tau} \, .
\end{align*}
The function $\Psi = \Psi(t, a_1, \, a_2, \, a_3)$ occurring in \eqref{continuity1} is defined on $[0, \, +\infty[^4$, it is continuous up to the boundary, and increasing in all arguments. Moreover, we can establish the important property that 
\begin{align}\label{Psinull}
\Psi(0, \, a_1,a_2,a_3) = \Psi^0(a_1,a_2) \, , 
\end{align}
where $\Psi^0$ is a continuous function independent on the last argument $a_3$ -- which stands for the norm of highest order of the data $q^*$ and $v^*$.

Let us now discuss the differences with respect to \cite{bothedruet} occurring in the proof of the estimate \eqref{continuity1}. For what the lower-order term $\tilde{\pi}(\varrho, \, q^*, \, \nabla v^*)$ occurring in the right-hand $g$ (cf.\ \eqref{A1right}) is concerned, this term can be estimated from above via
\begin{align*}
| \tilde{\pi}(\varrho(x,t), \, q^*(x,t), \, \nabla v^*(x,t))| \leq \phi\big(|\varrho(x,t)|, \, (\inf_{x\in\Omega}  \varrho(x,t))^{-1}, \, |q^*(x,t)|, \, \theta^*(t)\big) \, (1+|\nabla v^*|^2) \, ,
\end{align*}
with a certain continuous function $\phi$.
Hence, observing that $2 < 2 - 3/p + 3/(5-p)^+$ (since $p>3 > 5/2$), the quadratic growth in $|\nabla v^*|$ is subcritical, and the Lemma 8.1 of \cite{bothedruet} shows that this case is in fact also comprised in the analysis.

The second difference is that the coefficients $\eta$ and $\lambda $ are depending on the state. In order to deal with this situation, we must apply a localisation technique similar to the one used for the parabolic system concerning $q$ in \cite{bothedruet}. To avoid to much technicalities in this place, we delay the proof to the appendix. In the same way, we show in the appendix how to consider the case of nonzero boundary fluxes.

Presently, it thus remain only to show how, relying on \eqref{continuity1}, to control the distance of the solution $q$ to the boundary of the half plane $\mathcal{H}^N_-$.  

To this aim, we construct an extension $q^0 \in W^{2,1}_p(Q_{\bar{\tau}}, \, \mathbb{R}^N)$ of the initial data $q^0$ into $Q_{\bar{\tau}}$ by solving the Neumann problem for the heat equation in $Q_{\bar{\tau}}$ with initial data $q^0$. Then, the maximum principle also guarantees that
\begin{align*}
 \sup_{(x,t) \in Q_{\bar{\tau}}} q_N^0(x,\,t) \leq  \sup_{x \in \Omega} q^0_N(x) =: - 1/\theta_1 \, .
\end{align*}
For all $t \leq {\bar{\tau}}$, we employ the inequality of Lemma C.2 in \cite{bothedruet}, the continuity properties of the extension operator for $q^0$, and the inequality \eqref{continuity1} to obtain that
\begin{align*}
\begin{split} & \|q-q^0\|_{L^{\infty}(Q_t)} \leq  C(\bar{\tau}) \,  t^{\gamma} \, \|q-q^0\|_{W^{2,1}_p(Q_t; \, \mathbb{R}^N)}\\
 \leq &  C(\bar{\tau}) \, t^{\gamma} \, ( \|q^0\|_{W^{2,1}_p(Q_t, \, \mathbb{R}^N)} +\|q\|_{W^{2,1}_p(Q_t;\,\mathbb{R}^N)} ) \\
 \leq & \bar{C}(\bar{\tau}) \,t^{\gamma} \, (\|q^0\|_{W^{2-2/p}_p(\Omega; \, \mathbb{R}^N)} + \Psi_t^*) \, ,
 \end{split}
  \quad 
\gamma := \begin{cases}
            \frac{1}{2} \, (2 - \frac{5}{p}) & \text{ for } 3<p<5 \, ,\\
             \frac{p-1}{3+p} & \text{ for } 5 \leq p \, . 
            \end{cases}
            \end{align*}
Thus we get
\begin{align}\label{continuity20}
\sup_{(x,\tau) \in Q_t} q_{N}(x, \, \tau) \leq & \sup_{(x, \, \tau) \in Q_t} q^0_{N}(x, \, \tau)  +\bar{C}(\bar{\tau}) \,t^{\gamma} \, (\|q^0\|_{W^{2-2/p}_p(\Omega; \, \mathbb{R}^N)} + \Psi_t^*) \nonumber\\
\leq & -\frac{1}{\theta_1} + \bar{C}(\bar{\tau}) \,t^{\gamma} \, (\|q^0\|_{W^{2-2/p}_p(\Omega; \, \mathbb{R}^N)} + \Psi_t^*)
\, .
\end{align}
\begin{lemma}\label{selfmap}
We assume the validity of the estimates \eqref{continuity1}, \eqref{continuity20}, and we let $\Psi$, $\gamma >0$, $\bar{C} = \bar{C}(\bar{\tau})$ and $R_0 > 0$ be the data occurring there. Let $\Psi^0 = \Psi(0, \cdot)$ be the function occurring in \eqref{Psinull}.  
With $\bar{\theta}$ and $K_0$ being any two positive constants subject to the restrictions
\begin{align*}
\bar{\theta} > 2 \, \theta_1 , \quad K_0 > \Psi^0(R_0, \, \bar{\theta}) \, ,
\end{align*} 
we define $t_a = t_a(R_0,\bar{\theta},K_0)$ to be the largest number s.t.\ $\sup_{t \leq t_a} \Psi(t, \, R_0, \, \bar{\theta}, \, K_0) \leq K_0$, and
\begin{align*}
t^* := \min\Big\{t_a(R_0,\bar{\theta},K_0), \, \Big(\frac{1}{\bar{\theta} \, \bar{C} \, (\|q^0\|_{W^{2-2/p}_p} + K_0)}\Big)^{\frac{1}{\gamma}}\Big\} \, .
\end{align*} 
Then, for every $(q^*,v^*) \in \mathcal{Y}_{t^*,K_0,\bar{\theta}}$ the element $(q, \, v) = \mathcal{F}(q^*,v^*)$ belongs to $\mathcal{Y}_{t^*,K_0,\bar{\theta}}$.
\end{lemma}
\begin{proof}
Assume that $(q^*,v^*) \in \mathcal{Y}_{t^*, K_0, \bar{\theta}}$. 
In \eqref{continuity1} and \eqref{continuity20}, we have obtained for all $0 < t < t^*$ the following estimates:
\begin{align}\label{and}
& \|(q, \, v)\|_{\mathcal{Y}_t} \leq \Psi(t, \, R_0, \, \bar{\theta}, \, K_0)\nonumber\\
 \text{and } \quad &\\
& -\frac{1}{ \theta(t)} := \sup_{(x,\tau) \in Q_t}  q_{N}(x, \, \tau)  \leq - \frac{1}{\theta_1} + \bar{C} \, t^{\gamma} \, (\|q^0\|_{W^{2-2/p}} + \Psi(t, \, R_0, \, \bar{\theta}, \, K_0)) \nonumber\\
& \quad \leq - \frac{2}{\bar{\theta}}  + \bar{C} \, t^{\gamma} \, (\|q^0\|_{W^{2-2/p}} + \Psi(t, \, R_0, \, \bar{\theta}, \, K_0)) \nonumber \, .
\end{align}


With the definition of $t^*$, it now follows that
\begin{align*}
  \sup_{t \leq t^*} \Psi(t, \, R_0, \, \bar{\theta}, \, K_0) \leq \Psi(t_a, \, R_0, \, \bar{\theta}, \, K_0) = K_0 \, .
\end{align*}
Moreover, using the choice of $t^*$ again
\begin{align*}
\bar{ C} \, \sup_{t \leq t^*}  t^{\gamma} \, (\|q^0\|_{W^{2-2/p}_p} + \Psi(t, \, R_0, \, \bar{\theta}, \, K_0)) =\bar{ C} \, (t^*)^{\gamma} \, (\|q^0\|_{W^{2-2/p}_p} + \Psi(t^*, \, R_0, \, \bar{\theta}, \, K_0)) \\
 \leq \bar{C} \, (t^*)^{\gamma} \, (\|q^0\|_{W^{2-2/p}_p} + \Psi(t_a, \, R_0, \, \bar{\theta}, \, K_0)) = \bar{C} \, (t^*)^{\gamma} \, (\|q^0\|_{W^{2-2/p}_p} + K_0) \\
 \leq  \bar{C} \, (t^*)^{\gamma} \, (\|q^0\|_{W^{2-2/p}_p} + K_0) \leq\frac{1}{\bar{\theta}} \, .
\end{align*}
Thus, in view of \eqref{and}, we have obtained $ \|(q,v)\|_{\mathcal{Y}_{t^*}} \leq K_0$ and $ \theta(t^*) \leq \bar{\theta}$,
 and we see that $(q, \, v) \in \mathcal{Y}_{t^*,K_0,\bar{\theta}}$.
This means that the image $\mathcal{F}(q^*,\, v^*)$ is a well-defined element of $\mathcal{Y}_{t^*,K_0,\bar{\theta}}$ for $(q^*,v^*) \in \mathcal{Y}_{t^*,K_0,\bar{\theta}}$.
\end{proof}
With the self-mapping property at hand, the existence proof can be finalised by means of the same fixed-point iteration as in the section 9 of \cite{bothedruet}. Initialising $(q^1, \, v^1) = 0$, we define $(q^{n+1}, \, v^{n+1}) := \mathcal{F}(q^n, \, v^n)$ for $n \in \mathbb{N}$. This iteration converges in a lower-order norm on the interval $]0,t^*[$. This is sufficient to prove the existence of a unique limit solution. This point was extensively explained and proved there.

\vspace{0.2cm}

\section{A partial maximum principle for the energy equation}

Let $(q, \, \varrho, \, v) \in \mathcal{X}_{\bar{\tau}}$ be a solution of optimal mixed regularity to the system \eqref{mass2}, \eqref{mass2tot}, \eqref{momentum2} on some interval $]0, \, \bar{\tau}[$.

In this section we want to show that, under certain natural growth condition for the data of the problem, a strong solution with bounded state-space norm cannot break down due to blow-up of $\|T(\tau)\|_{L^{\infty}(\Omega)}$ as $\tau \rightarrow \bar{\tau}$. As we already explained, the regularity of $q_N = - 1/T$ alone does not prevent this type of blow-up. 


For our aim, we use the parabolicity of the equation for the internal energy density $\epsilon := \varrho u$. With a given solution $(q, \, \varrho, \, v)$, we first re-introduce the state variables
\begin{align}
\label{RHO} \rho(x, \, t) = & \sum_{k=1}^N R_k(\varrho, \, q) \, \eta^{k} + \varrho \, \eta^{N+1} \, ,\\
\label{U} \varrho u(x, \, t) =& R_{N}(\varrho, \, q)  (= \partial_{w_{N+1}} h^*\big(\mathcal{Q}^{\sf T} \, q + \mathscr{M}(\varrho, \, q) \, \xi^{N+1}\big)) \, ,
\end{align}
with the map $R$ of \eqref{RHONEWPROJ}. For $i = 1,\ldots,N$, the function $q_i$ belongs to $W^{2,1}_p(Q_{\bar{\tau}})$, while $\varrho$ possesses all generalised first derivatives in $L^{p,\infty}(Q_{\bar{\tau}})$. Together with the $C^2-$regularity of the coefficient $R$ in the domain $\mathbb{R}_+ \times \mathcal{H}^N_-$ (Lemma \ref{rhonewlemma}), the identities \eqref{RHO} and \eqref{U} allow to show that, for every $t$ such that
\begin{align}\label{noblowup}
\sup_{(x,\, \tau) \in Q_t} q_N(x, \, \tau) < 0 \, ,
\end{align}
the functions $\rho_i$ and $\varrho u$ belong to $W^1_p(Q_t)$. 
Hence, $\varrho u$ can be shown to be a weak solution to the energy equation \eqref{energy} for all $0<t< \bar{\tau}$ enjoying the property \eqref{noblowup}. To show this, we next restrict for simplicity to the case that $b^1 = \ldots = b^N$, which essentially means that the body forces reduce to gravity. In this case, we have $\tilde{b} \equiv 0$ so that $b$ does not contribute to the diffusion fluxes. Moreover, we simplify the proof assuming that there is no heat flux $- r^{\Gamma}_{\rm h} + J^{\Gamma}_{\rm h}$ on the boundary. We discuss afterwards the modifications necessary to handle the general case.
\begin{lemma}\label{weakform}
Let $\bar{\tau} > 0$ and assume that $(q, \, \varrho, \, v) \in \mathcal{X}_{\bar{\tau}}$ is a strong solution to the system \eqref{mass2}, \eqref{mass2tot}, \eqref{momentum2}, such that \eqref{noblowup} is valid for all $t < {\bar{\tau}}$. We define $\epsilon := R_{N}(\varrho, \, q)$. Then, for all $t < \bar{\tau}$, the function $\epsilon$ belongs to $W^{1}_p(Q_t)$ and, for all $\phi \in W^{1,0}_{p^{\prime}}(Q_t)$, it satisfies the integral identity
\begin{align*}
 \int_{Q_t} \partial_t \epsilon \, \phi \, dxd\tau - \int_{Q_t} (\epsilon \, v + J^{\rm h}) \cdot \nabla \phi \, dxd\tau = \int_{Q_t} (-p \, \mathbb{I} + \mathbb{S}) \, : \,  \nabla v \, \phi \, dxd\tau \, .
\end{align*}
\end{lemma}
\begin{proof}
For every $t < \bar{\tau}$, the optimal regularity together with the condition \eqref{noblowup} guarantee that $\epsilon$, $p$, the viscosities $\eta$ and $\lambda$ and all other continuous functions of $\varrho$ and $q$ are uniformly bounded on $Q_t$. The heat flux $J^{\rm h}$ satisfies an estimate in $L^{\infty,p}(Q_t; \, \mathbb{R}^3)$ at least. To see this, for $k = 1, \ldots,N-1$ we define $L_k :=  \xi^k \cdot (l, \, 0)$. Recalling \eqref{HEATFLUX} with $b = 0$ we have
\begin{align}\label{heatfluxbase}
 J^{\rm h} = \kappa \, \nabla \frac{1}{T} - L \cdot \nabla \bar{q} = - \kappa\Big(-\frac{1}{q_N}, \, \mathscr{R}(\varrho, \, q)\Big) \, \nabla q_N - L\Big(-\frac{1}{q_N}, \, \mathscr{R}(\varrho, \,q)\Big) \cdot \nabla \bar{q} \, .
\end{align}
With $\theta(t) := \sup_{Q_t} (-1/q_N) < + \infty$ and $M(t) := \sup_{Q_t} \varrho$, we see that
\begin{align*}
 |J^{\rm h}| \leq \sup_{\theta \leq \theta(t), \, |r| \leq M(t)} \{\kappa(\theta, \, r) + |L(\theta,r)|\} \, |\nabla q| = C(t) \, |\nabla q| \, .
\end{align*}
For $p > 3$, we have $W^{1,p}(\Omega) \subset L^{\infty}(\Omega)$. Hence, $J^{\rm h} \in L^{\infty,p}(Q_t)$ as claimed. 

Using that $W^{2-2/p}_p(\Omega) \subset L^{3p/(5-p)^+}(\Omega)$, the term $|\nabla v|^2$ satisfies
 \begin{align*}
  \||\nabla v|^2\|_{L^{z,\infty}(Q_t)} \leq c_0 \, \|v\|_{W^{2,1}_p(Q_t)}^2 \, \quad z := \frac{3p}{2 \, (5-p)^+} \, .
  \end{align*}
We note that $z > 3p/(p+3)$. With $p^{\prime} :=p/(p-1)$, the Sobolev embedding guarantees that $W^{1,p^{\prime}}(\Omega) \subset L^{3p/(2p-3)}(\Omega)$. Combining these properties and H\"older's inequality we see that
\begin{align*}
\int_{Q_t} |\nabla v|^2 \, |\phi| \, dxd\tau \leq \||\nabla v|^2\|_{L^{\frac{3p}{p+3},\infty}(Q_t)} \, \|\phi\|_{L^{\frac{3p}{2p-3},1}(Q_t)} \leq c_1 \, \|v\|_{W^{2,1}_p(Q_t)}^2 \, \|\phi\|_{W^{1,0}_{p^{\prime}}(Q_t)} \, . 
\end{align*}
Hence, all integral make sense in the weak form of the energy equation.
\end{proof}
Next we want to obtain an equivalent expression of the heat flux allowing us to use weak parabolic estimate techniques on the function $\epsilon$. To this aim, we need to introduce some further thermodynamic quantities. With the $h^*$ introduced in Lemma \ref{Legendre}, and the choices \eqref{labase} of the axes $\xi^1 ,\ldots, \xi^{N+1}$, we define
 \begin{align}
\label{coeffvarrho}  a_0 := & \frac{D^2_{w^*, w^*_{N+1}} h^* \cdot \xi^{N+1}}{D^2_{w^*, w^*} h^* \, \xi^{N+1} \cdot \xi^{N+1}} \, ,\\
\label{coeffq}  a_k := & D^2_{w^*, w^*_{N+1}} h^* \cdot \xi^{k} - \frac{D^2_{w^*, w^*_{N+1}} h^* \cdot \xi^{N+1} \, D^2_{w^*, w^*} h^* \, \xi^{N+1} \cdot \xi^{k}}{D^2_{w^*, w^*} h^* \, \xi^{N+1} \cdot \xi^{N+1}} \, \quad \text{ for } k = 1,\ldots, N-1 \, ,\\
\label{d0}  d_0 := & \partial^2_{w^*_{N+1}} h^* - \frac{(D^2_{w^*, w^*_{N+1}} h^* \cdot \xi^{N+1})^2}{D^2_{w^*, w^*} h^* \, \xi^{N+1} \cdot \xi^{N+1}} \, .
\end{align}
Therein, all occurrences of $h^*$ and its derivatives are evaluated at $w^* = \mathcal{Q}^{\sf T} q + \mathscr{M}(\varrho, \, q) \, \xi^{N+1}$. Hence, $a_0$, $a_k$ and $d_0$ are well-defined functions over $Q_{\bar{\tau}}$. Next, we show that these new objects occur naturally when we re-express the heat flux using the gradient of $\epsilon$.
\begin{lemma}
For the coefficient introduced in \eqref{d0}, we have
\begin{align*}
 d_0(\varrho, \, q) \geq \frac{\varrho}{q_N^2}  \, c_{\upsilon}(\varrho, \, q) \, \, (= T^2 \, \varrho \, c_{\upsilon}) \, ,
\end{align*}
with the function $c_{\upsilon}$ of \eqref{NEWCV}. Writing $q = (\bar{q}, \, q_N)$, the heat flux $J^{\rm h}$ is equivalently given as
\begin{align*}
J^{\rm h} = -\frac{\kappa}{d_0} \, (\nabla \epsilon - a_0 \, \nabla \varrho) -  \big(L -\frac{\kappa}{d_0} \, a\big) \cdot \nabla {\bar{q}} \, .
 \end{align*}
\end{lemma}
\begin{proof}
By the definition of the entropic variables, we have
\begin{align*}
 \varrho u = \partial_{w^*_{N+1}} h^*(\mathcal{Q}^{\sf T} q + \mathscr{M}(\varrho, \, q) \, \xi^{N+1}) \, .
\end{align*}
Hence, applying the chain rule yields
\begin{align*}
\nabla \epsilon = & \nabla \varrho u = \sum_{k=1}^{N-1} \sum_{i=1}^{N+1} \partial^2h^*_{w^*_{N+1},w_i^*} \, \xi^k_i \, \nabla q_k + \partial^2_{w^*_{N+1}}h^* \,  \nabla q_N  + \sum_{i=1}^N \partial^2_{w^*_{N+1},w_i^*} h^* \, \nabla \mathscr{M} \, ,\\
 \nabla \mathscr{M} =& \partial_{\varrho} \mathscr{M} \, \nabla \varrho + \sum_{k=1}^{N-1} \partial_{q_k} \mathscr{M} \, \nabla q_k + \partial_{q_N} \mathscr{M} \, \nabla q_N \, .
\end{align*}
Use of the equivalent expressions \eqref{Mnachrho} for the derivatives of $\mathscr{M}$ and of $T^2 \, \nabla q_N = \nabla T$ yields
\begin{align}\label{flux0}
 \nabla \epsilon = \sum_{k=1}^{N-1} a_k \, \nabla q_k + a_0 \, \nabla \varrho + \frac{d_0}{T^2} \, \nabla T\, .
\end{align}
The function $d_0$ defined in \eqref{d0} is clearly positive as $h^*$ is strictly convex. Due to \eqref{flux0}, we now have
\begin{align*}
& \nabla T = \frac{T^2}{d_0} \, \Big( \nabla \epsilon - \sum_{k=1}^{N-1} a_k \, \nabla q_k - a_0 \, \nabla \varrho \Big) \\
 \text{ implying that }  \quad &\\
& J^{\rm h} = - \frac{\kappa}{d_0} \, (\nabla \epsilon - a \cdot \nabla \bar{q} - a_0 \, \nabla \varrho) - L \cdot \nabla \bar{q} \, .
\end{align*}
This proves the equivalent representation of the heat flux.
Finally, we estimate $d_0$ from below. To this aim, we use that $\nabla_{w^*} h^*(\mu/T, \, -1/T) = (\rho, \, \varrho u)$. Therein, we now reinterpret $\mu$ and $\varrho u$ as functions of the main variables $(T, \, \rho_1,\ldots,\rho_N)$. 

For notational simplicity, let $\{A^*_{ik}\}_{i,k=1,\ldots,N}$ denote the upper left block of the matrix $D^2_{w^*}h^*$. For $i = 1,\ldots,N$, we differentiate in $T$ the identities $\partial_{w^*_i}h^* = \rho_i$ and $\partial_{w^*_{N+1}}h^* = \varrho u$, and thus
\begin{align*}
& \sum_{k=1}^{N} A^*_{ik} \, \partial_{T}(\mu_k/T) + \frac{1}{T^2} \partial^2_{w_{N+1}^*,w_i^*} h^* = 0\, , \quad \text{ for } i=1,\ldots,N \, ,\\
& \frac{1}{T^2} \, \partial^2_{w^*_{N+1}} h^* + \sum_{i=1}^N \partial^2_{w_{N+1}^*,w_i^*} h^* \, \partial_{T}(\mu_i/T) = \varrho \, c_{\upsilon} \, .
\end{align*} 
Combining both identities, we obtain that 
\begin{align*}
 \partial^2_{w^*_{N+1}} h^* = T^2 \, \varrho \, c_{\upsilon} + T^4 \, A^* \partial_{T}(\mu/T) \cdot \partial_{T}(\mu/T) \, .
\end{align*}
Hence, an equivalent form for $d_0$ is
\begin{align}\label{d0equiv}
 d_0 = T^2 \, \varrho \, c_{\upsilon} + T^4 \, \Big( A^* \partial_{T}(\mu/T) \cdot \partial_{T}(\mu/T) - \frac{ (A^* \partial_{T}(\mu/T) \cdot \xi^{N+1})^2}{A^* \xi^{N+1} \cdot \xi^{N+1}} \Big) \, .
\end{align}
Owing to the positivity of $A^*$, we see that $d_0 \geq T^2 \, \varrho \, c_{\upsilon}$, as claimed.
\end{proof}
We now re-express the integral identity of Lemma \ref{weakform} as
\begin{align}\label{weakwithe}
 & \int_{Q_t} \partial_t \epsilon  \, \phi \, dxd\tau + \int_{Q_t} \Big\{\frac{\kappa}{d_0} \, \nabla \epsilon + (L-\frac{\kappa}{d_0} \, a) \cdot \nabla \bar{q} - \frac{\kappa}{d_0} \, a_0 \, \nabla \varrho - \epsilon \, v\Big\} \cdot \nabla \phi \, dxd\tau\nonumber\\
 & \, \, = \int_{Q_t} (-p \, \mathbb{I} + \mathbb{S}) \, : \,  \nabla v \, \phi \, dxd\tau \, .
\end{align}
In order to obtain a maximum principle, we choose test functions of the form 
\begin{align*}
 w_k = (\epsilon - k)^+ := \max\{\epsilon-k, \ 0\} \, ,
\end{align*}
where $k$ is any parameter such that
\begin{align}\label{klargerk0}
 k > k_0 := \sup_{x \in \Omega} \epsilon(x, \, 0) \, .
\end{align}
We also define
 \begin{align*}
  \Omega_{t,\, k} := & \{x \in \Omega \, : \, \epsilon(x, \, t) > k\}, \quad
  Q_{t,k} := \{(x, \, \tau) \in Q_t \, : \, \epsilon(x, \, \tau) > k\} \, .
 \end{align*}
For all $0 < t < T$, we next insert $\phi(x, \, \tau) := w_k(x, \, \tau) \, \chi_{]0, \, t[}(\tau)$ in \eqref{weakwithe}, easily showing that 
\begin{align}\label{weakwhite}
& \frac{1}{2} \, \int_{\Omega_{t,k}} \varrho(x,t) \, w_k^2(x, \, t) \, dx + \int_{Q_{t,k}}  \frac{\kappa}{d_0} \, |\nabla \epsilon|^2 \, dxd\tau= \\
& \, - \int_{Q_{t,k}}\Big\{(L-\frac{\kappa}{d_0} \, a)  \cdot \nabla \bar{q} + \frac{\kappa}{d_0} \, a_0 \, \nabla \varrho + \epsilon \, v\Big\} \cdot \nabla \epsilon \, dxd\tau + \int_{Q_{t,k}} (-p \, \mathbb{I} + \mathbb{S}) \, : \,  \nabla v \, w_k \, dxd\tau \, . \nonumber
\end{align}
We also notice that
\begin{align*}
 \int_{Q_{t,k}} \epsilon \, v \cdot \nabla \epsilon \, dxd\tau & = \int_{Q_{t,k}} (\epsilon-k) \, v \cdot \nabla \epsilon \, dxd\tau + k \, \int_{Q_{t,k}} v \cdot \nabla \epsilon \, dxd\tau \\
 & = -\int_{Q_{t,k}} \divv v \, \Big(\frac{1}{2}\, w_k + k\Big) \, w_k \, dxd\tau \, .
\end{align*}
Using that $\frac{1}{2}\, w_k + k \leq \frac{3}{2}  \, \epsilon $ if $\epsilon \geq k$, we see that
\begin{align*}
 \left|\int_{Q_{t,k}} \epsilon \, v \cdot \nabla \epsilon \, dxd\tau \right| \leq \frac{3}{2} \,  \int_{Q_{t,k}} |\divv v| \, \epsilon \, w_k \, dxd\tau \, .
\end{align*}
By means of Young's inequality, we moreover estimate
\begin{align*}
\Big| \big\{(L-\frac{\kappa}{d_0} \, a)  \cdot \nabla \bar{q} + \frac{\kappa}{d_0} \, a_0 \, \nabla \varrho\big\} \cdot \nabla \epsilon\Big| \leq \frac{\kappa}{2d_0} \, |\nabla \epsilon|^2 + \frac{d_0}{2\kappa} \, \{|L-\frac{\kappa}{d_0} \, a|  \, |\nabla \bar{q}| + \frac{\kappa}{d_0} \, |a_0| \, |\nabla \varrho|\}^2 \, .
\end{align*}
Combining these ideas, \eqref{weakwhite} implies that
\begin{align}\label{weakwhite2}
& \frac{1}{2}\, \int_{\Omega_{t,k}} \varrho(x, \, t) \, w_k^2(x, \, t) \, dx + \int_{Q_{t,k}}  \frac{\kappa}{2d_0} \, |\nabla \epsilon|^2 \, dxd\tau \\
\leq & \int_{Q_{t,k}}\Big( \frac{d_0}{\kappa} \{|L-\frac{\kappa}{d_0} \, a| \, |\nabla \bar{q}| + \frac{\kappa}{d_0} \, |a_0| \, |\nabla \varrho|\}^2 +  \{|p| \, |\nabla v| + |\mathbb{S}| \, |\nabla v| + \frac{3}{2} \, \epsilon \, |\nabla\cdot v|\} \,  \, w_k\Big) \, dxd\tau \, .\nonumber
\end{align}
Depending on the asymptotic behaviour of the coefficients $\kappa$, $L$, $a$, etc.\ for large values of $\epsilon$ and fixed or bounded entropic variables $(\varrho, \, q)$, we can now obtain a maximum principle. In order to make the discussion more simple from the viewpoint of notations, we use the following convention.
\begin{conve}\label{Asi}
Consider a function $f$ of the thermodynamic state variables $(T, \, \rho)$ and $\alpha > 0$. We let $(\varrho, \, q)$ being the associated entropic variables and $\epsilon = \epsilon(T, \, \rho)$ the associated internal energy. If there are:
\begin{itemize}
\item a positive, continous function $k_1$ on $]0, \, +\infty[^2$;
\item a continuous, nonnegative function $\phi$ (a continuous, strictly positive function $\psi$) on $\mathbb{R}_+ \times \overline{\mathcal{H}^N_-}$
\end{itemize}
such that
\begin{align*}
f \leq \phi(\varrho, \, q) \, \epsilon^{\alpha}, \quad  (f \geq \, \psi(\varrho, \, q) \, \epsilon^{\alpha}) \qquad \text{ for all } \epsilon \geq k_1( \varrho, \, |q|) \, ,
\end{align*}
then we write $f \precsim \epsilon^{\alpha}, \, (f \succsim \epsilon^{\alpha})$.
We write $f \eqsim \epsilon^{\alpha}$ if both $f \precsim \epsilon^{\alpha}$ and $f \succsim \epsilon^{\alpha}$ hold.
\end{conve}

\begin{prop}\label{TECHON}
Assume that there are numbers $\beta \geq 1$ and $\beta_i \geq 0$ ($i = 0, \ldots,3$) such that the following estimates for the functions $\kappa$, $d_0$, $a_0$, $(a)$, $(L)$, $p$, $\eta$ and $\lambda$ are valid:
\begin{align}\label{growth} 
\begin{split}\frac{\kappa}{d_0} \succsim  \epsilon^{2(\beta-1)}, \quad & \frac{d_0}{\kappa} \, |L|^2 + \frac{\kappa}{d_0} \, |a|^2 \precsim  \epsilon^{\beta_0} , \quad \frac{\kappa}{d_0} \, |a_0|^2 \precsim  \epsilon^{\beta_1} \, , \\
& |p | \precsim  \epsilon^{\beta_2}, \quad \eta + |\lambda| \precsim \epsilon^{\beta_3} \, .\end{split}
\end{align}
%
We assume that for all $t < \bar{\tau}$, $\epsilon \in W^1_p(Q_{t})$ satisfies the weak form \eqref{weakwhite} and that $(\varrho, \, q, \, v) \in \mathcal{X}_{\bar{\tau}}$ is a solution of optimal regularity. We suppose, moreover, that $\epsilon \in L^{1,\infty}(Q_{\bar{\tau}})$. Then, under the conditons that $p > 5$ and that the exponents $p$, $\beta$ and $\beta_i$ occurring in the conditions \eqref{growth} are subject to
\begin{gather*}
  \max\{\beta_0, \, \beta_3+1\} <
\frac{6}{5} \, (1+\beta) \, , \qquad
 \beta_1 < \frac{6}{5}  \, (1+\beta - \frac{5}{3p}) \, ,\\
\beta_2 < \frac{1}{5} \, (6\beta+1) \, ,
 \end{gather*}
we have $\limsup_{t \rightarrow \bar{\tau}-} \|\epsilon\|_{L^{\infty}(Q_{t})} < + \infty$.
\end{prop}
\begin{proof}
Since $(\varrho, \, q, \, v) \in \mathcal{X}_{\bar{\tau}}$ by assumption, the norms $\|\varrho\|_{L^{\infty}(Q_{\bar{\tau}})}$ and $\|q\|_{L^{\infty}(Q_{\bar{\tau}}; \, \mathbb{R}^N)}$ are finite. Moreover, $\inf_{Q_{\bar{\tau}} } \varrho > 0$ follows from the continuity equation. We introduce $m(t) := \inf_{Q_t} \varrho > 0$, $M(t):= \sup_{Q_t} \varrho$ and $Q(t) := \|q\|_{L^{\infty}(Q_t)}$. Due to the assumptions \eqref{growth}, for all $$\epsilon \geq  \sup_{m(\bar{\tau}) \leq r \leq M(\bar{\tau}), \, |\xi| \leq Q(\bar{\tau})}  k_{1}(r, \, |\xi|) =: \bar{k}_1(\bar{\tau}) \, ,$$ we can rely on a bound
\begin{align*}
 \frac{\kappa}{d_0} \geq \psi(\varrho, \, q) \,   \epsilon^{2(\beta-1)} \geq \Big(\inf_{m(\bar{\tau}) \leq r \leq M(\bar{\tau}), \, |\xi| \leq Q(\bar{\tau})} \psi(r, \, \xi)\Big) \, \epsilon^{2(\beta-1)} = c(\bar{\tau}) \, \epsilon^{2(\beta-1)} \, . 
\end{align*}
Similarly, \eqref{growth} guarantees that
\begin{align*}
\frac{d_0}{\kappa} \, |L|^2 + \frac{\kappa}{d_0} \, |a|^2 \leq C(\bar{\tau}) \,   \epsilon^{\beta_0} , \quad \frac{\kappa}{d_0} \, |a_0|^2 \leq C(\bar{\tau}) \epsilon^{\beta_1} \, , \, \text{etc.} \quad \text{ for all } \epsilon \geq \bar{k}(\bar{\tau}) \, ,
\end{align*}
where the number $\bar{k}(\bar{\tau})$ is obtained as the maximum among all threshold functions $k_1$ associated with the coefficients (See the convention \ref{Asi}).
%
In \eqref{weakwhite2}, we restrict to $k > \max\{k_0, \, \bar{k}(\bar{\tau})\}$. Use of \eqref{weakwhite2} and of the growth conditions on $\kappa/d_0$ and the other coefficients yields
\begin{align}\label{weakwhite3}
& \frac{m(\bar{\tau})}{2} \, \int_{\Omega_{t,k}} w_k^2(x, \, t) \, dx + c(\bar{\tau}) \, \int_{Q_{t,k}}  \epsilon^{2(\beta-1)} \,  |\nabla \epsilon|^2 \, dxd\tau \\
& \leq C(\bar{\tau}) \,  \int_{Q_{t,k}} \Big(\epsilon^{\beta_0} \, |\nabla \bar{q}|^2 +\epsilon^{\beta_1} \, |\nabla \varrho|^2  + \{ ( \, \epsilon^{\beta_2} +\frac{3}{2} \, \epsilon) \, |\nabla v| + \epsilon^{\beta_3} \, |\nabla v|^2\} \, w_k\Big) \, dxd\tau \, . \nonumber
\end{align}
Since $\beta \geq 1$, we also observe that, for $\epsilon(x,t) \geq k$
\begin{align*}
 \epsilon^{2(\beta-1)} \,  |\nabla \epsilon|^2 \geq (\epsilon-k)^{2(\beta-1)} \,  |\nabla \epsilon|^2 = \frac{1}{\beta^2} \, |\nabla (\epsilon-k)^{\beta}|^2 \, .
\end{align*}
Moverover, since $w_k \leq \epsilon$, \eqref{weakwhite3} implies that, possibly with a larger constant $C(\bar{\tau})$,
\begin{align}\label{boundgen}
 & \int_{\Omega_k} w_k^2(x, \, t) \, dx + \int_{Q_{t,k}}  |\nabla w_k^{\beta}|^2 \, dxd\tau \\
 & \leq  C(\bar{\tau}) \, \int_{Q_{t,k}} \{\epsilon^{\beta_0} \, |\nabla \bar{q}|^2 + \epsilon^{\beta_1} \, |\nabla \varrho|^2 + ( \epsilon^{\beta_2+1} +\epsilon^2) \, |\nabla v| + \epsilon^{\beta_3+1} \, |\nabla v|^2\} \, dxd\tau \, . \nonumber
\end{align}
Due to the Lemma \ref{interpomoi}, restriction to $k >2 \,  \|\epsilon\|_{L^{1,\infty}(Q_t)}/\lambda_3(\Omega)$ implies that
\begin{align}\label{defsmallr}
 \|w_k\|_{L^r(Q_t)}^r \leq c_0 \, \|w_k\|_{L^{2,\infty}(Q_t)}^{\frac{4}{3}} \, \|\nabla w_k^{\beta}\|_{L^2(Q_t)}^2 \quad \text{ with } \quad  r = \frac{4}{3} + 2 \, \beta \, ,
 \end{align}
 with $c_0 = c_0(\Omega)$ being an embedding constant. Thus, with $R(t)$ denoting the right-hand of \eqref{boundgen}, we have
$   \|w_k\|_{L^r(Q_t)}^r \leq c_0 \, \left(R(t)\right)^{\frac{2}{3}} \, R(t)$
or, in other words,
\begin{align}\label{boundgen2}
 \|w_k\|_{L^r(Q_t)}^{\frac{3r}{5}} \leq  C(\bar{\tau}) \,  \int_{Q_{t,k}} \{\epsilon^{\beta_0} \, |\nabla \bar{q}|^2 + \epsilon^{\beta_1} \, |\nabla \varrho|^2 + ( \epsilon^{\beta_2+1} +\epsilon^2) \, |\nabla v| + \epsilon^{\beta_3+1} \, |\nabla v|^2\} \, dxd\tau \, .
 \end{align}
 We introduce functions and exponents
 \begin{align}
 \begin{split}\label{lesGs}
 G_0 := |\nabla \bar{q}|^2, \quad \alpha_0 := \beta_0, \quad G_1 := |\nabla \varrho|^2, \quad \alpha_1 := \beta_1\\
 G_2 := |\nabla v|, \quad \alpha_2 := \max\{2, \, \beta_2+1\}, \quad G_3 := |\nabla v|^2, \quad \alpha_3 := \beta_3+1 \, .
 \end{split}
 \end{align}
 Then, \eqref{boundgen2} can be rewritten as
 \begin{align}\label{boundgen3}
  \|w_k\|_{L^r(Q_t)}^{\frac{3r}{5}} \leq  C(\bar{\tau}) \, \sum_{i=0}^3 \int_{Q_{t,k}} G_i(x,\tau)  \, \epsilon^{\alpha_i} \, dxd\tau \, .
 \end{align}
In order to estimate a generic term of the form $\int_{Q_{t,k}} G_i(x,\tau)  \, \epsilon^{\alpha_i} \, dxd\tau$, we rely moreover on the information $\|\epsilon\|_{L^{1,\infty}(Q_t)} < +\infty$.

For $i = 0, \ldots3$, we choose numbers $ 0 \leq a_i < \min\{1, \, \alpha_i\}$ and we define $b_i := \alpha_i - a_i$, in order to re-express
\begin{align*}
\int_{Q_{t,k}} G_i(x,\tau)  \, \epsilon^{\alpha_i} \,  dxd\tau \int_{Q_{t,k}} (G_i(x,\tau) \, \epsilon^{a_i}) \, \epsilon^{b_i} \, dxd\tau \, .
\end{align*}
Use of $\epsilon = w_k + k$ on $\Omega_k$ implies that
\begin{align*}
\int_{Q_{t,k}} (G_i(x,\tau) \, \epsilon^{a_i}) \, \epsilon^{b_i} \, dxd\tau \leq 2^{b_i} \, \int_{Q_{t}} (G_i \, \epsilon^{a_i}) \,  w_k^{b_i} \, dxd\tau + (2k)^{b_i} \, \int_{Q_{t,k}} G_i \, \epsilon^{a_i} \, dxd\tau \, .
\end{align*}
Next, we assume that $b_i < r$ and we let $0< 1/s < 1-b_i/r$. Then, H\"older's inequality implies that
\begin{align*}
\int_{Q_{t,k}} (G_i \, \epsilon^{a_i})  \, \epsilon^{b_i} \, dxd\tau \leq & 2^{b_i} \, \|w_k\|_{L^r(Q_t)}^{b_i} \, \|G_i\, \epsilon^{a_i}\|_{L^{s}(Q_t)} \, |Q_{t,k}|^{1-\frac{1}{s}-\frac{ b_i}{r}} \\
& + (2k)^{b_i} \, \|G_i \, \epsilon^{a_i}\|_{L^{s}(Q_t)} \, |Q_{t,k}|^{1-\frac{1}{s}} \, .
\end{align*}
We further restrict the choice of $b_i$ and $r$ by the condition
\begin{align}\label{stronger}
 b_i < \frac{3r}{5} \, .
\end{align}
By means of Young's inequality $ x \, y < \delta \, x^{p} + c(\delta,p) \, y^{p/(p-1)}$, which we apply with $\delta > 0$ arbitrary, $ p = \xi_i := 3r/(5b_i)$ we get, with $\xi_i^{\prime} = \xi_i/(\xi_i-1)$,
\begin{align}\label{resultG}
\int_{Q_{t,k}} (G \, \epsilon^{a_i})  \, \epsilon^{b_i} \, dxd\tau \leq &\delta \, \|w_k\|_{L^r(Q_t)}^{\frac{3r}{5}} + c_{\delta} \,  \|G \, \epsilon^{a_i}\|_{L^{s}(Q_t)}^{\xi^{\prime}_i} \, |Q_k|^{\xi^{\prime}_i \, (1-\frac{1}{s}-\frac{ b_i}{r})} \nonumber\\
 & + c \, \, k^{b_i} \, \|G_i \, \epsilon^{a_i}\|_{L^{s}(Q_t)} \, |Q_{k,t}|^{1-\frac{1}{s}} \, .
\end{align}
We introduce numbers
\begin{align*}
 z_i := (1-\frac{1}{s}-\frac{b_i}{r}) \, \xi^{\prime}_i \, \frac{5}{3r}= (1-\frac{1}{s}-\frac{b_i}{r})\, \frac{5}{3r-5b_i} \, , \quad \tilde{z} = (1-\frac{1}{s}) \, \frac{5}{3r} \, , \quad
 \omega_i := \frac{5b_i}{3r} \, ,
\end{align*}
and we attain
\begin{align}\label{Gresult} 
  \int_{Q_{t,k}} G_i \, \epsilon^{\alpha_i} \, dxd\tau\leq & \delta \, \|w_k\|_{L^r(Q_t)}^{\frac{3r}{5}} \\
  & + 
  c_{\delta} \,  \|G_i \, \epsilon^{a_i}\|_{L^{s}(Q_t)}^{\xi^{\prime}_i} \, |Q_k|^{\frac{3r}{5} \, z_i}
 + c_0 \, \, k^{\frac{3r}{5} \, \omega_i} \, \|G_i \, \epsilon^{a_i}\|_{L^{s}(Q_t)} \, |Q_{k,t}|^{\frac{3r}{5} \, \tilde{z}}\, .\nonumber
  \end{align}
  We recall \eqref{boundgen3} and choose $\delta$ suitably small. Since $5\xi_i^{\prime}/(3r) = 5/(3r-5b_i)$, we obtain that
\begin{align*}
  \|w_k\|_{L^r(Q_t)} \leq C(\bar{\tau}) \, \left(1+\sum_{i=0}^3 \, \|G_i \, \epsilon^{a_i}\|_{L^{s}(Q_t)}^{\frac{5}{3r-5b_i}}\right) \,  \sum_{i=0}^3 \Big(|Q_{k,t}|^{z_i} + k^{\omega_i} \, |Q_{t,k}|^{\tilde{z}} \Big) \, .
  \end{align*}
  It also follows that
  \begin{align*}
    \|w_k\|_{L^1(Q_t)} \leq C(\bar{\tau}) \,\left(1+\sum_{i=0}^3 \, \|G_i \, \epsilon^{a_i}\|_{L^{s}(Q_t)}^{\frac{5}{3r-5b_i}}\right) \,  \sum_{i=0}^3 \Big(|Q_{t,k}|^{z_i+1-\frac{1}{r}} +k^{\omega_i} \,|Q_{t,k}|^{\tilde{z}+1-\frac{1}{r}}\Big) \, .
  \end{align*}
  Under the conditions
  \begin{align}\label{condiclear}
  z_i, \, \tilde{z} > \frac{1}{r} \quad \text{ and } \quad \omega_i \leq 1 + \tilde{z} - \frac{1}{r} \qquad \text{ for } i = 0,\ldots,3 \, , 
  \end{align}
  the Lemma \ref{lemboundmax} with $n = 3$, $\sigma_i := z_i+1-1/r$, $y := \tilde{z}+1-1/r$, and $A_i :=  1 + \sum_j \|G_j \, \epsilon^{a_j}\|_{L^s(Q_t)}^{\frac{5}{3r-5b_j}}=: B_i$ gives the bound
  \begin{align}\label{boundmax}
   \|\epsilon\|_{L^{\infty}(Q_t)} \leq k_1 + C\Big(\bar{\tau}, \, \Omega, \, k_1, \, \sum_{i=0}^3 \|G_i \, \epsilon^{a_i}\|_{L^s(Q_t)}^{\frac{5}{3r-5b_i}}, \, \|\epsilon\|_{L^1(Q_t)}\Big) \, ,
  \end{align}
  with $k_1 := \max\{k_0, \, \bar{k}(\bar{\tau}), \, 2 \, \|\epsilon\|_{L^{1,\infty}(Q_t)}/|\Omega|\} $ and a certain constant $C$.

  Verifying the conditions \eqref{condiclear} is elementary. They reduce to $s > 5/2$ and $b_i < 3r/5+2/5-1/s$, but the latter condition is weaker than $b_i < 3r/5$ which we already assumed in \eqref{stronger}. Thus, in order to obtain the bound as claimed in the Lemma, it is sufficient to assume that
  \begin{align*}
\alpha_i - a_i = b_i < \frac{3r}{5}  \quad \text{ and } \quad G_i \, \epsilon^{a_i} \in L^{s}(Q_T) \quad  \text{ for a } s > 5/2 \, .
  \end{align*}
In order to satisfy the latter condition, we choose appropriate numbers $0 \leq a_i < 1/s$. This is possible if
\begin{align}\label{localrestrict}
 \alpha_i < \frac{3r}{5} + \frac{1}{s} \, .
\end{align}
In this case, we can estimate
\begin{align*}
& \int_{\Omega} (G_i(x,t))^{s} \, \epsilon^{a_i \, s}(x,t) \, dx \leq \|G_i(\cdot,t)\|_{L^{\frac{s}{1-a_is}}(\Omega)}^s \, \|\epsilon(\cdot, \, t)\|^{a_is}_{L^{1}(\Omega)} \,\\
 \text{ yielding } &\\
& \|G_i \, \epsilon^{a_i}\|_{L^{s}(Q_t)} \leq \|\epsilon\|_{L^{1,\infty}(Q_t)}^{a_i} \, \|G_i\|_{L^{s/(1-a_is),s}(Q_t)} \, . 
\end{align*}
Suppose that $G_i \in L^{s_i,s}(Q_t)$ with $s_i \geq s > 5/2$, then we define $a_i := (s_i-s)/(s \, s_i) $. We can now choose any 
\begin{align}\label{exponeslats}
 \alpha_i < \frac{3r}{5} + \frac{1}{s} - \frac{1}{s_i} \, ,
\end{align}
a restriction which is stronger than \eqref{localrestrict}. Now \eqref{boundmax} implies that
\begin{align}\label{boundmax2}
   \|\epsilon\|_{L^{\infty}(Q_t)} \leq k_1 + C\Big(\bar{\tau}, \, \Omega, \, k_1, \, \sum_{i=0}^3 \left(\|\epsilon\|_{L^{1,\infty}(Q_t)}^{\frac{1}{s}-\frac{1}{s_i}} \,  \|G_i\|_{L^{s_i,s}(Q_t)}\right)^{\frac{5}{3r-5b_i}}, \, \|\epsilon\|_{L^1(Q_t)}\Big) \, .
  \end{align}
  Here $s_i > s > 5/2$ are arbitrary numbers. Finally, it remains to choose $s$ and $s_i$ in order that we are able to estimate the norms of the functions $G_i$ by the state-space norm.

  At this point, we recall the definitions \eqref{lesGs}.
Commencing with $G_1 = |\varrho_x|^2 $ we can at best rely on the bound $ \|G_1\|_{L^{p/2,\infty}(Q_t)}^2 \ \leq c \, \, \|\varrho\|_{W^{1}_{p,\infty}(Q_t)}^2$. Hence for all $\frac{5}{2} < s < s_i \leq \frac{p}{2}$, we get a bound for $G_1$ in $L^{s_i,s}(Q_t)$. This means that we have to restrict to $p > 5$,  and \eqref{exponeslats} restritcs $\beta_1$ via
\begin{align*}
 \beta_1 < \frac{2}{5}  \, \Big(2+3\beta+ 1 - \frac{5}{p}\Big) \, ,
\end{align*}
Next with $G_2 = |v_x|$, we have a bound $\|v_x\|_{L^{\infty,p}(Q_t)} \leq c_0 \, \|v\|_{W^{2,1}_p(Q_t)}$, hence we can choose $p \geq s > 5/2$ arbitrary and $s_i = + \infty$. Thus, \eqref{exponeslats} restricts the exponent $\alpha_2$ as
  \begin{align*}
  \alpha_2 = \max\{2, \, 1+\beta_2\}< \frac{3r}{5} + \frac{2}{5} = \frac{2}{5} \, (3+3\beta) \quad \text{ arbitrary.}
  \end{align*}
  Since $\beta \geq 1$, this condition reduces to $\beta_2 < \frac{1}{5} \, (6\beta-1+2)$.
  For $G_3 = |v_x|^2$ and $G_0 = |\bar{q}_x|^2$, we have the same parabolic regularity. Due to the Sobolev embedding, we have
  \begin{align*}
  \|G_i\|_{L^{z,\infty}(Q_t)} + \|G_i\|_{L^{\infty,p/2}(Q_t)} \leq c_0 \, (\|q\|_{W^{2,1}_p(Q_t)} + \|v\|_{W^{2,1}_p(Q_t)})^2 \, \quad z := \frac{3p}{2 \, (5-p)^+} \, .
  \end{align*}
  Thus, since we already restrict to $p > 5$, we can choose $s > 5/2$ arbitrary and $s_0, \, s_2 = +\infty$. Hence $
  \max\{\beta_0, \, \beta_3+1\} <
\frac{2}{5} \, (3+3\beta)$.
\end{proof}
\begin{rem}
In the case that there is a heat flux on the boundary, the weak form reads
\begin{align*}
 \int_{Q_t} \{\partial_t \epsilon \, \phi  - (\epsilon \, v + J^{\rm h}) \cdot \nabla \phi\} \, dxd\tau = & \int_{Q_t} (-p \, \mathbb{I} + \mathbb{S}) \, : \,  \nabla v \, \phi \, dxd\tau + \int_{S_t} (r^{\Gamma}_{\rm h} - J^{\Gamma}_{\rm h}) \, \phi \, dS_xd\tau\, .
\end{align*}
Here it is necessary to derive also growth conditions for the temperature-dependence of $r^{\Gamma}$. A not seldom, convenient case occurs if $T \mapsto r^{\Gamma}(x,t, \, T, \, \rho)$ is monotone decreasing. Then, $\varrho u \mapsto r^{\Gamma}(x,t , \ \hat{T}(\rho, \, \varrho u), \, \rho)$ is also monontone decreasing and there are no growth conditions coming from this term. For instance, the natural cooling condition $r^{\Gamma} = T^{\text{ext}} - T$ with the outer temperature $T^{\text{ext}}$ is of this form. As to $J_{\rm h}^{\Gamma}$, it is sufficient to assume that it possesses an extension of class $W^{1,0}_r(Q_{\bar{\tau}})$ with $r > 5$.
\end{rem}
\begin{rem}\label{bnotzero}
We considered the case $b^1 = \ldots = b^N$. Otherwise $\tilde{b} \neq 0$ and two differences have to be studied.

1. The heat flux gets, in comparison to \eqref{heatfluxbase}, an additional term
  \begin{align*}
 J^{\rm h} = \kappa(T, \, \rho) \, \nabla \frac{1}{T} - \sum_{k=1}^{N-1} \xi^k\cdot l(T, \, \rho) \, (\nabla q_k + q_N \, \tilde{b}^k(x,t))   \, .
\end{align*}
In the right-hand of 
\eqref{boundgen}, there is an additional term
$\int_{Q_{t,k}} \epsilon^{\beta_4} \, |\tilde{b}(x,t)|^2 \, dxd\tau \, $, where we assumed that $\beta_4$ is such that
\begin{align}\label{choceweak}
 \frac{d_0}{\kappa} \, q_N^2 \, l^2(\varrho, \, q) \precsim \epsilon^{\beta_4} \, .
\end{align}
Since $\tilde{b} \in L^{\infty,p}(Q_{\bar{\tau}})$ by assumption, we in fact do not need to reinforce the assumptions \eqref{growth}. The term $q_N^2$ being favourable for $T \rightarrow + \infty$ (meaning that $q_N \rightarrow 0-$), we have $\beta_4 \leq \beta_0$ and finish the proof as in Proposition \ref{TECHON}.

2. If $\tilde{b} \neq 0$, the right-hand side of the equation \eqref{energy} gets the additional term $J \, : \, b$. On the right-hand side of 
\eqref{boundgen}, there is an additional term
$\int_{Q_{t,k}} \epsilon^{\beta_5+1} \, (|\nabla q|+|q_N| \, |\tilde{b}(x,t)|) \, |\tilde{b}(x,t)| \, dxd\tau $ where $\beta_5$ is the growth exponent for the thermodynamic diffusivities $M_{ij}$:
\begin{align*}
|M_{ij}| \precsim \epsilon^{\beta_5} \quad \text{ for } \quad i,j = 1,\ldots,N \, .
\end{align*}
If $\beta_5+1$ is subject to the same restrictions as $\beta_0$, the term $\int_{Q_{t,k}} \epsilon^{\beta_5+1} \, |q_N| \, |\tilde{b}(x,t)|^2 \, dxd\tau $ can be treated as just shown under 1. of the present remark. For $\tilde{b} \in L^{\infty,p}(Q_{\bar{\tau}})$, the product $|\tilde{b}| \, |\nabla q|$ can be estimated in the same way as $|\nabla q|^2$ and $|\nabla v|^2$ in the proof of Proposition \ref{TECHON}. Hence, $\beta_5$ is also subject to the same restrictions as the exponent $\beta_3$. Summarising, the entries of $M$ are essentially subject to the same growth restrictions as the viscosity coefficients $\eta$ and $\lambda$.
\end{rem}
With the boundedness of $\varrho u$, we also obtain the maximum principles for $T$ and $\rho$.
\begin{coro}
Assume that the function $\epsilon$ in \eqref{varrhoubasic} is such that, for all $0 < m \leq M < +\infty$
\begin{align*}
 \liminf_{T \rightarrow + \infty} \inf_{m \leq |\rho| \leq M}\epsilon(T, \, \rho) = + \infty \, .
\end{align*}
 Under the assumptions of Proposition \ref{TECHON}, we then also have
 \begin{align*}
  \limsup_{t \rightarrow \bar{\tau}-} \|T\|_{L^{\infty}(Q_t)} < + \infty, \quad \liminf_{t \rightarrow \bar{\tau}-} \inf_{i=1,\ldots,N, \, x \in \Omega} \rho_i(x, \, t) > 0 \, .
 \end{align*}
\end{coro}
\begin{proof}
By assumption, the total mass density $\varrho$ associated with the solution remains inside of the interval $[m(\bar{\tau}), \, M(\bar{\tau})]$. Thus, the bound obtained in Proposition \ref{TECHON} implies that $\|T\|_{L^{\infty}(Q_{\bar{\tau}})}$ remains bounded.

Hence, we also have that $\sup_{Q_{\bar{\tau}}} q_N < 0$. Next we recall that $\rho = \mathscr{R}(\varrho, \, q)$, and since $(\varrho,  \, q)$ remains in the strict interior of the domain $\mathbb{R}_+ \times \mathcal{H}^N_-$, the mass densities cannot tend to zero.
\end{proof}
This achieves to prove the general form of the maximum principle for the system \eqref{mass}, \eqref{energy}. For a more concrete illustration of the condition \eqref{growth} in Proposition \ref{TECHON}, we refer to the Theorem \ref{MAIN2}.

The first task is to introduce a special construction of the entropy functional.

\section{A constitutive model for  ideal mixtures}\label{IDMIX}

With the concept of an ideal mixture, we refer to a concrete form of the chemical potentials
\begin{align}\label{muiideal}
\mu_i = \hat{\mu}_i(T, \, p, \, x_i) = g_i(T, \, p) + R \, T \, \ln x_i \, .
\end{align}
The structure \eqref{muiideal} \emph{implies that} the mixture is volume-additive and enthalpy-additive. Let us however remark that the predicates ''simple mixture'' or ''ideal mixture'' do not seem to possess a completely univoque meaning in the literature. For instance, a more restrictive characterisation of the concept of an ideal mixture is to be found in \cite{brdicka}, Ch.\ 4.1. \footnote{\emph{In an ideal mixture the same forces are acting in average between the molecules of the different materials as between the molecules of the pure components (...). A series of extensive properties of ideal mixtures are obtained in simple additive way from the corresponding properties of the pure components.} See \cite{brdicka}, page 317.}


For notational simplicity, the gas constant $R$ shall from now be normalised to one.
In \eqref{muiideal}, $g_i(T, \, p)$ denotes the Gibbs free energy density of the constituent ${\rm A}_i$.  
In particular, with $\hat{\rho}_i(T, \, p)$ denoting the mass density of the constituent ${\rm A}_i$ as a function of temperature and pressure, we have the identity $\partial_p g_i(T, \, p) = 1/\hat{\rho}_i(T, \, p)$. If full data are available for all ${\rm A}_i$, we can therefore construct the function $g_i$ from the following data concerning the constituent ${\rm A}_i$:
\begin{itemize}
 \item The equation of state $\hat{\rho}_i(T, \, p)$ giving the density or specific volume as function of temperature and pressure;
 \item The heat capacity at reference pressure $c^{i0}_p(T) := c^i_p(T, \, p^0)$ as a function of temperature only;
 \item The (constant) enthalpy $h^{i0}$ and entropy $s^{i0}$ at reference temperature $T^0$ and pressure $p^0$.
\end{itemize}
Then we have the formula (see \cite{bothedreyerdruet}, equation (32))
\begin{align}\label{thermocon}
g_{i}(T, \, p) = \int_{p^0}^p \frac{1}{\hat{\rho}_i(T, \, p^{\prime})} \, dp^{\prime} - \int_{T^0}^T\int_{T^0}^{\theta} \frac{c^{i0}_p(\theta^{\prime})}{\theta^{\prime}} \, d\theta^{\prime}d\theta - s^{i0} \, T + h^{i0} \, .
\end{align}
In the stable fluid phase, the data in the latter formula are restricted by several conditions of thermodynamic consistency. Denoting by $\upsilon^i$ the specific volume and $c_p^i, \, c_{\upsilon}^i$ the heat capacities of the constituent ${\rm A}_i$, we must have
\begin{align}
\begin{split}\label{cvda}
\partial^2_p g_i  = \partial_p \upsilon_i < 0, \quad \partial^2_T g_i = - \frac{c_{p}^i}{T} < 0 \, ,\\
\partial^2_T g_i - \frac{(\partial^2_{T,p} g_i)^2}{\partial^2_p g_i} = -\frac{c^i_{\upsilon}}{T}  < 0 \, .
\end{split}
\end{align}

In this section, we at first investigate the possibility to construct the constitutive function $h$ from the data in \eqref{muiideal} with general $g_1,\ldots, g_N$. At second, we shall set up a simple particular constitutive model.

\subsection{General results}

We start with giving the form of several thermodynamic functions implied by the definition \eqref{muiideal}.
\begin{lemma}\label{Densities}
Assume that for $i=1,\ldots,N$ the function $g_i$ belongs to $C^{2}(]0, \, + \infty[^2)$, is strictly concave, and it satisfies $\partial_pg_i > 0$. Assume moreover that, for all $T > 0$,
\begin{align}\label{gitozero}
\lim_{p \rightarrow 0+} \,  \max_{i = 1,\ldots,N} \partial_pg_i(T, \, p) =+ \infty \, , \quad
\lim_{p \rightarrow + \infty} \partial_p g_i(T, \, p) = 0 \text{ for } i = 1,\ldots,N \, .
\end{align}
If the chemical potentials are given by \eqref{muiideal}, then the Gibbs-Duhem equation \eqref{GIBBSDUHEMEULER} is satisfied iff $p = \tilde{p}(T, \, \rho)$, where the function $\tilde{p}$ of the main variables is defined as the root of the equation
\begin{align}\label{implicitpress}
 \sum_{i=1}^N \partial_pg_i(T, \, p) \, \rho_i = 1 \, \quad \Longleftrightarrow \quad \, p = \tilde{p}(T, \, \rho) \, .
\end{align}
Moreover, the free and internal energy densities and the entropy density are, up to a function of temperature only, uniquely determined as functions of the main variables via
\begin{align}\label{varrhopsispecial}
 \varrho\psi = & \sum_{i=1}^N g_i(T, \, \tilde{p}(T, \, \rho)) \, \rho_i - \tilde{p}(T, \, \rho) + T \, \sum_{i=1}^N (\rho_i/M_i) \, \ln \frac{\rho_i/M_i}{\sum_j (\rho_j/M_j)}  =: f(T, \, \rho)\, ,\\
 \label{varrhouspecial} \varrho u = & \sum_{i=1}^N (g_i(T,\, \tilde{p}(T,\rho)) - T \, \partial_Tg_i(T,\, \tilde{p}(T,\rho))) \, \rho_i - \tilde{p}(T, \, \rho) =: \epsilon(T, \, \rho) \, ,\\
 & \label{varrhosspecial}  - \varrho s =  \sum_{i=1}^N \partial_Tg_i(T, \, \tilde{p}(T, \, \rho)) \, \rho_i + \sum_{i=1}^N \frac{\rho_i}{M_i} \, \ln \frac{\rho_i/M_i}{\sum_j (\rho_j/M_j)}  =:  \tilde{h}(T, \, \rho) \, .
\end{align}
The heat capacity $c_{\upsilon} :=  \partial_T\epsilon/\varrho$ is strictly positive, and it obeys
\begin{align}\label{HCCAP}
  c_{\upsilon} = \tilde{c}_{\upsilon}(T,\rho) = - \frac{T}{\varrho} \, \left(\sum_{i=1}^N\partial^2_{T} g_i(T, \, \tilde{p}) \, \rho_i -\frac{\Big(\sum_{i=1}^N \partial^2_{T,p} g_i(T, \, \tilde{p}) \, \rho_i\Big)^2}{\sum_{i=1}^N \partial^2_{p} g_i(T, \, \tilde{p}) \, \rho_i} \right)
 \geq \inf_{i=1, \ldots,N} c^i_{\upsilon}(T, \, \tilde{p}) \, . 
\end{align}
\end{lemma}
\begin{proof}
Due to the definition $\varrho\psi = \varrho u -T \, \varrho s$ of the Helmholtz free energy, the Gibbs-Duhem equation \eqref{GIBBSDUHEMEULER} reads $p = - \varrho \psi + \sum_{i=1}^N \rho_i \, \mu_i$. We claim that this equation equivalently defines the partial mass densities as functions of temperature, pressure and the mole fractions. To see this, we write $\varrho\psi = f(T, \, \rho_1, \ldots,\rho_N)$ with a constitutive function $f$ of the main variables. With $\upsilon = 1/n$ being the molar volume, and $x_1, \ldots, x_N$ the mole fractions, we have $\rho_i = M_ix_i/\upsilon$. Hence, an equivalent expression of the Gibbs-Duhem equation is
\begin{align}\label{morle}
 p = - f\Big(T, \, \frac{1}{\upsilon} \, (M_1x_1, \ldots,M_Nx_N)\Big) + \frac{1}{\upsilon} \, \sum_{i=1}^N M_i \, x_i \, \hat{\mu}_i(T, \, p, \, x_i) \, .
\end{align}
The strict convexity of $f$ in the $\rho-$variables is a consequence of the basic definitions\footnote{Recall that $f$ is a function of $(T, \rho)$, while $h$ is a function of $(\rho,\varrho u)$. We have
\begin{align*}
 D^2_{\rho,\rho}f = \frac{1}{T} \, \Big(D^2_{\rho,\rho}h - \frac{D^2_{\rho,\varrho u}h \otimes D^2_{\rho,\varrho u}h}{\partial^2_{\varrho u} h }\Big) \, .
\end{align*}
}. This can now be used to show that \eqref{morle} defines $\upsilon$ as an implicit function $\hat{\upsilon}$ of $T, \, p, \, x$. 
Then, differentiating \eqref{morle} in $p$ yields $1 = \sum_{i=1}^N \rho_i \, \partial_p\hat{\mu}_i(T, \, p, \, x_i)$, which is nothing else but the left-hand of \eqref{implicitpress}.

Under the conditions \eqref{gitozero}, we have for all $T, \, \rho_1,\ldots,\rho_N>0$
\begin{align*}
\lim_{p \rightarrow +\infty} \sum_{i=1}^N \partial_p g_i(T, \, p) \, \rho_i = 0\quad  \text{ and } \quad \lim_{p \rightarrow 0+} \sum_{i=1}^N \partial_p g_i(T, \, p) \, \rho_i = +\infty \, .
\end{align*}
Owing to the strict concavity of $g_i$, we must also have $\partial^2_{p}g_i < 0$. Hence, for all $T, \, \rho_1,\ldots, \rho_N > 0$, the equation $\sum_i \partial_pg_i \, \rho_i = 1$ in \eqref{implicitpress} possesses exactly one root $p = \tilde{p}(T, \, \rho)$, achieving to justify \eqref{implicitpress}.

In the appendix of \cite{bothedreyerdruet}, we prove that the corresponding form of the Helmholtz free energy $\varrho\psi$ is, up to a function of temperature, given by \eqref{varrhopsispecial}.

Using the definition \eqref{implicitpress} of $\tilde{p}(T, \, \rho)$, we compute
\begin{align}\label{dpdT}
 \partial_T \tilde{p}(T, \, \rho) = - \frac{\sum_{i=1}^N \partial^2_{T,p} g_i(T, \, p) \, \rho_i}{\sum_{i=1}^N \partial^2_{p} g_i(T, \, p) \, \rho_i} \, .
\end{align}
We are now able to compute $\partial_T f$. Then, the identities $\varrho u = f - T \, \partial_T f$ and $\varrho s = -\partial_T f$ are used to derive the forms \eqref{varrhouspecial} and \eqref{varrhosspecial} of the internal energy and entropy densities. Using next the definition $\varrho  \, c_{\upsilon} = \partial_T \varrho u$ of the heat capacity, we obtain its representation \eqref{HCCAP} as a straightforward exercise. Finally we want to obtain a lower bound on $c_{\upsilon}$. With the mass fractions $y_i := \rho_i/\varrho$, we now can express
\begin{align*}
 c_{\upsilon} = - \frac{T}{\sum_{i=1}^N \partial^2_pg_i(T,p) \, y_i} \, \text{Det}\Big(\sum_{i=1}^N y_i \, D^2_{T,p} g_i(T,p)\Big) \, .
\end{align*}
The strict positivity of $c_{\upsilon}$ directly follows from the strict concavity of $(T,\, p) \mapsto g_i(T,\, p)$ for each $i$. Moreover, by the sub-additivity of the determinant on definite matrices
\begin{align*}
c_{\upsilon} \geq \frac{-T}{\sum_{i} \partial^2_p g_i(T,p) \, y_i} \, \sum_{i=1}^N y_i \, \text{Det}(D^2_{T,p}g_i(T,p)) = \frac{1}{\sum_{i} \partial^2_p g_i \, y_i} \, \sum_{i} y_i \, \partial^2_p g_i \, c_{\upsilon}^i \, ,
\end{align*}
from which it follows that $c_{\upsilon} \geq \min_i \, c^i_{\upsilon}$.
\end{proof}
In order to investigate further thermodynamic functions, we next introduce thresholds functions related to the definition \eqref{implicitpress} of the pressure variable.
\begin{lemma}\label{pminmax}
We adopt the assumptions of Lemma \ref{Densities}, and we moreover assume that 
$\lim_{p \rightarrow 0+} \partial_p g_i(T, \, p) = + \infty$ for all $i = 1,\ldots, N$ and all $T > 0$. For $T, \, \varrho > 0$ and $i=1,\ldots,N$, we define $p_i(T, \, \varrho)$ as the root of the equation
$ \partial_p g_i(T, \, p_i) = 1/\varrho$, and 
\begin{align*}
 p_{\min}(T, \, \varrho) := \min_{i=1,\ldots,N} p_i(T, \, \varrho), \quad p_{\max}(T, \, \varrho) := \max_{i=1,\ldots,N} p_i(T, \, \varrho) \, .
\end{align*}
Let $T, \, \rho_1, \ldots, \rho_N > 0$. It $\tilde{p}(T, \, \rho_1, \ldots, \rho_N)$ is defined via \eqref{implicitpress}, then with $\varrho = \sum_{i=1}^N \rho_i$ we have $p_{\min}(T, \, \varrho) \leq \tilde{p} \leq p_{\max}(T, \, \varrho)$.
\end{lemma}
\begin{proof}
In view of the asymptotic behaviour of $\partial_pg_i$ for $p \rightarrow \{0, \, +\infty\}$, the functions $p_i$, $p_{\min}$ and $p_{\max}$ are well-defined. If $\tilde{p}(T, \, \rho_1,\ldots,\rho_N)$ satisfies \eqref{implicitpress}, then the definition of $p_i$ implies that 
\begin{align*}
\min_{i} \partial_p g_i(T, \, \tilde{p}) \leq \frac{1}{\sum_{i=1}^N \rho_i} = \frac{1}{\varrho} = \partial_pg_i(T, \, p_i(T, \, \varrho)) \quad \text{ for } i = 1,\ldots,N \, .
\end{align*}
Thus, with a certain index $i_1$ such that $\partial_p g_{i_1}(T, \, \tilde{p}) = \min_i \partial_p g_i(T, \, \tilde{p})$, we obtain that
\begin{align*}
\partial_p g_{i_1}(T, \, \tilde{p}) \leq \partial_pg_{i_1}(T, \, p_{i_1}) \, .
\end{align*}
Using that $\partial_pg_{i_1}$ is strictly decreasing in the second argument yields $\tilde{p} \geq p_{i_1}(T, \, \varrho) \geq p_{\min}(T, \, \varrho)$. Similarly, we show that $\tilde{p} \leq p_{\max}(T, \, \varrho)$.
\end{proof}
%
%
%
In the next statement we show under which conditions for $g_1,\ldots,g_N$ we obtain a full thermodynamic model. With this notion, we mean that the entropy potential is a so-called function of Legendre type\footnote{A function $f$ defined in an open convex set $\mathcal{D} \subseteq \mathbb{R}^N$ is called \emph{of Legendre--type} if it is continuously differentiable and strictly convex in $\mathcal{D}$ and the gradient of $f$ blows up at every point of the boundary of $\mathcal{D}$ (see \cite{rockafellar}, Section 26).} on an open convex set $\mathcal{D} \subseteq \mathbb{R}^N_+ \times \mathbb{R}$, and that the Legendre transform $h^*$ is well-defined in the largest possible domain allowed by the constraint of temperature positivity.
\begin{prop}\label{yeswehave}
 We adopt the assumptions of Lemma \ref{Densities} and \ref{pminmax}. Moreover, we assume that there are  $\epsilon_{1}^{\min}, \ldots, \epsilon_N^{\min} \in \mathbb{R} \cup \{-\infty\}$ such that, for all $i = 1,\ldots, N$ and arbitrary $p_1, \, \varrho_1 > 0$
 \begin{align}\label{A11}
  \lim_{T\rightarrow 0+}  \partial_pg_i(T, \, p_1) = 0 \, , \quad \lim_{T, \, p \rightarrow 0+, \, p\geq p_{\min}(T, \, \varrho_1)} (g_i(T,p) - T \, \partial_T g_i(T,p)) = \epsilon^{\min}_i \, .
 \end{align}
 For $T_1 > 0$ arbitrary, assume that the heat capacities of the species, defined in \eqref{cvda}, satisfy
 \begin{align*}
 \inf_{i=1,\ldots,N, \, T \geq T_1, \, p >0} c^i_{\upsilon}(T,p) > 0 \, .
 \end{align*}
 We define an open, convex set $ \mathcal{D} := \{(\rho, \, \varrho u) \in \mathbb{R}^N_+ \times \mathbb{R} \, : \, \varrho u > \sum_{i=1}^N \epsilon^{\min}_i \, \rho_i\}$.
 Then, with $\epsilon$ defined in \eqref{varrhouspecial}, the root $T = \hat{T}(\rho_1, \ldots,\rho_N, \,\varrho u)$ of the equation 
 \begin{align}\label{implicittemp}
\epsilon(T, \, \rho_1, \ldots,\rho_N) = \varrho u \, 
\end{align}
is well-defined for all $(\varrho, \, \varrho u) \in \mathcal{D}$.  For $(\varrho, \, \varrho u) \in \mathcal{D}$, we introduce
$
h(\rho, \, \varrho u) :=  \tilde{h}(\hat{T}(\rho, \, \varrho u), \, \rho)$ with $\tilde{h}$ from \eqref{varrhosspecial}. Then, $h$ is strictly convex in $\mathcal{D}$ and of class $C^3(\mathcal{D})$. Suppose additionally that for all $T > 0$
\begin{align}\label{jeveuxsurj}
 \lim_{p \rightarrow +\infty} g_i(T, \, p) = + \infty \text{ for } i=1,\ldots, N \, , \quad \lim_{p \rightarrow 0+} \min_{i=1,\ldots,N} g_i(T, \, p) = - \infty \, . 
\end{align}
Then the image of $\nabla_w h$ on $\mathcal{D}$ is $ \mathbb{R}^N \times \mathbb{R}_- =:  \mathcal{D}^*$. Finally, suppose that for all $T_2 > T_1 > 0$ and $\varrho_1 > 0$
\begin{align}\begin{split}\label{jeveuxblow}
 & \lim_{T \rightarrow + \infty} \min_{i=1, \ldots, N} \sup_{\varrho \leq \varrho_1} \frac{g_i(T, \, p_{\max}(T,\varrho))}{T} = -\infty \, , \\
& \lim_{p \rightarrow +\infty} \max_{i=1,\ldots,N} \max_{T_2 \geq T \geq T_1} \partial_pg_i(T, \, p) = 0 \, ,\\ 
& \lim_{p \rightarrow 0+} \min_{i=1,\ldots,N} \,\max_{T_2 \geq T \geq T_1} g_i(T, \, p) = - \infty  \, .
\end{split}
\end{align}
Then $h$ is essentially smooth on $\mathcal{D}$.
\end{prop}
\begin{proof}
For fixed $\rho_1, \ldots, \rho_N > 0$, and $T \geq T_1$, we have\begin{align*}
\partial_T \epsilon(T, \, \rho) = \frac{1}{\varrho} \,
\tilde{c}_{\upsilon}(T, \, \rho) \geq \frac{1}{\varrho} \inf_{i=1,\ldots,N, \, T\geq T_1, \, p > 0} c_{\upsilon}^i(T, \, p) > 0 \, .                                                                                                                                                                                                                                                                                                                                                                                   \end{align*}
Hence $\lim_{T \rightarrow +\infty} \epsilon(T, \, \rho) = +\infty$. Due to \eqref{implicitpress}, we have $1/\varrho \leq  \max_{i=1, \ldots,N} \partial_pg_i(T, \, \tilde{p})$.
We next fix $\rho \in \mathbb{R}^N_+$, while letting $T \rightarrow 0+$, and want to show that $\limsup_{T \rightarrow 0+} \tilde{p}(T, \, \rho) = 0$. Suppose that it is not the case, that is $\limsup_{T\rightarrow 0+} \tilde{p}(T, \, \rho) =: p_1 > 0$. Then, \eqref{A11}$_1$ implies that
\begin{align}\label{pressebel}
\frac{1}{\varrho} \leq \limsup_{T \rightarrow 0+} \max_{i=1,\ldots,N} \partial_pg_i(T, \, p_1) = 0 \, ,
\end{align}
a clear contradiction. It then also follows by means of \eqref{A11}$_2$ that
\begin{align}\label{borde}
 \lim_{T \rightarrow 0+} \epsilon(T, \, \rho) = \sum_{i=1}^N  \lim_{T,p \rightarrow 0+, \, p\geq p_{\min}(T,\varrho)} (g_i(T,p) - T \, \partial_T g_i(T, p)) \, \rho_i = \sum_{i=1}^N \epsilon^{\min}_i \, \rho_i \, .
\end{align}
This shows that for all $\varrho u \in ]\sum_i \epsilon^{\min}_i \, \rho_i, \, + \infty[$, the equation \eqref{implicittemp} possesses a unique solution $T = \hat{T}(\rho, \, \varrho u)$. The derivatives of $\hat{T}$ in particular satisfy $ \partial_{\varrho u} \hat{T}(\rho, \, \varrho u) = 1/(\varrho \, c_{\upsilon})$.
 
For $(\rho, \, \varrho u) \in \mathcal{D}$, we introduce $h(\rho, \, \varrho u) :=  \tilde{h}(\hat{T}(\rho, \, \varrho u), \, \varrho)$, where $\tilde{h}$ is the $(T, \rho)$ representation of the negative entropy (See \eqref{varrhosspecial}). For the derivatives, we obtain that
\begin{align}\label{nablawh}
 \partial_{\rho_i}h = \frac{\mu_i(\hat{T}, \, \rho)}{\hat{T}} = g_i(\hat{T}, \, \tilde{p}(\hat{T}, \,\rho)) + \frac{1}{M_i} \, \ln \frac{\rho_i/M_i}{\sum_j (\rho_j/M_j)} \, , \quad 
 \partial_{\varrho u} h = - \frac{1}{\hat{T}} \, .
\end{align}
For the second derivatives, we obtain that
\begin{align}\begin{split}\label{d2wh}
 \partial^2_{\rho_i,\rho_j} h = \frac{1}{\hat{T}} \, \partial^2_{\rho_i,\rho_j}f(\hat{T}, \, \rho)+ \frac{\varrho \, c_{\upsilon}}{\hat{T}^2} \, \partial_{\rho_i}\hat{T} \, \partial_{\rho_j} \hat{T}\, ,\\
 \partial^2_{\varrho u, \rho_i} h = \frac{1}{\hat{T}^2} \, \partial_{\rho_i} \hat{T}\, , \quad 
 \partial^{2}_{\varrho u} h = \frac{1}{\varrho \, c_{\upsilon} \,\hat{T}^2} \, .
 \end{split}
\end{align}
Here $f$ is the $(T,\rho)$ representation of the free energy (See \eqref{varrhopsispecial}). Since $\varrho \,  c_{\upsilon} > 0$, the function $h$ thus possesses a strictly positive definite Hessian iff $\rho \mapsto f(T, \, \rho)$ possesses a strictly positive definite Hessian for all $T$. Now, the Hessian of the free energy at fixed temperature can again be computed from the representation of $\varrho\psi$, to the result
\begin{align*}
 \partial^2_{\rho_i,\rho_j} f(T, \, \rho) = -\frac{1}{\sum_{k} \partial^2_{p} g_k(T,\tilde{p}) \, \rho_k} \, \partial_pg_i(T, \, \tilde{p}) \, \partial_pg_j(T, \, \tilde{p}) + \frac{1}{M_i\,M_j} \, \Big(\frac{M_i \, \delta_{ij}}{\rho_i} - \frac{1}{\sum_k (\rho_k/M_k)}\Big) \, . 
\end{align*}
The latter matrix is strictly positive definite on all states.

 Next we want to show that for all $w^* \in \mathbb{R}^N \times \mathbb{R}_-$, the equations $\nabla_w h(\rho, \, \varrho u) = w^*$ possess a unique solution. We first introduce $ T = -1/w_{N+1}^{*}$. 
As a next step, we determine a solution $(p, \, x_1, \ldots, x_N)$, with the constraint $\sum_{i=1}^N x_i = 1$, for the equations
\begin{align}\label{xxx}
& w_i^{*} = \frac{\mu_i}{T} = \frac{g_i(T, \, p)}{T} + \frac{1}{M_i} \, \ln x_i \quad \text{ for } i =1,\ldots,N \nonumber\\
 \text{ yielding } \quad & & \nonumber\\
& x_i = \exp\Big(  M_i \, (w_i^* - \frac{g_i(T,\, p)}{T})\Big) \, .
\end{align}
For the sake of commodity, we shall from now also use the notation $ w^* = (\bar{w}^{*}, \, w_{N+1})$. To realise the constraint that $x_i$ ( which stands for the mole fractions) sum up to one, we find the root $p$ of the equation
\begin{align}\label{pressureimplicit}
 1 = \sum_{i=1}^N \exp\Big(  M_i \, (w^*_i - \frac{g_i(T, \, p)}{T})\Big) \, .
\end{align}
Owing to the conditions \eqref{jeveuxsurj}, this equation defines $p$ implicitely \emph{as a function of the variables} $T$ and $(w_{1}^* ,\ldots, w_{N}^*) = \bar{w}^{*} $. We denote this function by $\hat{p}(T, \, \bar{w}^{*})$.
%
We also define
\begin{align}\label{cehatrhoci}
&  \hat{\rho}_i(T, \, \bar{w}^{*}):= M_i \, \frac{\exp\Big(  M_i \, (w_i^* - \frac{g_i(T, \, \hat{p}(T, \, \bar{w}^{*}))}{T})\Big)}{\sum_{j=1}^N M_j \, \partial_pg_j(T, \, \hat{p}(T,\bar{w}^{*})) \, \exp\Big(  M_j \, (w_j^{\prime} - \frac{g_j(T, \, \hat{p}(T,\bar{w}^{*}))}{T})\Big)}
 \, 
\end{align}
and we choose
\begin{align*}
 \rho = \hat{\rho}\Big(-\frac{1}{w_{N+1}^*}, \, \bar{w}^{*}\Big)\, , \qquad \varrho u := \epsilon\Big(-\frac{1}{w^*_{N+1}}, \, \hat{\rho}\Big(-\frac{1}{w^*_{N+1}}, \, \bar{w}^{*}\Big)\Big) \, .
\end{align*}
We claim to have determined a solution $(\rho, \, \varrho u)$ to $\nabla_w h(\rho, \, \varrho u) = w^*$, which can be verified easily. This solution must be unique, since $h$ is a strictly convex function on a convex domain.

Finally we want to show that $h$ is an essentially smooth function on its domain $\mathcal{D}$. This means showing that $|\nabla_w h(\rho, \, \varrho u)| \rightarrow + \infty$ whenever $(\rho, \, \varrho u)$ approach a boundary point of $\mathcal{D}$. 
Consider arbitrary a sequence $w^m = (\rho^m, \, \varrho u^m) \subset \mathcal{D}$ such that $w^m \rightarrow w^0 = (\rho^0, \, \varrho u^0) \in \partial \mathcal{D}$. (In particular, this implies that $\sup_m |w^m| < +\infty$). The boundary of $\mathcal{D}$ consists of the surfaces
\begin{align*}
 S^1 = \{(\rho, \,\varrho u) \, : \, \rho \in \mathbb{R}^N_+, \, \varrho u = \sum_i \epsilon^{\min}_i \, \rho_i\} \,,  \quad S^2 = \{(\rho, \, \varrho u) \, : \, \inf_i \rho_i = 0, \, \varrho u > \sum_i \epsilon^{\min}_i \, \rho_i\} \, ,
\end{align*}
where $S^1$ is empty if $\inf_i \epsilon^{\min}_i = - \infty$. We define $1/T^m := -\partial_{\varrho u} h(\rho^m \, \varrho u^m)$. If $T^m \rightarrow 0$, then obviously $|\nabla_w h(\rho^m \, \varrho u^m)| \geq 1/T^m \rightarrow + \infty$, and we are done already.
Hence, we assume that $\inf_{m \in \mathbb{N}} T^m > 0$.

Since the sequence $\rho^m$ converges, it is bounded. As shown in \eqref{borde}, we can have $\varrho u^m - \sum_i \epsilon^{\min}_i \,  \rho_i^m \rightarrow 0$ only if $T^m \rightarrow 0+$ and we are done again. Hence, it remains to consider the case $w^m$ approaches a point on $S^2$. This means that $\inf_i \rho_i^m \rightarrow 0$.

Suppose now that $T^m \rightarrow + \infty$. Due to the definition of $p_{\max}$, we have
\begin{align*}
\frac{\mu_i^m}{T^m} \leq \frac{g_i(T^m, \, p^m)}{T^m}  \leq \frac{g_i(T^m, \, p_{\max}(T^m, \, \varrho^m))}{T^m} \leq  \sup_{\varrho \leq \sup_m \varrho^m}  \frac{g_i(T^m, \, p_{\max}(T^m, \, \varrho))}{T^m} \, .
\end{align*}
Hence, the condition \eqref{jeveuxblow}$_1$ implies that 
\begin{align*}
\inf_i \frac{\mu_i^m}{T^m} \leq \inf_i \sup_{\varrho \leq \sup_m \varrho^m}  \frac{g_i(T^m, \, p_{\max}(T^m, \, \varrho))}{T^m} \rightarrow - \infty \, .
\end{align*}
As $\mu_i^m/T^m = \partial_{\rho_i}h(\rho^m, \, \varrho u^m)$, this again would show that $|\nabla_w h(\rho^m, \, \varrho u^m)| \rightarrow +\infty$.

Thus, it remains only to consider the case where $\inf_i \rho_i^m \rightarrow 0$ while the sequence $T^m$ is uniformly bounded from below and above.
%
Then, we invoke \eqref{implicitpress} to see that
\begin{align}\label{pressebelvrai}
 \max_{i=1,\ldots,N, \, \sup_m T^m \geq T \geq \inf_m T^m} \partial_pg_i(T, \, p^m) \geq \frac{1}{\varrho^m}  \geq \frac{1}{\sup_m \varrho^m} > 0\, ,
\end{align}
and, with $p^m := \tilde{p}(T^m, \, \rho^m)$, the assumption \eqref{jeveuxblow}$_2$ implies that $\sup_m p^m < + \infty$. Thus,  $g_i(T^m, \, p^m)$ is uniformly bounded above for $i = 1,\ldots,N$. We then distinguish two cases. The first case is that the associated total mass densities $\varrho^m := \sum_{i=1}^N \rho^m_i $ are bounded away from zero. Since $x_i^m \leq M_{\max} \, \rho_i^m/(\varrho_m \, M_{\min})$, with the maximal/minimal molar masses, we must have $\inf_{i = 1,\ldots,N} x_i^m \rightarrow 0$, and we see that
\begin{align*}
\frac{1}{T^m} \,  \inf_i \mu_i^m \leq \inf_i \ln x_i^m + \sup_{i} \frac{g_i(T^m, \, p^m)}{T^m}  \leq \inf_i \ln x_i^m + C_0 \rightarrow - \infty \, .
\end{align*}
The second case is $\varrho^m \rightarrow 0$. Then, it is possible that all $x_i^m$ remain bounded from below. However, \eqref{pressebelvrai} shows that $p^m \rightarrow 0+$, and \eqref{jeveuxblow}$_3$ yields
\begin{align*}
\min_{i=1,\ldots,N} \sup_{\sup_m T^m \geq T\geq \inf_m T^m} g_i(T, \, p^m) \rightarrow  - \infty \, .
\end{align*}
We find again that $|\nabla_w h(\rho^m, \, \varrho u^m)| \rightarrow +\infty$, finishing the proof.
\end{proof}

The next step is to propose a concrete Ansatz realising all conditions of Proposition \ref{yeswehave}.

\subsection{Particular constitutive model}

We employ the equation \eqref{thermocon}, an we assume that the specific volume of each ${\rm A}_i$ possesses sublinear growth/decay in temperature/pressure with certain powers. This reflects an equation of state of Tait--type. For the heat capacity $c^{i0}_p(T)$, we assume some growth of monomial type in temperature. With this motivation\footnote{The canonical way is obviously to look up such data in tables, and then trying to fit a function with the available parameters. This is however not necessary to perform this step for a qualitative analysis like the present investigation.}, we consider the particular choice
\begin{align}\label{smallgparticular}
 g_i(T, \, p) = \frac{p^0}{\beta_i \, \rho_i^{\rm R}} \, \left(\frac{T}{T^0}\right)^{\alpha_i} \, \left(\frac{p}{p^0}\right)^{\beta_i} - \frac{c^{i0} \, T^0}{\gamma_i \, (\gamma_i+1)} \, \left(\frac{T}{T^0}\right)^{\gamma_i+1} + g^{i1} \, T + g^{i0} \, ,
\end{align}
subject to the restrictions
\begin{align}\begin{split}\label{expones}
0 \leq \alpha_i < 1, & \quad 0 < \beta_i < 1, \quad \alpha_i + \beta_i \leq 1\, ,\\
& \gamma_i \geq 0, \, \quad c^{i0} > 0 \, ,
\end{split}
\end{align}
and certain constants $\rho_i^{\rm R} > 0$ and $g^{i0}, \, g^{i1}$ that play no particular role for the analysis of the model.

The practically relevant case of ${\rm A}_i$ being an ideal gas is not covered by \eqref{smallgparticular}, but occurs for the limiting case $\beta_i \rightarrow 0$ and $\alpha_i \rightarrow 1$. In this case we in fact choose
\begin{align}\label{smallgparticular2}
g_i(T, \, p) = \frac{p^0}{\rho_i^{\rm R}} \, \frac{T}{T^0} \, \ln \frac{p}{p^0} - c^{i0} \, T \, \Big(\ln \frac{T}{T^0} - 1\Big) + g^{i1} \, T + g^{i0} \, .   
\end{align}
We here impose the restriction that
\begin{align}\label{idealrestrictc}
c^{i0} > \frac{p^0}{T^0 \, \rho_i^{\rm R}} \, .
\end{align}
\begin{lemma}\label{goodforall}
The conditions \eqref{expones} and \eqref{idealrestrictc} guarantee the thermodynamical consistency of the choices \eqref{smallgparticular} and \eqref{smallgparticular2}. 
\end{lemma}
\begin{proof}
For \eqref{smallgparticular}, we get
\begin{align*}
& \frac{1}{\hat{\rho}_i(T, \, p)} = \partial_p g_i(T, \, p) = \frac{p^0}{\rho_i^{\rm R} \, p} \, \left(\frac{T}{T^0}\right)^{\alpha_i} \, \left(\frac{p}{p^0}\right)^{\beta_i} \, ,\\
\text{ and } \quad & \\
& \partial^2_{p} g_i = \frac{(\beta_i-1) \, p^0}{\rho_i^{\rm R} \, p^2} \, \left(\frac{T}{T^0}\right)^{\alpha_i} \, \left(\frac{p}{p^0}\right)^{\beta_i}
 < 0\, ,\\
&  \partial^2_{T} g_i = \frac{\alpha_i \, (\alpha_i-1) \, p^0}{\beta_i \, \rho_i^{\rm R} \, T^2} \, \left(\frac{T}{T^0}\right)^{\alpha_i} \, \left(\frac{p}{p^0}\right)^{\beta_i} - \frac{c^{i0}}{T} \, \left(\frac{T}{T^0}\right)^{\gamma_i} < 0 \, .
 \end{align*}
Moreover
 \begin{align}\label{cvhier}
\partial^2_T g_i - \frac{(\partial^2_{T,p} g_i)^2}{\partial^2_p g_i} = & - \frac{\alpha_i \, p^0}{\beta_i \, (1-\beta_i) \, \rho^{\rm R}_i \, T^2} \, \left(\frac{T}{T^0}\right)^{\alpha_i}\, \left(\frac{p}{p^0}\right)^{\beta_i} \, (1- \alpha_i - \beta_i)  - \frac{c^{i0}}{T} \, \left(\frac{T}{T^0}\right)^{\gamma_i}<0 \, .
 \end{align}
For \eqref{smallgparticular2}, we get
\begin{align*}
& \frac{1}{\hat{\rho}_i(T, \, p)} = \frac{p^0}{\rho_i^{\rm R} \, p} \, \frac{T}{T^0} \, \quad
\text{ and } \quad \partial^2_{p} g_i = - \frac{p^0}{\rho_i^{\rm R} \, p^2} \, \frac{T}{T^0}
 < 0\, ,
\quad  \partial^2_{T} g_i = - \frac{c^{i0}}{T} < 0 \, .
 \end{align*}
Moreover
 \begin{align}\label{cvhierid}
\partial^2_T g_i - \frac{(\partial^2_{T,p} g_i)^2}{\partial^2_p g_i} = & - \frac{1}{T} \, \Big(c^{i0} - \frac{p^0}{T^0 \, \rho_i^{\rm R}}\Big) \, .
 \end{align} 
Under the condition \eqref{idealrestrictc}, we obtain that $c^i_{\upsilon}$ is a positive constant.
\end{proof}
\begin{lemma}
Assume that for $i = 1,\ldots,N$, the function $g_i$ is given either by \eqref{smallgparticular} or \eqref{smallgparticular2}, where we adopt the assumptions \eqref{expones} and \eqref{idealrestrictc}.  We additionally suppose that there exists at least one index $i$ such that $g_{i}$ obeys \eqref{smallgparticular2}.
Then, with $\epsilon^{\min}_i = g^{i0}$ for $i=1,\ldots,N$, all assumptions of Prop.\ \ref{yeswehave} are satisfied. 
\end{lemma}
\begin{proof}
We verify easily that $g_i(T,\, p) \rightarrow + \infty$ for $p \rightarrow + \infty$. Consider any index $i$ such that $g$ obeys \eqref{smallgparticular2}, which is the case for at least one. Then, for $0 < p < p^0$ and $T \geq T_1 > 0$ arbitrary we have
\begin{align*}
& g_{i}(T, \, p) \leq \frac{p^0}{\rho_i^{\rm R}} \, \frac{T_1}{T^0} \, \ln \frac{p}{p^0} + \sup_{T \geq T_1} \{-c^{i0} \, T \, (\ln T/T^0-1) + g^{i1} \, T + g^{i0}\} \, \\
 \Rightarrow \quad & \inf_i \, \sup_{T \geq T_1}  g_{i}(T, \, p) \leq c_0 \, \ln \frac{p}{p^0} + c_1 \, .
\end{align*}
Hence, $\lim_{p \rightarrow 0+} \inf_i \, \sup_{T \geq T_1}  g_{i}(T, \, p) = -\infty$.
 
 For \eqref{smallgparticular}, we have $\partial_pg_i(T, \, p_i) = 1/\varrho$ iff
 \begin{align*}
  \frac{p_0}{\rho_i^{\rm R} \, p_i} \, \left(\frac{T}{T^0}\right)^{\alpha_i} \, \left(\frac{p_i}{p^0}\right)^{\beta_i} = \frac{1}{\varrho} \quad \Leftrightarrow \quad p_i = p^0 \, \left(\frac{\varrho}{\rho_i^{\rm R}}\right)^{\frac{1}{1-\beta_i}} \, \left(\frac{T}{T^0}\right)^{\frac{\alpha_i}{1-\beta_i}} \, .
 \end{align*}
For \eqref{smallgparticular2}, similar arguments yield
\begin{align*}
 p_i = p^0 \, \frac{\varrho}{\rho_i^{\rm R}} \, \frac{T}{T^0} \, .
 \end{align*}
 Recall that $\alpha_i < 1-\beta_i$. Hence, for $\varrho \leq \varrho_1$, and $T$ sufficiently large the largest $p_i$ is always associated with the choice \eqref{smallgparticular2}, and it is of the form $c_1 \, T$ with a certain positive constant $c_1$. Hence, for all $g_i$ obeying \eqref{smallgparticular2} we have
\begin{align*}
 g_i(T, \, p_{\max}(T, \, \varrho)) \leq & \frac{p^0}{\rho_i^{\rm R} \, T^0} \ln (c_1 \, T/p^0) - c^{i0} \, T \, \ln (T/T^0) + \tilde{g}^{i1} \, T + g^{i0}\\
 \leq & (-c^{i0} +\frac{p_0}{\rho_i^{\rm R} \, T^0}) \, T \, \ln T +  \tilde{g}^{i1} \, T + \tilde{g}^{i0} \, .
\end{align*}
 Now it is readily seen that
 \begin{align*}
  \lim_{T \rightarrow + \infty} \min_i \, \sup_{\varrho \leq \varrho_1} \frac{ g_i(T, \, p_{\max}(T, \, \varrho))}{T} = - \infty \, .
 \end{align*}
For $g_i$ obeying \eqref{smallgparticular} we have
\begin{align*}
\lim_{T,p \rightarrow 0+} g_i(T, \, p) = g^{i0}, \, \quad \lim_{T,p \rightarrow 0+} T \, \partial_T g_i(T,p) = 0 \, . 
\end{align*}
For $g_i$ obeying \eqref{smallgparticular2},  we can show as above that for $\varrho_1 > 0$ fixed and $T$ sufficiently small, the function $p_{\min}(T, \, \varrho_1)$ is of the form $c_2 \, T$ with a positive constant $c_2$. Thus, for $p_1 \geq p \geq p_{\min}(T, \, \varrho_1)$ we have
\begin{align*}
T \, \ln p_1 \geq T \, \ln p \geq T \, (\ln T - C_1) \, 
\end{align*}
which, for \eqref{smallgparticular2}, allows showing that
\begin{align*}
 \lim_{T,p \rightarrow 0+, \, p \geq p_{\min}(T, \, \varrho_1)} g_i(T, \, p) = g^{i0}, \quad \lim_{T,p \rightarrow 0+, \, p \geq p_{\min}(T, \, \varrho_1)} T \, \partial_T g_i(T,p) = 0 \, .
\end{align*}
The other claims are proven similarly.
\end{proof}
\begin{rem}
The assumption that there is one index such that $g_i$ is given by \eqref{smallgparticular2} is rather artificial and in fact not needed. More natural and sufficient is the condition that at least one of the $g_i$ is singular in $p$ in a neighourhood of the origin $(T=0, \, p=0)$. For instance, let $g_i$ obey
\begin{align*}
 g_i(T, \, p) = \frac{p^0}{\rho_i^{\rm R}} \, \left(\frac{T}{T^0}\right)^{\alpha_i} \, \left(\ln \frac{p}{p^0} +\frac{1}{\beta_i} \, \left(\frac{p}{p^0}\right)^{\beta_i} \right) - \frac{c^{i0} \, T^0}{\gamma_i \, (\gamma_i+1)} \, \left(\frac{T}{T^0}\right)^{\gamma_i+1} + g^{i1} \, T + g^{i0} \, ,
\end{align*}
for at least one constituent. However, working with the explicit formula \eqref{smallgparticular}, \eqref{smallgparticular2} is simpler if we want to compute exponents. We remind the reader that these are just particular examples. 
\end{rem}

For later discussion of the maximum principle for the temperature, we will restrict to the case where, at large temperatures, there is a single dominant species with respect to the heat capacity. This concept is explained in the next definition. It relies on the representations \eqref{cvhier} and \eqref{cvhierid} of the heat capacities of the species.
\begin{defin}\label{gammamax}
Assume that for $i = 1,\ldots,N$, the function $g_i$ obeys either \eqref{smallgparticular} or \eqref{smallgparticular2}. We set $\gamma_i := 0$ in the case of \eqref{smallgparticular2}, and we define $$\gamma_{\max} := \max_{i=1,\ldots,N} \gamma_i \, .$$ We let $k_1, \ldots, k_m$ be an enumeration of the indices such that $\gamma_{k_j} = \gamma_{\max}$ (where $1 \leq m \leq N$). If there is a strict ordering: 
\begin{align*}
\begin{cases} c^{k_10} < c^{k_20} < \ldots < c^{k_m0} \, , & \text{ for } \gamma_{\max} > 0 \, ,\\
c^{k_10} - \frac{p^0}{\rho_{k_1}^{\rm R} \, T^0} <  c^{k_20} - \frac{p^0}{\rho_{k_2}^{\rm R} \, T^0}< \ldots < c^{k_N0} - \frac{p^0}{\rho_{k_N}^{\rm R} \, T^0} \, , & \text{ for }\gamma_{\max} = 0 \, ,
\end{cases}
\end{align*}
then, we define $I := k_m$ call the species ${\rm A}_I$ \emph{dominant with respect to the heat capacity}. 
\end{defin}

\vspace{0.2cm}

\section{Proof of the maximum principle for a special ideal mixture} 

In this section we prove the maximum principle for the free energy model relying on \eqref{muiideal} and \eqref{smallgparticular}, \eqref{smallgparticular2} (Theorem \ref{MAIN2}). At first we need to work out growth properties of the basic thermodynamic quantities for large temperatures.

\subsection{Growth properties of some basic thermodynamic quantities}

For a general ideal model, the pressure function $p = \tilde{p}(T, \, \rho)$ is defined as the root of the equation \eqref{implicitpress}.
If $g_i$ obeys \eqref{smallgparticular} or \eqref{smallgparticular2}, we have
\begin{align*}
\partial_p g_i(T, \, p) = \frac{1}{
\bar{\rho}_i^{\rm R}} \, \left(\frac{p^0}{p} \right)^{1-\beta_i} \, \left(\frac{T}{T^0}\right)^{\alpha_i} \, .
\end{align*}
Hence, computing the functions $p_{\min}$ and $p_{\max}$ (See Lemma \ref{pminmax}) is a straightforward exercise. For at least two indicies $i_0, \, i_1$ (the choice of which is presently depending on the state variables!), we obtain that
\begin{align}\label{Pboundsbasics}
p^0 \, \left(\frac{T}{T^0}\right)^{\frac{\alpha_{i_1}}{1-\beta_{i_1}}} \, \left(\frac{\varrho}{\rho_{i_1}^{\rm R}}\right)^{\frac{1}{1-\beta_{i_1}}} \leq p \leq p^0 \, \left(\frac{T}{T^0}\right)^{\frac{\alpha_{i_0}}{1-\beta_{i_0}}} \, \left(\frac{\varrho}{\rho_{i_0}^{\rm R}}\right)^{\frac{1}{1-\beta_{i_0}}} \, .
\end{align}
Now for $i = 1,\ldots,N$, we also see that
\begin{align}\label{fistterm}
& \Big( \frac{T}{T^0}\Big)^{\alpha_i} \, \Big(\frac{p}{p^0}\Big)^{\beta_i} \leq c_i \, \varrho^{\frac{\beta_i}{1-\beta_{i_0}}} \, T^{\alpha_i + \beta_i \, \frac{\alpha_{i_0}}{1-\beta_{i_0}}} \, , \nonumber\\
\text{ where } \quad & \\
& \alpha_i + \beta_i \, \frac{\alpha_{i_0}}{1-\beta_{i_0}} = \alpha_i + \beta_i - \beta_i \, \frac{1-\alpha_{i_0} - \beta_{i_0}}{1-\beta_{i_0}} \leq 1 \,. \nonumber
\end{align}
Recalling the ansatz \eqref{smallgparticular} for $g_i$, the latter implies that, for fixed $\varrho$, the functions $g_i$ having $\gamma_i > 0$ are all strictly negative if $T$ is sufficiently large. In fact, the heat capacity term dominates, and we have
\begin{align}\label{gidominates}
& g_i = - \frac{T^0}{(\gamma_i+1) \, \gamma_i}  \, \left(\frac{T}{T^0}\right)^{1+\gamma_i} \, (c^{i0} - o_i)\\
\text{ where } \quad & \nonumber\\
&  \frac{T^0}{(\gamma_i+1) \, \gamma_i} \, o_i := \frac{p^0}{\beta_i \, \rho_i^{\rm R}} \, \left(\frac{T}{T^0}\right)^{\alpha_i-\gamma_i-1} \, \left(\frac{p}{p^0}\right)^{\beta_i} + (g^{1i} \, T + g^{i0}) \, \Big(\frac{T}{T^0}\Big)^{-\gamma_i - 1} \, . \nonumber 
 \end{align}
 With the help of \eqref{fistterm}, and recalling that $1/T = - q_N$, we now have
 \begin{align*}
T^{\gamma_i} \, |o_i| \leq & c_i \, ( \varrho^{\frac{\beta_i}{1-\beta_{i_0}}} \, T^{\alpha_i + \beta_i \, \frac{\alpha_{i_0}}{1-\beta_{i_0}} - 1} + |g^{1i}|+ |g^{i0}| \, T^{- 1}) \, \\
 \leq & c_i \, (\varrho^{\frac{\beta_i}{1-\beta_{i_0}}} \, (-q_N)^{1 - \alpha_i - \beta_i \, \frac{\alpha_{i_0}}{1-\beta_{i_0}}} + 1 -q_N) 
 =:  c_i(\varrho, \, -q_N) \, .
\end{align*}
We see that $o_i$ is bounded for bounded entropic variables. Moreover
\begin{align}\label{tildegilowerorder}
|o_i| \leq c_i(\varrho, \, -q_N) \, (-q_N)^{\gamma_i} \rightarrow 0 \quad \text{ for } \quad q_N \rightarrow 0- \, . 
\end{align}
To obtain a similar result in the case of \eqref{smallgparticular2}, we again use \eqref{Pboundsbasics} to  estimate
\begin{align}\label{secondest}
 T \, \ln \frac{p}{p^0} \leq & T \, \ln \left(\frac{T}{T^0}\right)^{\frac{\alpha_{i_0}}{1-\beta_{i_0}}} +T\, \ln \left(\frac{\varrho}{\rho_{i_0}^{\rm R}}\right)^{\frac{1}{1-\beta_{i_0}}}
 \leq \frac{\alpha_{i_0}}{1-\beta_{i_0}} \, T \, \ln \frac{T}{T^0} + c \, T \, (1+\ln \varrho) \, .
\end{align}
Thus
\begin{align}\label{gidominate2}
& g_i = - T \, \ln \frac{T}{T^0} \, (c^{i0} -  \frac{p^0}{\rho_i^{\rm R} \, T^0} - o_i) \, ,
\\
\text{ where } \quad & \nonumber\\
& o_i := - \frac{p^0}{\rho_i^{\rm R} \, T^0} \, \Big(1 - \frac{\ln \frac{p}{p^0}}{\ln \frac{T}{T^0}}\Big) + \frac{1}{\ln \frac{T}{T^0}} \, (c^{i0} + g^{i1} + \frac{g^{i0}}{T}) \, .  \nonumber
%
%
\end{align}
We see with the help of \eqref{secondest} that
\begin{align*}
\ln T \, |o_i| \leq c_i \, \ln T\frac{1 + |\ln \varrho|}{\ln T - \ln T^0} = c_i \, \ln \frac{1}{|q_N|} \, \frac{1+|\ln \varrho|}{\ln \frac{1}{|q_N|} - \ln T^0} \, ,
\end{align*}
showing that $o_i$ is bounded for bounded entropic variables, and moreover that
\begin{align}\label{tildegilowerorder2}
|o_i| \leq  c_i( \varrho, \, -q_N) \, \frac{1}{\ln \frac{1}{|q_N|}}  \rightarrow 0 \quad \text{ for } \quad q_N \rightarrow 0- \, .
\end{align}

\vspace{0.2cm}

\subsection{Estimates for the growth in temperature at fixed entropic variables}

We showed in Remark \ref{NEWVARB} that all thermodynamic quantities can be introduced as functions of the variables $\varrho$ and $q_1, \ldots, q_N$. If we fix $\varrho$ and $\bar{q} = (q_1, \ldots, q_{N-1})$, a thermodynamic quantity is a function of $q_N = - 1/T$ only, hence a function of $T = - 1/q_N$ only. To hint at this situation we shall, in this paragraph, use the tilde superscripts again, for instance
\begin{align*}
\tilde{\rho}_i(T) = \tilde{\rho}_i(T; \, \varrho, \, \bar{q}) = \mathscr{R}_i(\varrho, \, \bar{q}, \, -1/T) \, .
\end{align*}
Now we want to investigate the behaviour for $T \rightarrow + \infty$ for several functions. This is a key idea to obtain a maximum principle for the temperature using the general energy equation like in the preceding section. For two functions $\tilde{f}_1(T), \, \tilde{f}_2(T)$, we use the notations
\begin{align*}
\tilde{f}_1(T) \precsim \tilde{f}_2(T), \quad (\tilde{f}_1(T) \succsim \tilde{f}_2(T)) \, 
\end{align*}
if there are continuous functions $\phi$ and $\psi$ on $\mathbb{R}_+ \times \overline{\mathcal{H}^N_-}$ with $\phi$ nonnegative and $\psi$ strictly positive, such that
\begin{align*}
 \tilde{f}_1(T) \leq \phi(\varrho, \, q) \, \tilde{f}_2(T), \quad (\tilde{f}_1(T) \geq \psi(\varrho, \, q) \, \tilde{f}_2(T)) \, .
\end{align*}
Moreover $\tilde{f}_1(T) \eqsim \tilde{f}_2(T)$ means that both $\tilde{f}_1(T) \succsim \tilde{f}_2(T)$ and $\tilde{f}_1(T) \precsim \tilde{f}_2(T)$ are valid.

To start with, the identities \eqref{muiideal} allow to epress, with $x_1, \ldots, x_N$ denoting the mole fractions,
\begin{align*}
 \frac{1}{M_i} \, \ln x_i = \frac{\mu_i}{T} - \frac{g_i(T,\, p)}{T} \quad \text{ for } \quad i = 1,\ldots,N 
\end{align*}
and, using the entropic variables with the transformation \eqref{TRAFO1} (recall that $\xi^N_i = 0$ for all $1 \leq i \leq N$)
\begin{align}\label{moule}
 \frac{1}{M_i} \, \ln x_i = \sum_{k=1}^{N-1} q_k \, \xi^k_i + \mathscr{M}(\varrho, \, q) - \frac{g_i(T,\, p)}{T} \, .
\end{align}
Since $x_i \leq 1$ by definition, we obtain that
\begin{align*}
& \frac{g_i(T, \, p)}{T} \geq \mathscr{M}(\varrho, \, q) - |\mathcal{Q}| \, |\bar{q}| \quad \text{ for all } i = 1,\ldots,N\\
\Longrightarrow \quad & \\
& \mathscr{M}(\varrho, \, q) \leq \frac{1}{T} \, \inf_{i} g_i(T, \, p) + |\mathcal{Q}| \, |\bar{q}| \, .
\end{align*}
The property $\sum_i x_i = 1$ also implies that there is at least one $i_2$ such that $x_{i_2} \geq 1/N$. Thus, \eqref{moule} also implies that
\begin{align*}
 \mathscr{M}(\varrho, \, q) \geq & \frac{g_{i_2}(T,\, p)}{T} - \frac{1}{M_{i_2}} \, \ln N - |\mathcal{Q}| \, |\bar{q}|\\
 \geq & \frac{1}{T} \, \inf_{i} g_i(T, \, p) - C \, (1+|\bar{q}|) \, .
\end{align*}
We hence see that the factor $\mathscr{M}(\varrho, \, q)$ satisfies
\begin{align}\label{resultforscriptm}
| \mathscr{M}(\varrho, \, q) - \frac{1}{T} \, \inf_{i} g_i(T, \, p)| \leq C \, (1+|\bar{q}|) \, .
\end{align}
%
%
%
%
%
%
%

Due to \eqref{gidominates}, \eqref{tildegilowerorder} and \eqref{gidominate2}, \eqref{tildegilowerorder2}, the heat capacity terms in the formula for the $g_i$s always dominate at large temperature. We now exploit this fact to determine the behaviour of $\inf_{i = 1, \ldots,N} \tilde{g}_i(T)$ for $T$ large.
\begin{lemma}\label{rhoasump}
We assume that $g_1, \ldots,g_N$ obey the representation \eqref{smallgparticular} or \eqref{smallgparticular2}, with exponents $\alpha_i$, $\beta_i$ and $\gamma_i$ subject to \eqref{expones} and, in the case \eqref{smallgparticular2}, subject also to \eqref{idealrestrictc}. We define $\gamma_{\max} := \max_{i=1,\ldots,N} \gamma_i$, and we assume that there is a single index $I$ such that the species ${\rm A}_I$ dominates in the sense of Definition \ref{gammamax}.
Then,
\begin{itemize}
 \item For all $j \neq I$, there is a strictly positive, continuous function $b_j(\varrho, \, \bar{q})$ such that
 \begin{align*}
  \tilde{\rho}_j(T) \eqsim \begin{cases}  e^{-b_j \, T^{\gamma_{\max}}} & \text{ if } \gamma_{\max} > 0 \, ,\\
                T^{-b_j} & \text{ if } \gamma_{\max} = 0 \, .
              \end{cases}
 \end{align*}
 \item $\lim_{T \rightarrow +\infty} \tilde{\rho}_i(T) = \varrho \, \delta_{Ii}$, with the Kronecker symbol $\delta$.
\end{itemize}
\end{lemma}
\begin{proof}
For two arbitary indices $i \neq j$, we want to investigate the asymptotic behaviour of the differences $\tilde{g}_i(T) - \tilde{g}_j(T)$ for $T \rightarrow + \infty$, while $(\varrho, \, \bar{q})$ remains fixed.

First we choose $i \neq j$ such that $\gamma_i > \gamma_j > 0$. Then, the formula \eqref{gidominates} implies that
\begin{align}\label{giovergj}
 \frac{\tilde{g}_j(T)}{\tilde{g}_{i}(T)} = \frac{(\gamma_i+1) \, \gamma_i}{(\gamma_j+1) \, \gamma_j}  \, \Big(\frac{T}{T^0}\Big)^{\gamma_j-\gamma_i} \, \frac{c^{j0} - o_j}{c^{i0} - o_i}  \, .
\end{align}
In view of the property \eqref{tildegilowerorder}, we can choose a number $T_1 = T_1(\varrho,i,j) > T^0$ such that
\begin{align*}
|o_i| \leq & \frac{1}{2} \, c^{i0} \quad \text{ and } \quad
2 \, \frac{c^{j0}}{c^{i0}}\, \frac{(\gamma_i+1) \, \gamma_i}{(\gamma_j+1) \, \gamma_j}  \, \Big(\frac{T}{T^0}\Big)^{\gamma_j-\gamma_i} \leq   \frac{1}{2} \quad \text{ for all } T \geq T_1 \, .
\end{align*}
Then $\tilde{g}_j(T)/\tilde{g}_i(T) \leq 1/2$ and, invoking \eqref{gidominates},
\begin{align}\label{differencefirst}
 \tilde{g}_i(T) - \tilde{g}_j(T) \leq \frac{1}{2} \, \tilde{g}_i(T) \leq - \frac{c_{0i}}{4} \,  \frac{T^0}{(\gamma_i+1) \, \gamma_i}  \, \Big(\frac{T}{T^0}\Big)^{1+\gamma_i} \, .
\end{align}
Second, we consider the case that $\gamma_i = \gamma_j > 0$. Then, by means of \eqref{giovergj}
\begin{align*}
 \frac{\tilde{g}_j(T)}{\tilde{g}_{i}(T)} =  \frac{c^{j0} - o_j}{c^{i0} - o_i} \, .
\end{align*}
Now, suppose that $c^{j0} < c^{i0}$. Then, by means of \eqref{tildegilowerorder}, we can find a $T_2(\varrho,i,j) > T^0$ such that, for all $T \geq T_2$
\begin{align*}
 c^{i0} - o_i > (1+\delta) \,c^{j0} - o_j,\quad 2\delta := \frac{c^{i0}}{c^{j0}} - 1  \, .
 \end{align*}
 Then, for $T \geq T_2$
 \begin{align}\label{differencesecond}
 \tilde{g}_i(T) - \tilde{g}_j(T) \leq \frac{\delta}{1+\delta} \, \tilde{g}_i(T) \leq - \frac{c^{i0}-c^{j0}}{c^{i0}+c^{j0}} \, c^{i0} \, \frac{T^0}{(\gamma_i+1) \, \gamma_i}  \,\Big(\frac{T}{T^0}\Big)^{1+\gamma_i} \, .
\end{align}
The third case is that $i \neq j$ are such that $\gamma_i > 0$ and $\gamma_j = 0$. Then we invoke \eqref{gidominates} and \eqref{gidominate2} to see that
\begin{align*}
 \frac{\tilde{g}_j(T)}{\tilde{g}_{i}(T)} =
 [T^0]^{\gamma_i} \, (\gamma_i+1) \, \gamma_i  \, \frac{\ln \frac{T}{T^0}}{T^{\gamma_i}}  \, \frac{c^{j0} -  p^0/(\rho_j^{\rm R} \, T^0) - o_j}{c^{i0} - o_i} \, .
\end{align*}
In the same way as for \eqref{differencefirst}, the difference $\tilde{g}_i - \tilde{g}_j$ is asymptotically equivalent to $- T^{\gamma_i+1}$ for sufficiently large $T$.

The fourth and last case is $\gamma_i = \gamma_j = 0$, and then
\begin{align*}
  \frac{\tilde{g}_j(T)}{\tilde{g}_{i}(T)} = \frac{c^{j0} - p^0/(\rho_j^{\rm R} \, T^0) - o_j}{c^{i0} - p^0/(\rho_i^{\rm R} \, T^0) - o_i} \, .
\end{align*}
If $c^{i0} - p^0/(\rho_i^{\rm R} \, T^0) > c^{j0} - p^0/(\rho_j^{\rm R} \, T^0)$, we can argue as for \eqref{differencesecond}, and find in this case with the help of \eqref{gidominate2} that
\begin{align}\label{differencethird}
\tilde{g}_i(T) - \tilde{g}_j(T) \leq  - \frac{c^{i0}- p^0/(\rho_i^{\rm R} \, T^0)-c^{j0}+ p^0/(\rho_j^{\rm R} \, T^0)}{c^{i0}-p^0/(\rho_i^{\rm R} \, T^0)+c^{j0}-p^0/(\rho_j^{\rm R} \, T^0)} \, (c^{i0}- p^0/(\rho_i^{\rm R} \, T^0)) \, T \, \ln \frac{T}{T^0} \, .
\end{align}
Due to these considerations, we see that for fixed entropic variables $\varrho$ and $\bar{q} = (q_1, \ldots, q_{N-1})$
\begin{align*}
\lim_{T \rightarrow \infty} (\tilde{g}_i(T)-\tilde{g}_j(T)) = - \infty \quad \text{ for } \begin{cases}
 \gamma_i > \gamma_j \, , &\\
 \gamma_i = \gamma_j > 0 & \text{ if } c^{i0} > c^{j0}\, ,\\
\gamma_i = \gamma_j = 0 & \text{ if } c^{i0} - \frac{p^0}{\rho_i^{\rm R} \, T^0} > c^{j0} - \frac{p^0}{\rho_j^{\rm R} \, T^0} \, .
\end{cases}
\end{align*}
For fixed density $\varrho$ and sufficiently large temperature $T$, we have $\inf_{i = 1,\ldots,N} \tilde{g}_i(T) = \tilde{g}_{I}(T)$ and the index $I$ is uniquely determined as follows:
\begin{align*}
& \text{ For } \max\{\gamma_1, \ldots,\gamma_N\} > 0 \, : \,\gamma_{I} = \max\{\gamma_1, \ldots,\gamma_N\} \text{ and } c^{I0} > c^{k0} \text{ for all } \gamma_k = \gamma_{\max}\, ,\\
& \text{ For } \max\{\gamma_1, \ldots,\gamma_N\} = 0 \, : \,
c^{I0} - \frac{p^0}{\rho_{I}^{\rm R} \, T^0} = \max_k \Big(c^{k0} - \frac{p^0}{\rho_k^{\rm R} \, T^0}\Big)  \, .
\end{align*}
Then, in view of \eqref{resultforscriptm}, we see that, for $\gamma_{\max} > 0$,
\begin{align*}
 \Big| \mathscr{M}(\varrho, \, q) + \frac{c^{I0}}{\gamma_{\max}  (\gamma_{\max}+1)}  \, \Big(\frac{T}{T^0}\Big)^{\gamma_{\max}}\Big| \leq C \, (1+|\bar{q}| + c(\varrho,1/T)) \,,
\end{align*}
where $c$ occurs in \eqref{tildegilowerorder} and is uniformly bounded.
If $\gamma_{\max} = 0$, then with the $c(\varrho, \, 1/T)$ of \eqref{tildegilowerorder2}, we have instead
\begin{align*}
 \Big| \mathscr{M}(\varrho, \, q) +  \ln \frac{T}{T^0} \, \Big(c^{i0} -  \frac{p^0}{\rho_i^{\rm R} \, T^0}\Big) \Big| \leq C \, (1+|\bar{q}| + c(\varrho,1/T)) \, .
\end{align*}
We recall \eqref{moule}. If $i = I$, then \eqref{resultforscriptm} shows that
\begin{align*}
& (\widetilde{x_{I}}(T))^{\frac{1}{M_I}} = e^{\bar{q} \cdot \xi_{I}} \, e^{\mathscr{M} - \frac{\tilde{g}_{I}(T)}{T}}\\
\text{ implying that } & \\
& e^{\bar{q} \cdot \xi_{I}} \, e^{-C \, (1+|\bar{q}| + c( \varrho, \, 1/T))}  \leq (\widetilde{x_{I}}(T))^{\frac{1}{M_I}} \leq e^{\bar{q} \cdot \xi_{I}} \, e^{C \, (1+|\bar{q}| + c(\varrho, \, 1/T))} \, .
\end{align*}
In short, the fraction $x_I$ is bounded and strictly positive independently on the temperature.
For $j \neq I$, we have
\begin{align*}
\frac{1}{M_j} \, \ln \widetilde{x_j}(T) = \frac{1}{M_{I}} \, \ln \widetilde{x_{I}}(T) + \frac{\tilde{g}_{I}(T) - \tilde{g}_j(T)}{T} + \bar{q} \cdot (\xi_j - \xi_{I}) \, 
\end{align*}
and therefore
\begin{align}\label{xjasu}
(\widetilde{x_j}(T))^{\frac{1}{M_j}} = (\widetilde{x_I}(T))^{\frac{1}{M_I}} \, e^{ \bar{q} \cdot (\xi_j - \xi_{I})} \, e^{\frac{\tilde{g}_{I}(T) - \tilde{g}_j(T)}{T}}  \, .
\end{align}
Now, as shown in \eqref{differencefirst}, \eqref{differencesecond}, for $\gamma_{\max} > 0$, we find $T_1 > T_0$ such that $\widetilde{g}_I(T) - \widetilde{g}_{j}(T) = - c_{Ij} \, T^{1+\gamma_{\max}}$, with $c_{Ij}$ positive and bounded, for all $T \geq T_1$. Hence $\exp(\frac{\tilde{g}_{I}(T)}{T} - \frac{\tilde{g}_j(T)}{T}) = e^{- c_{Ij} \, T^{\gamma_{\max}}}$.
If $\gamma_{\max} = 0$, then for all $T \geq T_2$, \eqref{differencethird} yields
$\tilde{ g}_I(T) - \tilde{g}_{j}(T) = - c_{Ij} \, T \, \ln T$. Therefore $\exp(\frac{\tilde{g}_{I}(T)}{T} - \frac{\tilde{g}_j(T)}{T}) = T^{-c_{Ij}}$.

Overall, for all $j \neq I$ \eqref{xjasu} shows that
\begin{align*}
\widetilde{ x}_j(T) \eqsim \begin{cases}  e^{-b_{j} \, T^{\gamma_{\max}}} & \text{ for } \gamma_{\max} > 0 \, ,\\
 T^{-b_j}  & \text{ for }  \gamma_{\max} = 0 \, ,
\end{cases}
\end{align*}
with strictly positive, continuous functions $b_j$ of the entropic variables.
\end{proof}
\begin{lemma}\label{otherasymptotics}
With the index $I$ of Lemma \ref{rhoasump}, we also have
\begin{gather*}
 \tilde{\rho}_i(T) \, \widetilde{\partial_p g_i}(T), \,\, \tilde{\rho}_i(T) \,  T \, \widetilde{\partial^2_{T} g_i}(T)  \longrightarrow 0 \quad \text{ for all } \quad i \neq I  \quad \text{ as } T \rightarrow \infty\, ,\\
\tilde{p}(T) \eqsim T^{\frac{\alpha_I}{1-\beta_I}} \, .
\end{gather*}
\end{lemma}
\begin{proof}
In the case that $\gamma_{\max} > 0$, $i \neq I$ implies that $ \tilde{\rho}_i(T) \eqsim e^{-b_i \, T^{\gamma_{\max}}}$. On the other hand, we have $\partial_p g_i = c_i \, T^{\alpha_i} \, p^{\beta_i-1}$, and \eqref{Pboundsbasics} implies that
\begin{align*}
\widetilde{ \partial_p g_i}(T) \leq \tilde{c}_i \, T^{\alpha_i- \alpha_{i_1} \, \frac{1-\beta_i}{1-\beta_{i_1}}} \, . 
\end{align*}
Thus, $\lim_{T \rightarrow \infty} \tilde{\rho}_i(T) \, \widetilde{\partial_pg_i}(T) \leq \tilde{c}_i \,  \lim_{T \rightarrow \infty} \, e^{-b_i \, T^{\gamma_{\max}}} \, T^{\alpha_i- \alpha_{i_1} \, \frac{1-\beta_i}{1-\beta_{i_1}}}= 0$.

If $\gamma_{\max} = 0$, then $\partial_p g_i = c_i \, T/p$, and \eqref{Pboundsbasics} implies that $\partial_p g_i \leq \tilde{c}_i$. 
Hence
$$\lim_{T \rightarrow \infty}  \tilde{\rho}_i(T) \, \widetilde{\partial_pg_i}(T) \leq \tilde{c}_i \,  \lim_{T \rightarrow \infty} \, T^{-b_i} = 0 \, .$$

We can prove the same for $\tilde{\rho}_i(T) \, \widetilde{\partial_T^2 g_i}(T)$. If $i\neq I$ and $\gamma_{\max} > 0$, then $T \, \widetilde{\partial_T^2 g_i}(T)$ possesses at most polynomial growth. If $\gamma_{\max} = 0$, then $T \, \widetilde{\partial_T^2 g_i}(T)$ is bounded.

To investigate the growth exponent of $\tilde{p}(T)$, we use the preceding result, showing that
\begin{align*}
 \rho_{I} \, \partial_p g_{I} =  1 - \sum_{j \neq I} \rho_j \, \partial_p g_j = 1 - \delta(T; \, \varrho, \, q) \, ,
\end{align*}
where $\delta \rightarrow 0$ as $T \rightarrow +\infty$. Thus, for all $T \geq T_1$ such that $\delta \leq 1/2$, we have
\begin{align*}
 \tilde{p}(T) \leq  c_i \, \left(\frac{\rho_I}{\varrho}\right)^{\frac{1}{1-\beta_{I}}} \, \varrho^{\frac{1}{1-\beta_{I}}} \, T^{\frac{\alpha_{I}}{1-\beta_{I}}} \, .
\end{align*}
We prove the lower bound similarly.
\end{proof}
Next we investigate other basic thermodynamic functions. For notational relief, we introduce the enthalpies
\begin{align*}
H_i := g_i(T, \, \tilde{p}) - T \, \partial_T g_i(T, \, \tilde{p})\, , \, \qquad
 \varrho H:= \sum_{i=1}^N\rho_i \, (g_i(T, \, \tilde{p}) - T \, \partial_T g_i(T, \, \tilde{p})) \, .
\end{align*}
\begin{prop}\label{ENERGU}
 Assumptions of Lemma \ref{rhoasump}. Then
 \begin{align*}
  \widetilde{c_{\upsilon}}(T) \eqsim  T^{\gamma_{\max}} \, , \quad  \widetilde{\varrho \, u}(T), \, \widetilde{\varrho H}(T) \eqsim T^{1+\gamma_{\max}} \, .
 \end{align*}
\end{prop}
\begin{proof}
We recall that
\begin{align*}
 \varrho \, c_{\upsilon} = \partial_T \varrho u = - T \, \left(\sum_{i=1}^N\partial^2_{T,T} g_i(T, \, p) \, \rho_i -\frac{\Big(\sum_{i=1}^N \partial^2_{T,p} g_i(T, \, p) \, \rho_i\Big)^2}{\sum_{i=1}^N \partial^2_{p} g_i(T, \, p) \, \rho_i} \right) \, . 
\end{align*}
We have shown in Lemma \ref{otherasymptotics} that $ T \, \widetilde{\partial^2_T g_i}(T) \, \tilde{\rho}_i(T) \rightarrow 0$ as $ T \rightarrow \infty$ for all $i \neq I$.
Moreover
\begin{align}\label{lici} 
\frac{\Big(\sum_{i=1}^N \partial^2_{T,p} g_i \, \rho_i\Big)^2}{\sum_{i=1}^N \partial^2_{p} g_i \, \rho_i} - \frac{\rho_I \, (\partial^2_{T,p} g_I)^2}{\partial^2_{p} g_I}
= \frac{\rho_I \, (\partial^2_{T,p} g_I)^2}{\partial^2_{p} g_I} \, \Big( \frac{(1 + \sum_{i\neq I} (\rho_i/\rho_I) \, (\partial^2_{T,p} g_i/\partial^2_{T,p} g_I))^2}{1+\sum_{i\neq I} (\rho_i/\rho_I) \, (\partial^2_{p} g_i/\partial^2_{p} g_I) } - 1\Big) \, .
\end{align}
Since $ \partial^2_{T,p} g_i = \frac{\alpha_i}{T} \, \partial_p g_i$, and since $\partial^2_{p}g_i = -\frac{1-\beta_i}{p} \, \partial_p g_i$, it follows that
\begin{align*}
 - \rho_I \, \frac{(\partial^2_{T,p} g_I)^2}{\partial^2_{p} g_I} =  \alpha_I^2 \, (1 -\beta_I) \, \frac{p}{T^2} \, \rho_I \, \partial_pg_I \, \leq  \alpha_I^2 \, (1 -\beta_I) \, \frac{p}{T^2} \ ,
\end{align*}
where we use that $\rho_I \, \partial_p g_I \leq 1$. Hence\begin{align*}
\Big|T \, \tilde{\rho}_I \, \frac{(\widetilde{\partial^2_{T,p} g_I}(T))^2}{\widetilde{\partial^2_{p} g_I}(T)}  \Big| \leq c \, \frac{\tilde{p}(T)}{T} \leq c \, .
\end{align*}
Moreover, invoking Lemma \ref{otherasymptotics} again,
\begin{align*}
 \frac{\tilde{\rho}_i(T)}{\tilde{\rho}_I(T)} \,
 \frac{\widetilde{\partial^2_{T,p} g_i}(T)}{\widetilde{\partial^2_{T,p} g_I}(T)} = c_i \, \frac{\tilde{\rho}_i(T) \, \widetilde{\partial_p g_i}(T)}{\tilde{\rho}_I(T) \, \widetilde{\partial_pg_I}(T)} \longrightarrow 0\, , \quad
 \frac{\tilde{\rho}_i(T)}{\tilde{\rho}_I(T)} \, \frac{\widetilde{\partial^2_{p} g_i}(T)}{\widetilde{\partial^2_{p} g_I}(T)} = c_i \, \frac{\tilde{\rho}_i(T) \, \widetilde{\partial_p g_i}(T)}{\tilde{\rho}_I(T) \, \widetilde{\partial_pg_I}(T)} \longrightarrow 0\, .
\end{align*}
%
%
%
%
 Thus, \eqref{lici} implies that
 \begin{align*}
 \varrho \, \widetilde{c_{\upsilon}}(T) = - T \, \left(\widetilde{\partial^2_{T} g_I}(T) \, \tilde{\rho}_I(T) -\frac{\tilde{\rho}_I(T) \, \Big( \widetilde{\partial^2_{T,p} g_I}(T) \, \Big)^2}{\widetilde{\partial^2_{p} g_I}(T)} \right) + o(1/T) = \tilde{\rho}_I(T) \, \widetilde{c^I_{\upsilon}}(T) + o(1/T) \, ,
 \end{align*}
where $o(1/T) \rightarrow 0$ for $T \rightarrow \infty$. Since we know that the heat capacity of the species $I$ is asymptotic equivalent with $T^{\gamma_{\max}}$, the claim for $c_{\upsilon}$ follows.
Next we look at the internal energy density and the enthalpy density. If we insert the particular choices \eqref{smallgparticular}, we first obtain that
\begin{align*}
H_j(T, \, p) = g_j(T,p) - T \, \partial_T g_j(T,p) = \frac{p^0 \, (1-\alpha_j)}{\beta_j \, \rho_j^{\rm R}} \, \left(\frac{T}{T^0}\right)^{\alpha_j} \, \left(\frac{p}{p^0}\right)^{\beta_j} + \frac{c^{j0} \, T^0}{\gamma_j+1} \,  \left(\frac{T}{T^0}\right)^{\gamma_j+1} \, .
\end{align*}
Thus
\begin{align*}
 \varrho u + p = \varrho H = \sum_{i=1}^N \rho_i \, (g_i - T \, \partial_T g_i)  \geq \rho_I \, \frac{c^{I0}}{(\gamma_{\max}+1) \, [T^0]^{\gamma_{\max}}} \, T^{1+\gamma_{\max}} \, .
\end{align*}
If $\gamma_{\max} > 0$, and $j \neq I$, the mass density $\tilde{\rho}_j(T)$ decays exponentially, and we therefore easily show that $\tilde{\rho}_j(T) \, \tilde{H}_j(T) \rightarrow 0$ for $T \rightarrow \infty$. Hence $\widetilde{\varrho H}(T) \eqsim T^{1+\gamma_{\max}}$.

If $\gamma_{\max} > 0$, we have $\tilde{p}(T)/T^{1+\gamma_{\max}} \rightarrow 0$. Thus, it also follows that $\widetilde{\varrho u}(T) \eqsim \widetilde{\varrho H}(T)  \eqsim T^{1+\gamma_{\max}}$. 

For \eqref{smallgparticular2}, we have $H_j(T,p) = g_j(T,p) - T \, \partial_T g_j(T,p) = c^{j0} \, T$ and $\varrho H = \varrho u + p = T \, \sum_{i=1}^N c^{i0} \, \rho_i$. In this case, we have also $p = p^0 \, \frac{T}{T_0} \, \sum_{i} (\rho_i/\rho_i^{\rm R})$. Hence $\widetilde{\varrho H}(T) \eqsim T$ and 
\begin{align*}
 \varrho u = T \, \sum_{i=1}^N (c^{i0} - \frac{p^0}{T^0 \, \rho_i^{\rm R}}) \, \rho_i  \geq  \rho_I \, (c^{I0} - \frac{p^0}{T^0 \, \rho_I^{\rm R}}) \, T \, .
\end{align*}
\end{proof}

\subsection{Partially explicity formulas for the derivatives of transformed coefficients depending on the entropic variables}

Now, we need estimating the combinations $a_0$, $a_1, \ldots, a_{N-1}$ and $d_0$ introduced in \eqref{coeffvarrho}, \eqref{coeffq} and \eqref{d0}. In order to obtain estimates for these functions, we first have to obtain more explicit formula relating them to the quantities $T$, $\tilde{\rho}(T)$, $\widetilde{\varrho u}(T)$, etc.\ of which we know the asymptotic behaviour.

As already shown in the papers \cite{dredrugagu20,druetmixtureincompweak}, in the case of ideal mixtures \eqref{muiideal}, we can achieve partially explicit formula. 
\begin{prop}\label{jerepresente}
 We adopt the assumptions of Lemma \ref{Legendre}. For the derivatives of the function $h^*$, for all $i,k \leq N$, we obtain the representations
\begin{align}
\begin{split}\label{d2upperleft}
 & \partial^2_{w^*_i,w^*_k}h^*(\bar{w}^{*}, \, -1/T) = \rho_i \, \Big( M_i \, \delta_{ik} - \rho_k \, \big(M_i \, \partial_pg_i + M_k \, \partial_p g_k\big) \\
&  \qquad \qquad + \rho_k \, \sum_j (\partial_p g_j)^2 \, \rho_j \, M_j - T \, \rho_k \, \sum_j \partial^2_p g_j \rho_j\Big)  \, ,  
\end{split}
\\
\begin{split}\label{d2upperright}
 & \partial^2_{w^*_i,w^*_{N+1}}h^*(\bar{w}^{*}, \, -1/T) = \rho_i \, \, (M_i \, H_i - \sum_j M_j \, \rho_j \, \partial_p g_j \, H_j)\\
& \qquad + \rho_i \, \varrho H \, (- M_i \, \partial_p g_i + \sum_j M_j \, \tilde{\rho}_j \, (\partial_pg_j)^2 - T \, \sum_j \tilde{\rho}_j \, \partial_p^2g_j) - T^2 \,  \rho_i \, \sum_j \rho_j \, \partial^2_{T,p}g_j \, , 
\end{split}
\\
\begin{split}\label{d2lowerright}
&  \partial^2_{w^*_{N+1}} h^*(w^*,-1/T) =  \sum_i M_i \, \rho_i \, H_i^2 - 2 \, \varrho H \, \sum_iM_i \rho_i \partial_pg_i \, H_i \\
  & +(\varrho H)^2 \, (\sum_{i} M_i \rho_i (\partial^2_pg_i)^2 - T \sum_i \rho_i \, \partial^2_pg_i) - 2 \,\varrho H \, T^2 \, \sum_i \rho_i \,\partial^2_{T,p}g_i - T^3\,  \sum_i \, \rho_i \, \partial^2_{T} g_i \, .\end{split}
\end{align}
\end{prop}
\begin{proof}
The key idea is using the equation \eqref{pressureimplicit} and the representation \eqref{cehatrhoci} of $\rho_i$ in the variables $T$ and $\bar{w}^{*}$.
%
Differentiating in \eqref{pressureimplicit} yields, after using the identity \eqref{xxx}, 
\begin{align}\label{tildepressnachT}
 \frac{1}{T} \, \sum_{i=1}^N M_i \, x_i \, \partial_pg_i(T, \, \hat{p}(T, \, \bar{w}^{*})) \, \partial_{w_j^{\prime}} \hat{p}(T, \, \bar{w}^{*})  = M_j \, x_j \quad \text{ for } \quad j = 1,\ldots,N\, .
\end{align}
Thus, invoking \eqref{tildepressnachT} and \eqref{cehatrhoci}
\begin{align*}
\partial_{w^*_j} \hat{p}(T, \, \bar{w}^{*}) = \frac{T \, M_j \, x_j}{\sum_{i=1}^N M_i \, x_i \, \partial_pg_i(T, \, \hat{p}(T,\bar{w}^{*}))} = T \, \hat{\rho}_j(T, \, \bar{w}^{*}) \, .
\end{align*}
With the same ideas, we also find that
\begin{align*}
 \partial_T \hat{p}(T, \, \bar{w}^{*}) = \frac{1}{T} \, \sum_{i=1}^N \hat{\rho}_i \, (g_i(T, \, \hat{p}(T, \, \bar{w}^{*})) - T \, \partial_T g_i(T, \, \hat{p}(T, \, \bar{w}^{*}))) \, .
\end{align*}
Moreover,
direct calculations yield
\begin{align*}
 \partial_{w_k^{\prime}}\hat{\rho}_i(T, \, \bar{w}^{*}) & = \rho_i\, \Big(M_i \, \delta_{ik} - \rho_k \, \big(M_i \, \partial_pg_i + M_k \, \partial_p g_k\big) \\
 & + \rho_k \, \sum_j (\partial_p g_j)^2 \, \rho_j\, M_j - T \, \rho_k\, \sum_j \partial^2_p g_j \rho_j\Big) \, ,
\end{align*}
in which $g_i$ and its derivatives are evaluated at $(T, \, p(T,\bar{w}^{*}))$.

The expression for $\partial_T \hat{\rho}_i(T, \, \bar{w}^{*})$ is sligthly more complex. 
Again, direct calculations yield
\begin{align*}
& \partial_T \hat{\rho}_i(T, \, \bar{w}^{*}) = \frac{\rho_i}{T^2} \, \, (M_i \, H_i - \sum_j M_j \, \rho_j \, \partial_p g_j \, H_j)\\
& + \frac{\rho_i}{T^2} \, \varrho H \, (- M_i \, \partial_p g_i + \sum_j M_j \, \rho_j \, (\partial_pg_j)^2 - T \, \sum_j \, \rho_i \, \partial_p^2g_j)
- \rho_i \, \sum_j \rho_j \, \partial^2_{T,p}g_j \, .
\end{align*}
Note that $\varrho u = \widehat{\varrho u}(T, \, \bar{w}^{*}) = \sum_{i=1}^N \hat{\rho}_i\, (g_i(T,\hat{p}) - T \, \partial_T g_i(T,\hat{p})) - \hat{p}$. Thus
\begin{align*}
 \partial_T \widehat{\varrho u}(T, \, w^*) =& \sum_{i=1}^N H_i \, \partial_T\hat{\rho}_i + \sum_{i=1}^N \hat{\rho}_i \, (\partial_p g_{i} \, \partial_T \hat{p} - T \, \partial^2_{T,T}g_i - T \, \partial^2_{T,p} g_i \, \partial_T \hat{p}) - \partial_T \hat{p}\\
 =& \sum_{i=1}^N H_i \, \partial_T\hat{\rho}_i - T \,  \sum_{i=1}^N \hat{\rho}_i \,  \partial^2_{T,T}g_i - T \, \sum_i \hat{\rho}_i \, \partial^2_{T,p} g_i \, \partial_T \hat{p} \, .
\end{align*}
Inserting the corresponding identities for $\partial_T \hat{p}$ and $\partial_T \hat{\rho}$, we get
\begin{align*}
&  \partial_T \widehat{\varrho u}(T, \, w^*) = \frac{1}{T^2} \, \sum_i M_i \, \rho_i \, H_i^2 - 2 \frac{\varrho H}{T^2} \, \sum_iM_i \rho_i \partial_pg_i \, H_i \\
   + &\frac{(\varrho H)^2}{T^2} \, (\sum_{i} M_i \rho_i (\partial^2_pg_i)^2 - T \, \sum_i \rho_i \, \partial^2_pg_i) - 2 \,\varrho H\, \sum_i \rho_i \,\partial^2_{T,p}g_i - T \sum_i \, \rho_i \, \partial^2_{T} g_i \, .
\end{align*}

Now, we have defined the entropic variables via
\begin{align*}
 \nabla_{w^*} h^*(w^*) = w =  [\rho, \,\varrho u]  =  [\hat{\rho}(-1/w_{N+1}^*, \, \bar{w}^{*}), \, \widehat{\varrho u}(-1/w_{N+1}^*, \, \bar{w}^{*})] \, .
\end{align*}
This allows to compute
\begin{align*}
 \partial^2_{w^*_i,w^*_k}h^* = & \partial_{w^*_k}\hat{\rho}_i(-1/w_{N+1}^*, \, \bar{w}^{*}) \quad \text{ for } \quad 1 \leq i,k \leq N\\
 \partial^2_{w^*_i,w^*_{N+1}}h^* = & \partial_T \hat{\rho}_i(-1/w_{N+1}^*, \, \bar{w}^{*}) \, \frac{1}{(w^*_{N+1})^2} \quad \text{ for } \quad 1 \leq i \leq N\\
 \partial^2_{w^*_{N+1}}h^* = & - \partial_{T} \widehat{\varrho u}(-1/w_{N+1}^*, \, \bar{w}^{*}) \, \frac{1}{(w^*_{N+1})^2} \, ,
\end{align*}
and we obtain the representations
\eqref{d2upperleft}, \eqref{d2upperright}, and \eqref{d2lowerright}.
\end{proof}
The formula of Proposition \ref{jerepresente} also yield
\begin{align}\label{sumupperleft}
\sum_{k=1}^N \partial^2_{w^*_i,w^*_k}h^*(\bar{w}^{*}, \, -1/T) = & \rho_i \, \Big(M_i - M_i \, \partial_{p}g_i \, \varrho - \sum_k \rho_k \, M_k \, \partial_p g_k \nonumber\\
& + \varrho \, \sum_j (\partial_p g_j)^2 \, \rho_j \, M_j - T \, \varrho \,  \sum_j \partial^2_p g_j \,  \rho_j\Big) \, .
\end{align}
In addition, we have
\begin{align}\label{d211}
D^2_{w^*,w^*}h^* \xi^{N+1} \cdot \xi^{N+1} = \sum_{i,k=1}^N   \partial^2_{w^*_i,w^*_k}h^*
= \sum_{i=1}^N M_i \, \rho_i \,(1-\partial_p g_i \, \varrho)^2 + T \, \varrho^2 \, \sum_{i=1}^N \rho_i \, |\partial^2_pg_i| \, . 
\end{align}
We also compute
\begin{align}\label{sumupperright}
& \sum_{i=1}^N \partial^2_{w^*_i,w^*_{N+1}}h^*(\bar{w}^{*}, \, -1/T) =  
 \sum_{j}  \rho_j \,M_j \, H_j\, (1- \varrho  \, \partial_p g_j)\nonumber\\
 & + \varrho H \, \big(\sum_j \rho_j \, M_j \, \partial_p g_j \, (\varrho \, \partial_p g_j - 1) - T \, \varrho \, \sum_j \rho_j \, \partial_p^2g_j\big)
 - T^2 \, \varrho \, \sum_j \rho_j \, \partial^2_{T,p}g_j \, .
 \end{align}
%
%
%
%
%

\vspace{0.2cm}

\subsection{Some estimates}

We exploit the representation \eqref{d2upperleft} and the asymptotic properties in Lemma \ref{otherasymptotics} and Proposition \ref{ENERGU}. We see that
\begin{align}\label{d22allthebest}
\widetilde{ \partial^2_{w^*_i,w^*_k}}h^*(T) \eqsim \frac{T \, \varrho^2}{\tilde{p}(T)} \, \delta_{iI} \, \delta_{kI} \quad \text{ and } \quad \widetilde{D^2h^*_{w^*,w^*}}(T) \xi^{N+1} \cdot \xi^{N+1} \eqsim \frac{T \, \varrho^2}{\tilde{p}(T)} \, .
\end{align}
We obtain that
\begin{align}\label{d2overd211first}
 \frac{\widetilde{\partial^2_{w^*_i,w^*_k}}h^*(T)}{\widetilde{D^2h^*_{w^*,w^*}}(T) \xi^{N+1} \cdot \xi^{N+1}} \precsim 1 \, .
\end{align}


In order to estimate \eqref{d2upperright}, we use the result of Proposition \ref{ENERGU} for $\widetilde{\varrho H}(T)$. By the same means
\begin{align}\label{jocke}
 \widetilde{\partial^2_{w^*_{N+1}} h^*}(T) \precsim (\widetilde{\varrho H}(T))^2 \eqsim T^{2+2\gamma_{\max}} \, , \quad \widetilde{ \partial^2_{w^*_i,w^*_{N+1}}}h^*(T) \precsim  T^{1+\gamma_{\max}} \, .
\end{align}
It also follows that
%
\begin{align}\label{d2overd211second}
  \frac{| \widetilde{ \partial^2_{w^*_i,w^*_{N+1}}}h^*(T)| }{\widetilde{D^2h^*_{w^*,w^*}}(T) \xi^{N+1} \cdot \xi^{N+1}} \precsim T^{\gamma_{\max}} \, \tilde{p}(T) \eqsim T^{\gamma_{\max} + \frac{\alpha_I}{1-\beta_I}}  \, .
\end{align}
Finally we estimate the coefficients occurring in the weak form of the energy equation.
\begin{lemma}\label{cescoeffsla}
Define $a_0$, $a_k$ and $d_0$ as in \eqref{coeffvarrho}, \eqref{coeffq} and \eqref{d0}. Then
\begin{align*}
|\widetilde{a_0}(T)| \precsim T^{\gamma_{\max} + \frac{\alpha_I}{1-\beta_I}}, \quad |\widetilde{a_k}(T)| \precsim T^{1+\gamma_{\max}}, \quad \widetilde{d_0}(T) \eqsim T^{2+\gamma_{\max}} \, .
\end{align*}
\end{lemma}
\begin{proof}
The estimates for $a$ and $b$ is a direct consequence of the definitions \eqref{coeffvarrho} and \eqref{coeffq}, where we use \eqref{jocke} and \eqref{d2overd211second}. In \eqref{d0equiv}, we proved that
\begin{align*}
 d_0 = T^2 \, \varrho \, c_{\upsilon} + T^4 \, \Big( A^* \partial_{T}(\mu/T) \cdot \partial_{T}(\mu/T) - \frac{ (A^* \partial_{T}(\mu/T) \cdot \xi^{N+1})^2}{A^* \xi^{N+1} \cdot \xi^{N+1}} \Big) \, ,
\end{align*}
in which $A^* = D^2_{\bar{w}^{*},\bar{w}^{*}}h^*$. In the case of $\gamma_{\max} > 0$, we express
\begin{align*}
A^*_{ij} = \delta_{Ii} \, \delta_{Ij} \, \frac{ T \, \varrho^2}{\tilde{p}(T)} + \tilde{B}_{ij}(T) \,  ,
\end{align*}
where, exploiting the result of Lemma \ref{rhoasump}, the coefficients $\tilde{B}_{ij}(T)$ decay exponentially for $T \rightarrow \infty$. Then, for $z \in \mathbb{R}^N$ arbitrary
\begin{align*}
 A^* z \cdot z = a^*_{II} \, z_I^2 \, + \tilde{B} z \cdot z\, , \quad 
 (A^* z \cdot \xi^{N+1})^2 = z_I^2 \, (a^*_{II} )^2 + (\tilde{B} z \cdot \xi^{N+1})^2 \, .
\end{align*}
Hence
\begin{align*}
& A^* z \cdot z - \frac{(A^* z \cdot \xi^{N+1})^2}{A^* \xi^{N+1} \cdot \xi^{N+1}} \\
& \quad =
\frac{a_{II}^* \, z_I^2 \, \tilde{B}z \cdot \xi^{N+1} + \tilde{B} z \cdot z \, a^*_{II} - (\tilde{B}z \cdot \xi^{N+1})^2 - 2 \, a_{II}^* \, z_I \, (\tilde{B}z \cdot \xi^{N+1})}{A^*\xi^{N+1} \cdot \xi^{N+1}} \, ,
\end{align*}
and we see that
\begin{align}\label{goal}
 A^* z \cdot z - \frac{(A^* z \cdot \xi^{N+1})^2}{A^* \xi^{N+1} \cdot \xi^{N+1}} \leq |\tilde{B}| \, \frac{|A^*| \, |z|^3 + |\tilde{B}| \, |z|^2 + 2 \, |A^*| \, |z|^2}{A^* \xi^{N+1} \cdot \xi^{N+1}} \, .
\end{align}
We also note that, for the chemical potentials as functions of $T$ and $\rho$,
\begin{align*}
\partial_{T} (\mu_i/T) = -\frac{1}{T^2} \, (g_i - T \, \partial_{T} g_i - T \, \partial_p g_i \, \partial_{T} \tilde{p}) 
= \frac{1}{T^2} \, H_i - \frac{1}{T} \, \partial_p g_i \, \partial_{T}\tilde{p} \, .
\end{align*}
We recall that 
\begin{align*}
 \partial_T \tilde{p}(T, \, \rho) = - \frac{\sum_{i=1}^N \partial^2_{T,p} g_i(T, \, p) \, \rho_i}{\sum_{i=1}^N \partial^2_{p} g_i(T, \, p) \, \rho_i} \, .
\end{align*}
Since $\partial^2_p g_i = (\beta_i-1) \, p^{-1} \, \partial_p g_i$ and $\partial^2_{T,p} g_i = \alpha_i \, T^{-1} \,  \partial_p g_i$, we have
\begin{align*}
 \partial_p g_i \, \partial_{T}\tilde{p} = \frac{\alpha_i}{1-\beta_i} \, \frac{p}{T} \, \partial_p g_i \, .
\end{align*}
Hence, $T^2 \,  \widetilde{\partial_{T} (\mu_i/T)}(T) \eqsim H_i \precsim T^{1+\gamma_{\max}}$. We choose $z := T^2 \,  \widetilde{\partial_{T} (\mu_i/T)}(T)$ in \eqref{goal}, and we see that
\begin{align*}
T^4 \, \Big( A^* \partial_{T}(\mu/T) \cdot \partial_{T}(\mu/T) - \frac{ (A^* \partial_{T}(\mu/T) \cdot \xi^{N+1})^2}{A^* \xi^{N+1} \cdot \xi^{N+1}} \Big)\end{align*}
decays with $\tilde{B}(T)$ exponentially to zero for $T \rightarrow + \infty$. Thus $d_0 \eqsim T^2 \, \varrho \, \widetilde{c_{\upsilon}}$.

In the case $\gamma_{\max} = 0$, we can directly compute that
\begin{align*}
\partial_{T} (\mu_i/T) & = - \frac{1}{T} \, (c^{i0} - \frac{p_0}{\rho_i^{\rm R} \, T_0}) \, , \\
T^4 \, \Big( A^* \partial_{T}(\mu/T) \cdot \partial_{T}(\mu/T) - \frac{ (A^* \partial_{T}(\mu/T) \cdot \xi^{N+1})^2}{A^* \xi^{N+1} \cdot \xi^{N+1}} \Big) 
= & T^2 \,  \Big(A^* c \cdot c - \frac{ (A^* c \cdot \xi^{N+1})^2}{A^* \xi^{N+1} \cdot \xi^{N+1}} \Big)\, ,\end{align*}
with the constants $c_i = c^{i0} -p_0/(\rho_i^{\rm R} \, T_0)$ of modified heat capacities. Here again, we can write $A^* = a_{II}^* \, \mathbb{I} + \tilde{B}$ where $\tilde{B}(T) \rightarrow 0$ for $T \rightarrow +\infty$ in polynomial decay. Hence $A^* c \cdot c - (A^* c \cdot \xi^{N+1})^2/A^* \xi^{N+1} \cdot \xi^{N+1}$ tends to zero, allowing to conclude again that $d_0 \eqsim \varrho \, T^2 \, \widetilde{c_{\upsilon}}$. 
\end{proof}

\vspace{0.2cm}

\subsection{Sufficient conditions for the maximum principle}

In this last paragraph we show how to derive the growth conditions on the thermodynamic diffusivities and viscosities. We give the statement for $b^1 = \ldots = b^N$ which corresponds to the case that the body forces reduce to the gravitational attraction. For the general case, there is an additional restriction concerning the growth of the functions $M_{ij}(T,\rho)$ in temperature (cf.\ Remark \ref{bnotzero}).
\begin{theo}\label{MAIN2}
We consider the ideal mixture of Section \ref{IDMIX}, based on the choice of $g_1, \ldots,g_N$ according to \eqref{smallgparticular} or \eqref{smallgparticular2}. Moreover, assume that the function $T \mapsto r^{\Gamma}_{\rm h}(x,t, \, T, \, \rho)$ is nonpositive if $T$ is large enough\footnote{More precisely, we assume that there is a function $T_1 \in C(\mathbb{R}_+)$ such that $r^{\Gamma}(x,t, \, T, \, \rho) \leq 0$ for all $T \geq T_1(|\rho|)$.}. We assume that $J^{\rm h} \in W^{1-\frac{1}{r},0}_r(S_{\bar{\tau}})$ with $r > 5$.
We assume that, for some $I \in \{1, \,\ldots,N\}$, the species ${\rm A}_I$ has dominant heat capacity according to the Definition \ref{gammamax}. Let $p > 5$ and assume that the exponents $\gamma = \gamma_{\max}$, $\alpha_I$ and $\beta_I$ are such that
\begin{align*}
\frac{\alpha_I}{1-\beta_I} < 1 + \Big(\frac{2}{5}-\frac{1}{p}\Big) \, (1+\gamma) \, .
\end{align*}
Then we choose any $\beta$ in the interval
\begin{align*}
\beta \in \Big[\max\Big\{1, \, \frac{1}{6} \, \Big(\frac{5\alpha_I}{1-\beta_I} - 1\Big) \Big\}, \, \min\Big\{3 + \frac{5\gamma}{2(1+\gamma)}, \, 3-\frac{5}{p} + \frac{5}{\gamma+1} \, (1 - \frac{\alpha_I}{1-\beta_I})\Big\}\Big[ \, ,
\end{align*}
and any $s_0 \leq s_1$ in the interval
\begin{align*}
\Big[2 \, \beta \, (1+\gamma)-\gamma, \, \min\Big\{\frac{6}{5} \, (1+\beta) \, (1+\gamma), \,  \frac{6}{5} \, (1+\gamma) \, (1+\beta -\frac{5}{3p}) + 2  \, (1 - \frac{\alpha_I}{1-\beta_I}) - \gamma\Big\}\Big[ \, .
\end{align*}
Suppose that for all $0 < m \leq M < +\infty$, the coefficients $\kappa$, $l_1, \ldots, l_N$, $\eta$ and $\lambda$ satisfy the following growth conditions:
%
\begin{align*}
 & \liminf_{T \rightarrow + \infty} \frac{1}{T^{s_0}} \, \inf_{m \leq |\rho| \leq M} \kappa(T, \, \rho) >0, \quad  \limsup_{T \rightarrow + \infty} \frac{1}{T^{s_1}} \, \sup_{m \leq |\rho| \leq M} \kappa(T, \, \rho) < +\infty \, ,\\
& \limsup_{T\rightarrow \infty} \frac{1}{T^{\frac{s_0}{2} + \frac{3}{5} \, (1+\beta) \, (1+\gamma) -1 - \frac{\gamma}{2}}} \sup_{m \leq |\rho| \leq M} |l(T, \, \rho)| < +\infty \, , \\
& \limsup_{T\rightarrow +\infty} \frac{1}{T^{(\frac{6}{5} \, (1+\beta)-1) \, (1+\gamma)}} \, \sup_{m \leq |\rho| \leq M} \Big(\eta(T, \, \rho) +|\lambda(T, \, \rho)| \Big) < +\infty \, , 
\end{align*}
Then, for every solution of optimal mixed regularity on $]0, \, \bar{\tau}[$ subject to the assumptions of Theorem \ref{MAIN} and satisfying $\|\varrho u\|_{L^{1,\infty}(Q_{\bar{\tau}})} < + \infty$, we have
\begin{align*}
\|T\|_{L^{\infty}(Q_{\bar{\tau}})} < + \infty, \quad
 \inf_{(x, \, t) \in Q_{\bar{\tau}}, \, i = 1,\ldots,N} \rho_i(x, \, t) > 0 \, .
\end{align*}
\end{theo}
A small application of this result is the following. Consider the typical choice of a mixture of ideal gases, that is, for all $i$, the function $g_i$ obeys \eqref{smallgparticular2}. Then the heat capacities $c_{\upsilon}^i$ of the species are constant, we have $\gamma_{\max} = 0$, $\alpha_I = 1$ and $\beta_I = 0$. We assume that $\kappa = \kappa_0 \, T^2$ with a constant $\kappa_0>0$. We choose $\beta = 1$ and $s_0 = s_1 = 2$ which is compatible with the conditions in Theorem \ref{MAIN2}. Then, we obtain the maximum principle if the growth of $|M|$, $|l|$ and $\eta$, $|\lambda|$ in $T$ does not exceed $T^{\frac{6}{5}}$.
\begin{proof}
Let $I$ be the index of the dominant species. We assume that $p > 5$, $\gamma_{\max}$, $\alpha_I$ and $\beta_I$ satisfy the conditions of Theorem \ref{MAIN2}.

Then we choose any $\beta$ and $s_0 \leq s_1$ as in the statement of the theorem.
Suppose that $T^{s_0} \precsim \tilde{\kappa}(T) \precsim T^{s_1}$ and the coefficients $l$, $\eta$ and $\lambda$ satisfy
\begin{align*}
|l| \precsim T^{\frac{s_0}{2} + \frac{3}{5} \, (1+\beta) \, (1+\gamma) -1 - \frac{\gamma}{2}} \, , \quad
\eta +|\lambda| \precsim T^{(\frac{6}{5} \, (1+\beta) - 1) \, (1+\gamma_{\max})} \, .
\end{align*} 
We claim that the coefficients $\kappa$, $d_0$, $a_0$, $(a)$, $(l)$, $p$, $\eta$ and $\lambda$ satisfy the requirements of Proposition \ref{TECHON}.

To see this, we at first recall that $\epsilon = \varrho u \eqsim T^{1+\gamma_{\max}}$ (Prop. \ref{ENERGU}), and that $d_0 \eqsim T^{2+\gamma_{\max}}$ (Lemma \ref{cescoeffsla}). The condition \eqref{growth}$_1$ is equivalent to
\begin{align}\label{condkappa1}
\kappa \succsim T^2 \,  T^{2(\beta-1) \, (1+\gamma_{\max}) + \gamma_{\max}} \, .
\end{align}
In order to ensure \eqref{growth}$_2$, we require separately that
\begin{align}\label{condkappa2}
\frac{1}{\kappa} \, |l|^2 \precsim T^{-2-\gamma_{\max}} \, T^{\beta_0 \, (1+\gamma_{\max})}, \quad \kappa \precsim T^{\beta_0 \, (1+\gamma_{\max})} \, .
\end{align}
Then, \eqref{growth}$_2$ follows from the growth rates established in Lemma \ref{cescoeffsla} for $d_0$ and $a_1, \ldots, a_{N-1}$.

Here, $\beta_0$ is subject to $\beta_0 < \frac{6}{5} \, (1+\beta)$. Moreover the condition \eqref{growth}$_3$ for $a_0$ yields
\begin{align}\label{condkappa3}
\kappa \, T^{\gamma_{\max} + 2\frac{\alpha_I}{1-\beta_I}}\precsim T^{\beta_1 \, (1+\gamma_{\max})} \, T^{2} \, ,
\end{align}
with $\beta_1$ subject to $\beta_1 < \frac{6}{5}  \, (1+\beta - \frac{5}{3p})$.

Finally, \eqref{growth}$_4$ and \eqref{growth}$_5$ impose the conditons
\begin{align}\label{moule1}
\frac{\alpha_I}{1-\beta_I} \leq \beta_2 \, (1+\gamma_{\max}) \quad \text{ and } \quad \eta + |\lambda| \precsim  T^{\beta_3 \, (1+\gamma_{\max})}  \, . 
\end{align}
Here $\beta_2$ and $\beta_3$ are subject to
\begin{gather*}
\beta_3+1 <
\frac{6}{5} \, (1+\beta) \, , \quad
\beta_2 < \frac{1}{5} \, (6\beta+1) \, .
 \end{gather*}
The condition \eqref{moule1}$_1$ is independent. With $\delta := \alpha_I/(1-\beta_I)$, it imposes the restriction
\begin{align}\label{moule2}
 \delta \leq \frac{1}{5} \, (6\beta+1) \, (1+\gamma_{\max}) \, .
\end{align}
To satisfy the other conditions, we assume that $\kappa \succsim T^{s_0}$. Then, in view of \eqref{condkappa1}, $s_0$ is subject to
\begin{align*}
s_0 \geq 2 \,  + 2(\beta-1) \, (1+\gamma_{\max}) + \gamma_{\max} \, .
\end{align*}
Moreover, if $\kappa \precsim T^{s_1}$, then  $s_1 \geq s_0$. Due to \eqref{condkappa2}$_2$ and \eqref{condkappa3} $s_1$ is also subject to
\begin{align*}
s_1 \leq \beta_0 \, (1+\gamma_{\max}) < \frac{6}{5} \, (1+\beta) \, (1+\gamma_{\max})\\
s_1 + 2 \, \delta+\gamma_{\max} < 2+ \frac{6}{5}  \, (1+\beta - \frac{5}{3p}) \, (1+\gamma_{\max}) \, .
\end{align*}
We can satisfy these algebraic conditions under the assumptions of the Theorem.
The remaining conditions \eqref{condkappa2}$_1$ and \eqref{moule1}$_2$ then restrict the growth of the coefficients $l$ and $\eta, \, \lambda$ as
\begin{align*}
|l| \precsim T^{\frac{s_0}{2} + \frac{\beta_0}{2} \, (1+\gamma_{\max}) -1 - \frac{\gamma_{\max}}{2}} \, , \quad
\eta +|\lambda| \precsim T^{\beta_3 \, (1+\gamma_{\max})} \, .
\end{align*}
\end{proof}


\newcommand{\etalchar}[1]{$^{#1}$}

\appendix

\section{Estimates for linearised problems}\label{howtolin}

\subsection{The linearised system for the variable $q$}
We want to derive the desired estimate for the linear system \eqref{linearT2}, recalling that $\varrho \in W^{1,1}_{p,\infty}(Q_{\bar{\tau}})$, $q^* \in W^{2,1}_p(Q_{\bar{\tau}}; \, \mathbb{R}^N)$ and $v^* \in W^{2,1}_p(Q_{\bar{\tau}}; \, \mathbb{R}^3)$ are given. In order to reduce this problem to the situation of zero flux boundary conditions considered in the paper \cite{bothedruet}, we at first show how to homogenise the boundary data given in the condition $\nu \cdot \nabla q = g^{\Gamma}$ on $S_{\bar{\tau}}$, with boundary data given in the form of $$g^{\Gamma}(x, \, t) = G(x,t, \, q^*(x,t), \, \varrho(x,t), \, v^*(x,t))\, .$$
Here, the function $G$ is a generalisation of $\tilde{\pi}^{\Gamma}$ that occurs in the original problem. The natural domain of definition of $G$ is the set
$  S_{\bar{\tau}} \times \mathcal{H}^N_- \times \mathbb{R}_+ \times \mathbb{R}^3$. 
In order to state the assumptions on $G$, we also introduce for parameters $0 < m < M < +\infty$ and $K_0, \, \bar{\theta} > 0$ the convex sets
\begin{align*}
 & E(m, \, M, \, K_0, \, \bar{\theta}) \\
 & \qquad := \{(z, \, r, \, w) \in \mathcal{H}^N_- \times \mathbb{R}_+ \times \mathbb{R}^3 \, : \, |z| + |w| \leq K_0, \, m \leq r \leq M, \quad z_N \leq -1/\bar{\theta}\} \, .
\end{align*}
\begin{lemma}
Let $(q^*, \varrho, \, v^*) \in \mathcal{X}_{\bar{\tau},+}$.
For $t \geq 0$, we define $m^*(t) = \inf_{Q_t} \varrho$, $M^*(t) := \sup_{Q_t} \varrho$ and $\theta^*(t) := -1/\sup_{Q_t} q_N$. For $(x,t) \in Q_{\bar{\tau}}$, we define 
\begin{align*}
 g^{\Gamma}(x,t) := G\Big(x,t, \, \big(q^*, \varrho, \, v^*\big)(x,t)\Big) \, ,
\end{align*}
where the function $G$ satisfies the follwowing assumptions:
\begin{enumerate} 
 \item For all $(z, \, r, \, w) \in \mathcal{H}^{N}_- \times \mathbb{R}_+ \times \mathbb{R}^3$, the function $(x,t) \mapsto G(x, \, t, \, (z, \, r, \, w))$ belongs to $C^{\lambda,\frac{\lambda}{2}}(S_{\bar{\tau}})$ with $\lambda > 1-1/p$. For all $0 < m < M < +\infty$ and $K_0, \, \bar{\theta} > 0$ we have
\begin{align*}
\sup_{(z,\,r,\, w) \in E(m, \, M,\, K_0, \, \bar{\theta})} \|G(\cdot, (z,\,r,\, w))\|_{C^{\lambda,\frac{\lambda}{2}}_p(S_{\bar{\tau}})} < +\infty \, .
\end{align*}

 \item For all $(x,t) \in S_{\bar{\tau}} $, the map $(z, \, r, \, w) \mapsto G(x, \, t, \, (z, \, r, \, w))$ is of class $C^1( \mathcal{H}^{N}_- \times \mathbb{R}_+ \times \mathbb{R}^3)$, and $\|G\|_{L^{\infty}(S_{\bar{\tau}}; \, C^1(E(m, \, M, \, K_0, \, \bar{\theta})))} < +\infty$.
\end{enumerate}
Then, there exists a continuous function $\Psi = \Psi(t, \, b_1, \ldots , \, b_4)$ on $[0, \, +\infty[^5$ such that
 \begin{align*} 
 \|g^{\Gamma}\|_{W^{1-\frac{1}{p},\frac{1}{2}-\frac{1}{2p}}_p(S_{t}; \, \mathbb{R}^N)} \leq \Psi(t, \, (m^*(t))^{-1}, \,  M^*(t), \, \theta^*(t), \, \|(q^*, \, \varrho, \, v^*)\|_{\mathcal{X}_t})  \, .
 \end{align*}
In addition, $\Psi$ is nondecreasing in all arguments and $\Psi(0, \, b) \equiv 0$ for all $ b \in \mathbb{R}_+^4$.
\end{lemma}
\begin{proof}
For arbitrary $u \in W^{1,0}_{p,\infty}(Q_t)$ we have
$\|D_{x} u\|_{L^{p}(Q_{t})} \leq   t^{\frac{1}{p}} \, \|D_x u\|_{L^{p,\infty}(Q_t)}$, hence
\begin{align*}
\|u\|_{L^p(0,t; \, W^{1-\frac{1}{p}}_{p}(\partial \Omega))} \leq c_0 \, \|u\|_{W^{1,0}_p(Q_t)} \leq c_0 \, t^{\frac{1}{p}} \, \|u\|_{W^{1,0}_{p,\infty}(Q_t)}  \, .
\end{align*}
Moreover, if $u \in W^{1}_p(Q_t)$ we also have
\begin{align*}
\|u\|_{L^p(\partial {\Omega}; \, W^{\frac{1}{2}-\frac{1}{2p}}_p(0,t))} \leq & t^{\frac{1}{2} -\frac{1}{2p}} \, \|u\|_{L^p(\partial {\Omega}; \, W^{1-\frac{1}{p}}_p(0,t))}
\leq  t^{\frac{1}{2} -\frac{1}{2p}} \, \|u\|_{W^{1-\frac{1}{p}}_p(S_t)}\\
\leq & c_1 \, t^{\frac{1}{2} -\frac{1}{2p}} \, \|u\|_{W^{1}_p(Q_t)}
\end{align*}
Hence
\begin{align*}
\|u\|_{W^{1-\frac{1}{p},\frac{1}{2}-\frac{1}{2p}}(S_t)} \leq c \, \max\{t^{\frac{1}{p}}, \, t^{\frac{1}{2}-\frac{1}{2p}}\} \, (\|u\|_{W^1_p(Q_t)} + \|D_x u\|_{L^{p,\infty}(Q_t)}) \, .
\end{align*}
We apply this inequality to $\varrho$ and the components of $(q^*, \, v^*)$. Clearly, it follows that
\begin{align}\label{menschenea}
 \|(q^*, \, \varrho, \, v^*)\|_{W^{1-\frac{1}{p},\frac{1}{2}-\frac{1}{2p}}(S_t)} \leq c \, \max\{t^{\frac{1}{p}}, \, t^{\frac{1}{2}-\frac{1}{2p}}\} \, \|(q^*, \, \varrho, \, v^*)\|_{\mathcal{X}_t}\, .
\end{align}
Now, for $(x,\tau)$ and $(y,s)$ arbitrary on $S_t$, we have, with $E^*(t) := E(m^*(t),  \, M^*(t), \, \|q^*\|_{L^{\infty}(Q_t)} + \|v^*\|_{L^{\infty}(Q_t)}, \, \theta^*(t))$ 
\begin{align*} 
 & |G(x,\tau, \, (q^*, \, \varrho, \, v^*)(x,\tau)) - G(y,s, \, (q^*, \, \varrho, \, v^*)(y,s))| \\
 & \leq   |G(x,\tau, \, (q^*, \, \varrho, \, v^*)(x,\tau)) - G(y,s, \, (q^*, \, \varrho, \, v^*)(x,\tau))| \\ & \quad + |G(y,s, \, (q^*, \, \varrho, \, v^*)(x,\tau)) - G(y,s, \, (q^*, \, \varrho, \, v^*)(y,s))|\\
 & \leq \|G(\cdot, \, (q^*, \, \varrho, \, v^*)(x,\tau))\|_{C^{\lambda,\frac{\lambda}{2}}(S_t)} \, (|x-y|^{\lambda} + |\tau-s|^{\frac{\lambda}{2}})\\
 & \quad + \|DG(y,s, \cdot)\|_{L^{\infty}(E^*(t))} \, |(q^*, \, \varrho, \, v^*)(x,\tau)-(q^*, \, \varrho, \, v^*)(y,s)|\\
 & \leq \sup_{(z,r,w) \in E^*(t)} \, \|G(\cdot, \, (z,r,w))\|_{C^{\lambda,\frac{\lambda}{2}}(S_t)} \, (|x-y|^{\lambda} + |\tau-s|^{\frac{\lambda}{2}}) \\
 & \quad + \sup_{(y,s) \in S_t} \|DG(y,s, \cdot)\|_{L^{\infty}(E^*(t))} \, |(q^*, \, \varrho, \, v^*)(x,\tau)-(q^*, \, \varrho, \, v^*)(y,s)|  \, .
\end{align*}
In particular, it follows that
\begin{align*} 
 [g^{\Gamma}]^{(1-\frac{1}{p})}_{p,x,S_t} \leq &  c \, t^{\frac{1}{p}} \, \sup_{(z,r,w) \in E^*(t)} \, \|G(\cdot, \, (z,r,w))\|_{C^{\lambda,0}(S_t)} \\
 & + \sup_{(y,s) \in S_t} \|DG(y,s, \cdot)\|_{L^{\infty}(E^*((t))} \,  [(q^*, \, \varrho, \, v^*)]^{(1-\frac{1}{p})}_{p,x,S_t} \, ,
 \end{align*}
 and also that
 \begin{align*}
 [g^{\Gamma}]^{(\frac{1}{2} - \frac{1}{2p})}_{p,t,S_t} \leq & c \,t^{\frac{1}{2} \, (\lambda +\frac{1}{p} -1)} \, \sup_{(z,r,w) \in E^*(t)} \, \|G(\cdot, \, (z,r,w))\|_{C^{0,\frac{\lambda}{2}}(S_t)} \\
 & + \sup_{(y,s) \in S_t} \|DG(y,s, \cdot)\|_{L^{\infty}(E^*)} \,  [(q^*, \, \varrho, \, v^*)]^{(\frac{1}{2}-\frac{1}{2p})}_{p,t,S_t} \, . 
\end{align*}
Invoking \eqref{menschenea}, we are done.
\end{proof}
Consider now the parabolic system \eqref{linearT2}. Using the preceding Lemma, we have
\begin{align*}
 \|\tilde{\pi}^{\Gamma}(\cdot, \, \varrho(\cdot), \, q^*(\cdot))\|_{W^{1-\frac{1}{p},\frac{1}{2}-\frac{1}{2p}}_p(S_{\bar{\tau}}; \, \mathbb{R}^N)} \leq \Psi_t^{\Gamma} \, ,
\end{align*}
with a factor $\Psi^{\Gamma}_t = \Psi^{\Gamma}(t, \, \vec{b})$ of the desired structure. For $k = 1,\ldots,N$, we solve the problem
\begin{align}\label{parabole0}
\begin{split}
& \partial_t s_k - \Delta s_{k} =  0 \quad \text{ in } Q_{\bar{\tau}} \, , \\
\nu(x)\cdot \nabla s_k = & \tilde{\pi}^{\Gamma}_k(x,t, \, \varrho, \, q^*) \, \quad \text{ on } S_{\bar{\tau}} \, , \qquad  
s_k(x, \, 0) =  q_k^0(x) \quad \text{ in } \Omega \, .
\end{split} 
\end{align}
Then, $s = (s_1, \ldots,s_N)$ satisfies (see \cite{ladu}, Th.\ 9.1 and page 351, or \cite{denkhieberpruess}, Th.\ 2.1)
\begin{align}\begin{split}\label{pertrand}
& \|s_k\|_{W^{2,1}_p(Q_{t}; \, \mathbb{R}^N)} + \sup_{\tau \leq t} \|s_k(\cdot, \tau)\|_{W^{2-\frac{2}{p}}_p(\Omega; \, \mathbb{R}^N)} \\
& \leq C_1 \, (\|q^0\|_{W^{2-\frac{2}{p}}(\Omega)}  + \|\tilde{\pi}^{\Gamma}\|_{W^{1-\frac{1}{p},\frac{1}{2}-\frac{1}{2p}}_p(S_{t})}) =: \Psi_{t} \, .
\end{split}
\end{align}
We remark that $q$ solves the equations \eqref{linearT2} if and only if $\tilde{q} := q - s$ solves
\begin{align}\label{parabole1}
\begin{split}
& R_{q}(\varrho, \,q^*) \, \partial_t \tilde{q} - \divv (\widetilde{\mathcal{M}}(\varrho, \, q^*) \, \nabla \tilde{q})  =  \tilde{g}(x, \, t, \, q^*, \, \varrho,\, v^*, \, \nabla q^*, \, \nabla \varrho, \, \nabla v^*) \, ,\quad \text{ in } Q_{\bar{\tau}} \, \\
& \nu(x)\cdot \nabla \tilde{q}=  0\, \quad \text{ on } S_{\bar{\tau}} \, , \qquad  
\tilde{q}(x, \, 0) = 0 \quad \text{ in } \Omega \, ,
\end{split} \, .
\end{align}
With $g$ from \eqref{linearT2}, we here have
\begin{align*}
 \tilde{g} := g - R_{q}(\varrho, \,q^*) \, \partial_t s + \divv (\widetilde{\mathcal{M}}(\varrho, \, q^*) \, \nabla s) \, .
\end{align*}
It is readily shown that the norm of $\tilde{g}$ in $L^p(Q_t)$ can be controlled by the norm of $g$ and the norm of $s$ occurring in \eqref{pertrand}.
For the problem \eqref{parabole1}, we can apply the Proposition 7.1 of \cite{bothedruet}.

\subsection{Linearised parabolic system for the variable $v$}

At second, we consider the parabolic system \eqref{linearT3} with given right-hand side
\begin{align}\label{parabvlinear}
\begin{split}
& \varrho \, \partial_t v - \divv \mathbb{S}(\varrho, \, q^*, \, \nabla v) =  f(x, \,t)  \, ,\quad \text{ in } Q_{\bar{\tau}} \quad \\
& v =  0 \, \quad \text{ on } S_{\bar{\tau}} \, , \quad  v(x, \, 0) = v^0(x) \, , \quad \text{ in } \Omega \, .
\end{split}
\end{align}
We recall that the meaning of $\mathbb{S}(\varrho, \, q^*, \, \nabla v)$ was given in \eqref{stressnew}, that is, we here assume that the coefficients $\eta$ and $\lambda$ are $C^2$ functions of their arguments.

In the present case, the data $f$ and $v^0$ are subject to
\begin{align*}
f \in L^p(Q_{\bar{\tau}}; \, \mathbb{R}^3),  \quad v^0 \in W^{2-\frac{2}{p}}(\Omega; \, \mathbb{R}^3) \, ,
\end{align*}
and to the compatibility condition $v^0 = 0$ on $\partial \Omega$.
\begin{lemma}
If $v \in W^{2,1}_p(Q_{\bar{\tau}})$ is a solution to \eqref{linearT3}, then
\begin{align*}
\|v\|_{W^{2,1}_p(Q_{\bar{\tau}})} + \sup_{\tau \leq \bar{\tau}} \|v\|_{W^{2-\frac{2}{p}}_p(\Omega)} \leq C_2 \, (1+[\varrho]_{C^{\beta,\frac{\beta}{2}}} + [q^*]_{C^{\beta,\frac{\beta}{2}}})^{\frac{2}{\beta}} \,  (\|v^0\|_{W^{2-\frac{2}{p}}_p(\Omega)} + \|f\|_{L^p(Q_t)}) \, .
\end{align*}
\end{lemma}
\begin{proof}
Here the difference to Prop.\ 7.5 of \cite{bothedruet} is the state-dependence of the viscosity functions. We deal with this case with the same localisation technique applied in \cite{bothedruet}, Appendix B.
At first we rewrite the PDEs in the form
\begin{align*}
& \varrho \, \partial_t v - \eta \, \Delta v - (\lambda+\eta) \, \nabla \divv v =   f - 2 \, \nabla \eta \cdot (\nabla v)_{\text{sym}} - \nabla \lambda \, \divv v\\
= & f -2\,  (\eta_{\varrho} \, \nabla \varrho +\eta_{q} \cdot \nabla q^*) \cdot (\nabla v)_{\text{sym}} - (\lambda_{\varrho} \, \nabla \varrho + \lambda_q \cdot \nabla q^*)  \, \divv v =: \tilde{f}\, .
\end{align*}
Let $(x^0, \, t^0) \in Q_t$ and $r > 0$. We consider a (small) volume $$Q_r^0 = Q_r(x^0, \, t^0) := \{(x,t) \in \mathbb{R}^4 \, : \, |x-x^0| \leq \sqrt{r}, \quad |t-t^0| \leq r\} \, .$$ We can localise the boundary value problem \eqref{parabvlinear} in $Q_r^0$ by choosing a cutoff function $\zeta \in C^{\infty}_c(Q_r^0)$ and satisfies the condition $|\partial_t \zeta| + |\partial^2_{x} \zeta| \leq c_0 \, r^{-1}$. Then, $w := \zeta \, v$ satisfies
\begin{align*}
\begin{split}
& \varrho \, \partial_t w - \eta \, \Delta w - (\lambda+\eta) \, \nabla \divv w =  \hat{f}  \, ,\quad \text{ in } Q_{t}  \quad \\
& w =  0 \, \quad \text{ on } S_t\, , \quad  w(x, \, 0) = v^0(x) \, \eta(x, \, 0) \, , \quad \text{ in } \Omega   \, .
\end{split}
\end{align*}
Here
\begin{align*}
\hat{f} = \zeta \, \tilde{f} + \varrho \, \partial_t \zeta \, v + \eta \, (\nabla \zeta \cdot \nabla v + \divv( v \, \nabla \zeta)) + (\lambda+\eta) \, (\divv v \, \nabla \zeta + \nabla v \cdot \nabla \zeta + D^2\zeta \, v) \, .
\end{align*}
Next, we write the PDEs in the form
\begin{align*}
& \varrho(x^0,t^0) \, \partial_t w - \divv \mathbb{S}(\varrho(x^0,t^0), \, q^*(x^0,t^0), \, \nabla w) =  \hat{f} \\
& \qquad 
+ [\varrho(x^0,t^0) -\varrho] \, \partial_t w+  [\eta(\varrho, \, q^*)-\eta(\varrho(x^0,t^0), \, q^*(x^0,t^0))] \, \Delta w\\
& 
\qquad + [(\lambda+\eta)(\varrho, \, q^*)-(\lambda+\eta)(\varrho(x^0,t^0), \, q^*(x^0,t^0))] \, \nabla \divv w \, .
\end{align*}
Recalling that $w$ is localised in $Q^0_r$, the perturbations on the right-hand sides satisfy estimates
\begin{align*}
& | [\varrho(x^0,t^0) -\varrho] \, \partial_t w| \leq [\varrho]_{C^{\beta,\frac{\beta}{2}}(Q_t)} \, r^{\frac{\beta}{2}} \, |\partial_t w| \, ,\\
& |\eta(\varrho, \, q^*)-\eta(\varrho(x^0,t^0), \, q^*(x^0,t^0))| \, |\Delta w| \leq [\eta]_{C^{1}(E^*_t)} \, ([\varrho]_{C^{\beta,\frac{\beta}{2}}(Q_t)} + [q^*]_{C^{\beta,\frac{\beta}{2}}(Q_t)}) \, r^{\frac{\beta}{2}} \, |D^2w| \, ,  
\end{align*}
and the same for the grad-div term. Here $E^*_t$ is the range over $Q_t$ of the variables $(\varrho, \, q^*)$ in $\mathbb{R}_+ \times \mathcal{H}^N_-$. We obtain an estimate
\begin{align*}
\|w\|_{W^{2,1}_p(Q_t)} &\leq c \, (\|w(\cdot,0)\|_{W^{2-\frac{2}{p}}_p(\Omega)} + \|\hat{f}\|_{L^p(Q_t)})\\
& + c \, (1+[\eta]_{C^{1}(E^*_t)} + [\lambda]_{C^{1}(E^*_t)}) \, ([\varrho]_{C^{\beta,\frac{\beta}{2}}(Q_t)} + [q^*]_{C^{\beta,\frac{\beta}{2}}(Q_t)}) \, r^{\frac{\beta}{2}} \, \|w\|_{W^{2,1}_p(Q_t)} \, .
\end{align*}
Here $c$ is the norm of the inverse of the right-hand side operator, hence it depends only on $t$, $\varrho(x^0,t^0)$ and the values of the viscosity coefficients at $\varrho(x^0,t^0), \, q^*(x^0,t^0)$. Choosing $r$ in such a way that
\begin{align*}
 c \, (1+[\eta]_{C^{1}(E^*_t)} + [\lambda]_{C^{1}(E^*_t)}) \, ([\varrho]_{C^{\beta,\frac{\beta}{2}}(Q_t)} + [q^*]_{C^{\beta,\frac{\beta}{2}}(Q_t)}) \, r^{\frac{\beta}{2}} = \frac{1}{2} \, ,
\end{align*}
we obtain that
\begin{align*}
\|w\|_{W^{2,1}_p(Q_t)} \leq 2c \, (\|w(\cdot,0)\|_{W^{2-\frac{2}{p}}_p(\Omega)} + \|\hat{f}\|_{L^p(Q_t)}) \, .
\end{align*}
Now, it remains to estimate the norm of the right-hand side. By appropriate covering of $Q_t$ with sets of the structure of $Q^0_r$, we can sum up these estimates to the result. We refer to the Appendix B of \cite{bothedruet} for details. 
\end{proof}

\section{Two technical statements for the maximum principle}

We denote by $\lambda_3$ the three-dim.\ Lebesque measure. 
We begin with a short preliminary remark. Consider a Lipschitz domain $\Omega$ in $\mathbb{R}^3$. Then, there exists $c_0 = c_0(\Omega)$ such that $ \|u\|_{L^2(\Omega)} \leq c_0 \, \|\nabla u\|_{L^2(\Omega)}$ for all $u \in W^{1,2}(\Omega)$ with the property $$\lambda_3(\{x \in \Omega \, : \, |u(x)| = 0\}) \geq \lambda_3(\Omega)/2 \, .$$
This can be proved as usual by showing that the negation yields a contradiction.
\begin{lemma}\label{interpomoi}
Let $u \in C^1(\overline{Q_{t}})$ and $\beta > 0$. There is $c_0 = c_0(\Omega)$ such that, for all $k >2 \,  \|u\|_{L^{1,\infty}(Q_t)}/\lambda_3(\Omega)$, the function $w_k = \max\{u - k, \, 0\}$ satisfies 
\begin{align*}
 \|w_k\|_{L^r(Q_t)}^r \leq c_0 \, \|w_k\|_{L^{2,\infty}(Q_t)}^{\frac{4}{3}} \, \|\nabla w_k^{\beta}\|_{L^2(Q_t)}^2, \quad r = \frac{4}{3} + 2 \, \beta \, ,
 \end{align*}
\end{lemma}
\begin{proof}
For $\tau \in ]0,t[$ arbitrary, and for $k$ as in the assumptions, we have
\begin{align*}
\lambda_3(\{x \, : \, u(x,\tau) > k\}) \leq \frac{\|u\|_{L^{1,\infty}(Q_t)}}{k} \leq \frac{\lambda_3(\Omega)}{2} \, .
\end{align*}
Hence, $\lambda_3(\{x \, : \, w_k(x, \, \tau) = 0\}) \geq \lambda_3(\Omega)/2$. 

Moreover $\|w_k(\cdot,\tau)\|_{L^r(\Omega)}^r = \int_{\Omega} |w_k|^{\frac{4}{3}} \, |w_k|^{2\beta} \, dx$. Use of H\"older's inequality yields $$\|w_k(\cdot,\tau)\|_{L^r(\Omega)}^r \leq \|w_k(\cdot,\tau)\|_{L^2(\Omega)}^{\frac{4}{3}} \, \|w_k^{\beta}(\cdot,\tau)\|_{L^6(\Omega)}^2 \, .$$ Since $W^{1,2}(\Omega) \subset L^6(\Omega)$ with continuous embedding, we get
\begin{align*}
\|w_k^{\beta}(\cdot,\tau)\|_{L^6(\Omega)}^2 \leq c_1 \, \|w_k^{\beta}(\cdot,\tau)\|_{W^{1,2}(\Omega)}^2 \leq  c_1 \, c_2 \, \|\nabla w_k^{\beta}(\cdot,\tau)\|_{L^{2}(\Omega)}^2 \, ,
\end{align*}
where we also use the preliminary consideration. The claim follows.
\end{proof}

\begin{lemma}\label{lemboundmax}
 Let $u \in L^1(Q_t)$. For $k \in \mathbb{R}$, we denote $Q_{t,k} := \{(x,\tau) \in Q_t \, : \, u(x,\tau) > k\}$. Assume that for all $k \geq k_1$, the function $w_k = \max\{u - k, \, 0\}$ satisfies the inequality
  \begin{align*}
    \|w_k\|_{L^1(Q_t)} \leq  \sum_{i=0}^n  \Big(A_i \, |Q_{t,k}|^{\sigma_i} +B_i \, k^{\omega_i} \,|Q_{t,k}|^{y}\Big) \, ,
  \end{align*}
  with $n \in \mathbb{N}$, $y > 1$ and, for $i = 1,\ldots,n$, constants $1 < \sigma_i$, $0 \leq \omega_i \leq y$ and $A_i, \, B_i \geq 0$. Then
  \begin{align*}
   \sup_{Q_t} u \leq k_1 + C(t, \, |\Omega|, \, k_1, \, |A|_1, \,  |B|_1, \, \|u\|_{L^1(Q_t)}) \, ,
  \end{align*}
  where $C$ is a positive continuous function.
\end{lemma}
\begin{proof}
Let $f(k) :=  \|w_k\|_{L^1(Q_t)}/|Q_t|$. Then $f$ is differentiable at every $k$ in the classical sense, and $- f^{\prime}(k) = |Q_{t,k}|/|Q_t|\leq 1$. Hence $f \in C^{0,1}(\mathbb{R})$. Using the assumptions, for all $k \geq k_1$, we can show that
\begin{align*}
|Q_t| \, f(k) \leq  \left(\sum_{i=0}^n A_i \, |Q_t|^{\sigma_i}\right) \, (-f^{\prime}(k))^{\sigma_{\min}} + \left(\sum_{i=0}^n B_i \, k_1^{\omega_i} \right) \, \left(\frac{k}{k_1}\right)^{\omega_{\max}}\, |Q_t|^y \, (-f^{\prime}(k))^{y} \, .
\end{align*}
From this we deduce that
\begin{align*}
&-f^{\prime}(k) \geq  \min\left\{\left(\frac{f(k)}{2\sum_{i=0}^n A_i \, |Q_t|^{\sigma_i-1}}\right)^{\frac{1}{\sigma_{\min}}}, \, \left(\frac{f(k)}{2(k/k_1)^{\omega_{\max}} \, |Q_t|^{y-1} \, \sum_i B_i \, k_1^{\omega_i}}\right)^{\frac{1}{y}}\right\}  \\
\geq & \frac{\min\{(|Q_t| \, f(k))^{\frac{1}{\sigma_{\min}}}, \, (|Q_t| \, f(k))^{\frac{1}{y}}\}}{2^{\alpha}\max\{\left(\sum_{i=0}^n A_i \, |Q_t|^{\sigma_i}\right)^{\frac{1}{\sigma_{\min}}}, \, |Q_t| \, \left(\sum_{i=0}^n B_i \, k_1^{\omega_i}\right)^{\frac{1}{y}}\}} \, \left(\frac{k_1}{k}\right)^{\frac{\omega_{\max}}{y}} \, ,
\end{align*}
with $\alpha := \max\{1/y, \, 1/\sigma_{\min}\} < 1$. Since $|Q_t| \, f(k)/\|u\|_{L^1(Q_t)} \leq 1$, with $\nu := \omega_{\max}/y \leq 1$ we also obtain that
\begin{align*}
-f^{\prime}(k) \geq \frac{k_1^{\nu} \, \min\{\|u\|_{L^1(Q_t)}^{\frac{1}{\sigma_{\min}}-\alpha}, \, \|u\|_{L^1(Q_t)}^{\frac{1}{y}-\alpha}\}}{2^{\alpha}\max\{\left(\sum_{i=0}^n A_i \, |Q_t|^{\sigma_i}\right)^{\frac{1}{\sigma_{\min}}}, \, |Q_t| \, \left(\sum_{i=0}^n B_i \, k_1^{\omega_i}\right)^{\frac{1}{y}}\}}  \, |Q_t|^{\alpha} \, \frac{(f(k))^{\alpha}}{k^{\nu}} \, .
\end{align*}
We are now in the situation of Lemma 5.1 in \cite{laduellipt}. If $\nu < 1$, we get a bound
\begin{align*}
& \sup_{Q_t} u \leq  k_1 + c(\alpha,\nu)  \, \Sigma^{\frac{1}{1-\nu}} \, \left(\frac{\|u\|_{L^1(Q_t)}}{|Q_t|}\right)^{\frac{1-\alpha}{1-\nu}} \, ,\\ 
\Sigma := &  \|u\|_{L^1(Q_t)}^{|\frac{1}{y}-\frac{1}{\sigma_{\min}}|} \, \frac{2^{\alpha}\max\{\left(\sum_{i=0}^n A_i \, |Q_t|^{\sigma_i}\right)^{\frac{1}{\sigma_{\min}}}, \, |Q_t| \, \left(\sum_{i=0}^n B_i \, k_1^{\omega_i}\right)^{\frac{1}{y}}\}}{k_1^{\nu} \, |Q_t|^{\alpha}}\, .
\end{align*}
If $\nu = 1$, then with the same $\Sigma$ the bound takes the form
\begin{align*}
\sup_{Q_t} u \leq k_1 \, \exp\Big(\frac{\Sigma}{1-\alpha} \, \Big(\frac{\|u\|_{L^1(Q_t)}}{|Q_t|}\Big)^{1-\alpha} \Big) \, .
\end{align*}
\end{proof}

\section{Analytical investigations on fluid mixtures}\label{discussion}

We started with the PDEs \eqref{mass}, \eqref{energy}, \eqref{momentum} of multicomponent fluid dynamics; This is the Navier--Stokes--Fick--Onsager--Fourier system.

In other investigations devoted to mixtures, the modelling seems to be different at first sight. 
In many mixture theories indeed, the varying composition is handled by means of the volume fractions. 

To understand this viewpoint, let us  call a mixture of $N$ constituents \emph{volume-additive} if for every control volume $\mathcal{V}$ in the fluid, the equation
\begin{align}\label{VA}
|\mathcal{V}| = \sum_{i=1}^N m_i(\mathcal{V}) \, \hat{\upsilon}_i \, 
\end{align}
is valid, with $m_i(\mathcal{V})$ being the total mass of ${\rm A}_i$ available in $\mathcal{V}$, and $\hat{\upsilon}_i = 1/\hat{\rho}_i$ denoting the specific volume of ${\rm A}_i$ \emph{as pure substance}. For instance, ideal mixtures are volume-additive in this sense (\cite{brdicka}, section 4.1). The mixtures characterised by the splitting \eqref{muiideal} of the chemical potentials are also a subclass of the volume-additive mixtures (cp.\ \eqref{implicitpress})\footnote{For mixtures which are not volume-additive, the equation \eqref{VA} can be postulated too, but the specific volumes $\hat{\upsilon}^i$ are not independent on the molecular composition.}.

The specific volume/the mass density of a constituent are functions of temperature and pressure. 
In the continuum limit, the definition \eqref{VA} implies that the equation
\begin{align}\label{EOSVA}
\sum_{i=1}^N \frac{\rho_i}{\hat{\rho}_i(T, \, p)} = 1
\end{align}
is valid at every point of a volume-additive mixture. In fact, \eqref{EOSVA} is the \emph{equation of state} of any volume-additive mixture. 

The quantities $\varphi_i := \rho_i/\hat{\rho}_i(T, \, p)$ sum up to one at every point and $\varphi_1, \ldots,\varphi_N$ are called the \emph{volume fractions}. \\[0.5ex]

To determine the volume fractions, one is faced with the difficulty that, up to the equation \eqref{EOSVA} -- which, depending on the properties of the mixture, we might postulate or not -- conservation principles are not formulated for the volume, but for the mass. For this reason, the volume-fractions are not always subject to similar conservation laws like the partial mass densities. If we account for the fact that $\hat{\rho}^i$ remains dependent on $T$ and $p$ in the mixture, dividing in \eqref{mass} with $\hat{\rho}_i(T,p)$ does not yield a divergence-form equation for $\varphi_i$ except, of course, for isothermal and isobaric systems.

However, another special case occurs if the involved constituents are all perfectly incompressible. 
In this case $\partial_p \hat{\rho}_i = 0$ for each ${\rm A}_i$ by definition and, due to the inequality (See \cite{mueller} or \cite{bothedreyerdruet}, (36))
\begin{align*}
(\partial_T \hat \upsilon_i)^2\leq - \frac{c_p^i}{T} \, \partial_p\hat \upsilon_i \, ,
\end{align*}
also $\partial_T \hat{\rho}_i = 0$. Hence $\hat{\rho}_i = \hat{\rho}_i^{\rm R}$ with the constant reference bulk mass density of the constituent ${\rm A}_i$. Then, \eqref{EOSVA} reduces to a \emph{constraint} for the partial mass densities:
\begin{align}\label{EOSVAINCOMP}
\sum_{i=1}^N \frac{\rho_i}{\rho_i^{\rm R}} = 1 \, ,
\end{align}
expressing the equation of state of a \emph{volume-additive mixture of incompressible constituents}\footnote{In the paper \cite{josef}, the same type of mixture is called ''simple mixture''. In \cite{bothedreyer}, Section 15, ''simple mixture'' refers to a different concept. A mixture is simple if the internal energy density and the pressure can be additively constructed from the corresponding partial constitutive quantities of the constituents.}. In this case, the volume fractions of the constituents satisfy conservation laws: We divide in \eqref{mass} with $\rho_i^{\rm R}$, yielding with $\varphi_i := \rho_i/\rho_i^{\rm R}$,
\begin{align}\label{volcon}
\partial_t \varphi_i +\divv((\rho_i^{\rm R} )^{-1} \, \jmath^i) = (\rho_i^{\rm R} )^{-1} \, r_i \, .
\end{align}
This form of the equations of mass conservation is used for instance in \cite{josef}, \cite{ottoweinan}, \cite{boyer}, \cite{abelsgarckegruen} and many more. However, note that other papers like \cite{lowe} on Cahn--Hilliard regularisations of two-phase binary mixtures do not rely on the assumption that the mass densities of the constituents are constant\footnote{In \cite{lowe}, the partial mass density is called ''apparent mass density'' and the mass density of the constituent is called ''actual mass density''.}.

We sum up the equations \eqref{volcon} and, in the absence of chemical reactions, we get 
\begin{align*}
\divv \Big(\sum_{i=1}^N (\rho_i^{\rm R})^{-1} \, \jmath^i \Big) = 0 \, .
\end{align*}
Hence we recover the statement of \cite{boyer} that the volume-averaged velocity $$v^{\rm vol} :=\sum_{i=1}^N (\rho_i^{\rm R})^{-1} \, \jmath^i$$
is divergence-free. This is however only true in a volume-additive mixture of incompressible fluids where there are, moreover, no chemical reactions.

In order to work with the equation \eqref{volcon} while preserving $\sum_{i=1}^N \varphi_i = 1$, several approaches have been tried. For instance, in \cite{ottoweinan}, two modelling possibilities are discussed in order to guarantee that the net flux $\sum_{i=1}^N (\rho_i^{\rm R} )^{-1} \, \jmath^i$ -- or at least its divergence -- vanish. In this paper, the modelling results into reaction-diffusion systems, which raises also the question on how to couple mass transport with mechanics. In \cite{boyer} or  \cite{abelsgarckegruen}, another solution is applied. The rescaled mass fluxes $$(\rho_i^{\rm R} )^{-1} \, \jmath^i = \varphi_i \, v^{\rm vol} + \tilde{J}^i \, ,$$ are employed. The modified diffusion fluxes $\tilde{J}^i$ have to sum up to zero due to the definition of the volume-averaged velocity. Now, it is postulated that $v^{\text{vol}}$ is the mechanical velocity occurring in Newton's law and diffusion is modelled in the spirit of the theory of irreversible processes as relative motion to this new velocity.
This approach is restricted to ideal mixtures of perfectly incompressible constituents without chemical reactions, and it is incompatible with the identity \eqref{netmassfluxequalsrhotimesv}. Hence, the continuity equation $\partial_t \varrho + \divv (\varrho \, v) = 0$ is not valid with the same velocity field as the one occurring in the equations of momentum balance. 


At last, let us remark that serious conceptual problems are still generated by the appropriate notion of incompressibility for the multicomponent case. 
In \cite{ottoweinan}, incompressibility is defined as $\sum_{i=1}^N \varphi_i = 1$. In \cite{boyer,abelsgarckegruen}, incompressibility is defined as $\divv v^{\rm vol} = 0$. This constraint is enforced by using the pressure in the momentum equation (for $v^{\rm vol}$) as a Lagrange multiplier, in the spirit of the single-component incompressible Navier--Stokes equations. As a drawback, the pressure is
decoupled from the thermodynamic potential and does not obey \eqref{GIBBSDUHEMEULER}.

%
%

As already shown in \cite{josef}, it is not to be expected that the barycentric velocity field $v^{\rm mass}$ is solenoidal even in volume additive mixtures of incompressible fluids. The only possible exceptions are: (i) that all fluids possess the same density, which is seldom the interesting case in practice or, (ii) that $N-1$ components are dilute in one dominating fluid, which of course occurs frequently. But mathematically, the case (ii)  reduces to the single-component variant of the Navier--Stokes equations, which essentially decouple.

From this latter viewpoint, a reliable way to enforce the constraint \eqref{EOSVAINCOMP} while staying consistent with \eqref{GIBBSDUHEMEULER} and the general form of the free energy functional $\varrho\psi(T, \, \rho_1, \ldots, \rho_N)$, is to introduce the thermodynamic pressure as a Lagrange multiplier directly for the algebraic constraint \eqref{EOSVAINCOMP} instead of the differential constraint $\divv v = 0$. We refer to \cite{bothedreyerdruet} for a recent model derivation -- similar observations had been done in \cite{josef} and \cite{lowe} with the notion of quasi-incompressible fluid
-- and to \cite{feima16}, \cite{druetmixtureincompweak}, \cite{bothedruetincompress} for first analytical results. 
Here too, $\divv v^{\rm vol} = 0$ results as a trivial consequence of the mass conservation equations $\partial_t \rho_i + \divv \jmath^i = 0$, the equation of state \eqref{EOSVAINCOMP} (= the incompressibility constraint), and the assumption that the constituents are incompressible -- that is, $\hat{\rho}_{i}(T,p) = \rho_i^{\rm R}$ constant. Hence, $\divv v^{\rm vol} = 0$ does not need being enforced by changing the definition of pressure or Newton's conservation principle.

Another interesting recent result proved in this context is that not only volume-additive mixtures of incompressible fluids, but even general incompressible mixtures must obey the equation \eqref{EOSVAINCOMP}. Hence, the partial molar volumina are independent of the composition if the mixture is incompressible. This has been proved and discussed in \cite{bothedreyerdruet}.

\end{document}